\documentclass[reqno]{amsart}

\usepackage[utf8]{inputenc}
\usepackage[T1]{fontenc}
\usepackage{amsmath,amsfonts,amssymb,mathtools}
\usepackage{bm}
\usepackage{enumitem}
\usepackage[margin=2.7cm]{geometry}
\usepackage[colorlinks=true,linkcolor=blue,citecolor=blue,urlcolor=blue]{hyperref}
\usepackage{tikz}
\usetikzlibrary{arrows.meta,positioning}

\newtheorem{theorem}{Theorem}[section]
\newtheorem{proposition}{Proposition}[section]
\newtheorem{corollary}{Corollary}[section]
\newtheorem{lemma}{Lemma}[section]
\newtheorem{definition}{Definition}[section]
\newtheorem{assumption}{Assumption}[section]
\theoremstyle{remark}
\newtheorem{remark}{Remark}[section]

\title[Hydrostatic energy structure of the barotropic CPE]
{On the intrinsic hydrostatic energy structure of the barotropic compressible primitive equations}

\author{Mamadou Korca Ba}
\address{Université Assane Seck de Ziguinchor, Laboratoire de Mathématiques et Applications (LMA), Sénégal \and Institut de Mathématiques de Toulon (IMATH, EA 2134), Université de Toulon, La Garde, France}
\email{mamadou-korca-ba@etud.univ-tln.fr}
\email{m.ba20150526@zig.univ.sn}

\author{Mehmet Ersoy}
\address{Institut de Mathématiques de Toulon (IMATH, EA 2134), Université de Toulon, La Garde, France}
\email{mehmet.ersoy@univ-tln.fr}

\author{Timack Ngom}
\address{Université Assane Seck de Ziguinchor, Laboratoire de Mathématiques et Applications (LMA), Sénégal}
\email{timack.ngom@univ-zig.sn}


\begin{document}

\begin{abstract}
In the compressible primitive equations, gravity is not merely a source
term: through the hydrostatic balance it forces, it organizes the
system's energy geometry. The hydrostatic balance forces the sum of
the barotropic enthalpy and the geopotential to be independent of the
vertical coordinate -- an intrinsic first hydrostatic integral,
coinciding with the classical Montgomery potential of geophysical
fluid dynamics -- and testing the continuity equation against it gives
the total mechanical energy identity, with no separate estimate on the
vertical velocity. The core of this note is an augmented
hydrostatic identity, obtained by testing the momentum equation
against a Bresch--Desjardins-type augmented velocity: differentiating
the first integral again produces two structural functions governing
its cross terms, whose vanishing cases are characterized exactly here
-- one vanishes identically if and only if the pressure is quadratic,
the other if and only if it is isothermal, no law annihilating both.
The resulting augmented identity also isolates a coupling involving
the horizontal gradient of the vertical velocity that is not
controlled by the energetic structure developed here, a limitation of
this direct route to closing a Bresch--Desjardins-type entropy, not an
impossibility theorem. For the isothermal law, the
vanishing of the second function forces the vertical density profile
to be a pure exponential, exactly the factorization underlying the
Ersoy--Ngom--Sy change of variables -- the multiplicative expression of
the Montgomery first integral at that one law -- on which basis we
give a detailed return-to-original-variables argument for the
existence theorem of Wang, Dou and Jiu.
\end{abstract}

\maketitle

\setcounter{tocdepth}{2}
\tableofcontents

\section{Introduction}

The compressible primitive equations (CPE) describe a viscous
compressible fluid in the hydrostatic regime, in which the vertical
momentum balance degenerates to the hydrostatic equilibrium
\begin{equation}
\label{eq:intro-hydrostatic}
\partial_yp(\rho)=-g\rho,
\end{equation}
where $g>0$ is the gravitational intensity: the vertical velocity is no
longer governed by an independent evolution equation, but reconstructed
from the continuity equation and the impermeability condition. This
approximation, introduced by Lions--Temam--Wang
\cite{LionsTemamWang1,LionsTemamWang2} and derived in the compressible
setting by Ersoy, Ngom and Sy \cite{ErsoyNgomSy}, underlies work on
strong solutions \cite{LiuTiti1,CaoTiti}, asymptotic stability with
physical vacuum \cite{LiuTiti2}, low Mach number limits
\cite{LiuTiti3}, and global weak solutions \cite{WangDouJiu,LiuTiti4}.

The basic energy estimate -- testing the momentum equation against the
velocity itself -- controls only the mechanical energy and gives no
information on the horizontal density gradient $\nabla_x\rho$. For a
viscosity that degenerates at vacuum, as is generically needed to admit
vacuum states, this is insufficient: without a compactness mechanism
for the density, mass and mechanical-energy bounds alone do not
provide, within the compactness framework considered here, the strong
convergence of $\rho$ needed to pass to the limit in the pressure and
convective nonlinearities.

This is precisely the role the Bresch--Desjardins entropy method was
built to play for the compressible Navier--Stokes equations
\cite{BD1,BD2}: testing the momentum equation against a velocity
augmented by a density-gradient drift produces an explicit bound on
$\nabla_x\rho$ (or $\nabla_x\sqrt\rho$), from which strong compactness
of the density follows by an Aubin--Lions argument. Mellet and Vasseur
\cite{MelletVasseur}, Vasseur and Yu \cite{VasseurYu,VasseurYu2}, and,
in its most flexible current form, the $\kappa$-entropy method of
Bresch, Vasseur and Yu \cite{BVY} extend this mechanism to absorb
degenerate-viscosity terms, each changing how the drift couples to the
viscosity profile, but none dispensing with the underlying bound on
$\nabla_x\rho$ itself.

Compressed into a single chain: a closed Bresch--Desjardins-type
identity $\to$ control of $\nabla_x\rho$ $\to$ strong compactness of the
density $\to$ passage to the limit in the nonlinearities $\to$ a global
weak solution. Reproducing this chain for the CPE with gravity is not
routine: in the parent compressible Navier--Stokes equations the
vertical velocity is tested against its own viscous momentum equation,
whereas in the hydrostatic reduction it is instead reconstructed from
the continuity equation alone, with no viscous term of its own from
which to extract a gradient bound. Along the \emph{direct
augmented-velocity route} considered here, this produces a further
vertical coupling -- distinct from the \emph{kinematic} one encoded
later by the structural functions $\Lambda$ and $\mathcal{N}$ (Section
\ref{sec:Lambda-N}) -- localizing the remaining control problem in
mechanism (b) of Theorem \ref{thm:kappa-entropy}. We do not claim that a
direct bound on $\nabla_xv$ is necessary for every conceivable route to
a general weak theory, only that the direct route developed here does
not supply it.

For the three-dimensional barotropic CPE with density-dependent,
possibly degenerate viscosity -- the class considered here -- the two
global weak-solution frameworks closest to the structural questions
addressed below are those of Wang, Dou and Jiu and of Liu and Titi;
neither isolates the energetic content of the
hydrostatic balance \eqref{eq:intro-hydrostatic} itself. Wang, Dou and
Jiu \cite{WangDouJiu} treat the isothermal law $p(\rho)=c^2\rho$ with
gravity and vacuum, using the Ersoy--Ngom--Sy change of variables
\cite{ErsoyNgomSy}, which factorizes the hydrostatic density profile so
that the transformed density no longer depends on the vertical
coordinate -- a reduction specific to the linear (isothermal) law,
whose multiplicative factorization does not extend to a general
barotropic law $p(\rho)=a\rho^\gamma$, $\gamma>1$. Liu and Titi
\cite{LiuTiti4} instead establish a global weak solution theory for the
polytropic case with degenerate viscosity, following the
Bresch--Desjardins/Mellet--Vasseur route
\cite{BD1,BD2,MelletVasseur,VasseurYu}, but by removing gravity
entirely, eliminating the hydrostatic coupling between density and
pressure. Neither work formulates the exceptional isothermal structure
through the general criterion $\mathcal N\equiv0$, or identifies the
vertical Bresch--Desjardins dissipation -- already present in the
entropy computations of Ersoy, Ngom and Sy and of Wang, Dou and Jiu,
though absent from Liu and Titi's gravity-free model -- as the
transformed counterpart of mechanism (b) arising from a general
barotropic augmented identity in the original variables. This is the
gap the present note addresses: differentiating the hydrostatic balance
to build the augmented equation this route requires produces exactly
the additional vertical cross terms just described, for which the existing
Bresch--Desjardins/Mellet--Vasseur/Vasseur--Yu framework does not
directly provide, in the present hydrostatic formulation, the control
it supplies for the corresponding non-hydrostatic terms.

This note isolates, independently of either route and for a
\emph{general} barotropic pressure law, the energetic structure the
hydrostatic balance \eqref{eq:intro-hydrostatic} imposes by itself:
gravity organizes the system's energy geometry rather than entering it
as a mere source term. The structural fact underlying every result
below is that the hydrostatic balance alone forces a single combination
of the barotropic enthalpy and the geopotential -- a \emph{first
hydrostatic integral} -- to be independent of the vertical coordinate;
this combination coincides with the classical \emph{Montgomery
potential} $\mathfrak M=\mathfrak h(\rho)+\Phi$ of geophysical fluid
dynamics (Definition \ref{def:hydrostatic-potential}). What is shown
here is not the potential itself, which is classical, but that its
vertical constancy organizes the energetic analysis of the barotropic
CPE at every order of differentiation, providing the thread connecting
the results below: at zeroth order it yields the total mechanical
energy identity (Section \ref{sec:total-energy}) with no separate
estimate on the vertical velocity; differentiated once more, it isolates
the structural functions $\Lambda,\mathcal{N}$ governing the cross terms
of a Bresch--Desjardins-type augmented-velocity identity (Section
\ref{sec:Lambda-N}), whose classification, two residual mechanisms, and
distinguished isothermal specialization are detailed in the
contributions below.

The classification just described opens a second chain, specific to the
one law it singles out for which $\mathcal{N}\equiv0$: the isothermal law. There
the vertical constancy of $\mathfrak M$ takes an exceptional form -- it
forces the vertical density profile itself to be a pure exponential
(Proposition \ref{lem:Lambda-N-geometric}), exactly the factorization
performed by the change of variables of Ersoy, Ngom and Sy
\cite{ErsoyNgomSy}, which we show is nothing but $\mathfrak M$ written
multiplicatively (Proposition \ref{prop:ens-montgomery}): on this
reading, the change of variables underlying the existence theorem of
Wang, Dou and Jiu \cite{WangDouJiu} -- itself known, as is their
existence theorem -- is not a device fitted to the isothermal law from
outside, but the same hydrostatic structure re-expressed in transformed
coordinates, illustrated for both the polytropic and the isothermal
pressure laws (Section \ref{sec:polytropic}). As a concrete instance, we
use the $\mathcal{N}\equiv0$ characterization to make explicit the
``return to original variables'' step in the existence theorem of Wang,
Dou and Jiu (\S\ref{subsec:wdj-transfer}). Beyond this transfer and the
explicit absorption above, the results are structural: no new
compactness or existence statement is claimed, and the construction is
compared with the related static-energy formulations of Tailleux in
\S\ref{subsec:tailleux-comparison}.

In summary, and in more detail:
\begin{itemize}
\item we identify the first hydrostatic integral and derive from it the
total mechanical energy identity, with no separate estimate on the
vertical velocity (Theorem \ref{thm:total-energy});
\item we classify the barotropic pressure
laws annihilating the structural functions $\Lambda,\mathcal{N}$ generated by the
augmented identity below: $\Lambda\equiv0$
if and only if the
law is quadratic, $\mathcal{N}\equiv0$ if and only if it is isothermal, no law
annihilating both (Theorem \ref{thm:Lambda-N-classification});
\item we construct a genuine two-velocity augmented identity in which
the drift's own kinetic energy and equation cancel identically a
formally expected third, Hessian-type mechanism, leaving only
mechanisms (a) and (b) in the final balance (Theorem
\ref{thm:kappa-entropy});
\item we identify mechanism (b) -- consisting of a
vertical-shear term and a density-gradient term, both coupled to
$\nabla_xv$ -- as a structural limitation of
the direct augmented-velocity route developed here, which does not
supply a closed Bresch--Desjardins entropy for the CPE with gravity,
traced explicitly to the hydrostatic degeneracy of the vertical
momentum balance (Remark \ref{prop:mechanism-b-obstruction}, Remark
\ref{rem:gravity-role});
\item we show that the
Ersoy--Ngom--Sy change of variables, already known, is the first
hydrostatic integral written multiplicatively at the isothermal law
(Proposition \ref{prop:ens-montgomery});
\item we give a rigorous, term-by-term
transfer of the weak solutions of Wang, Dou and Jiu \cite{WangDouJiu}
to the original hydrostatic variables, making explicit the argument
behind their Theorem 2.2 (Theorem \ref{thm:wdj-transfer}), and, at the
isothermal law and $\kappa=1$, identify the classical vertical
Bresch--Desjardins dissipation of Ersoy, Ngom and Sy and of Wang, Dou
and Jiu as the negative of the transformed mechanism (b) in the general
barotropic augmented identity (Proposition
\ref{prop:wdj-mechanism-reduction}, Remark
\ref{rem:wdj-cross-term-favorable}).
\end{itemize}
Figure \ref{fig:overview} summarizes the general chain just described.

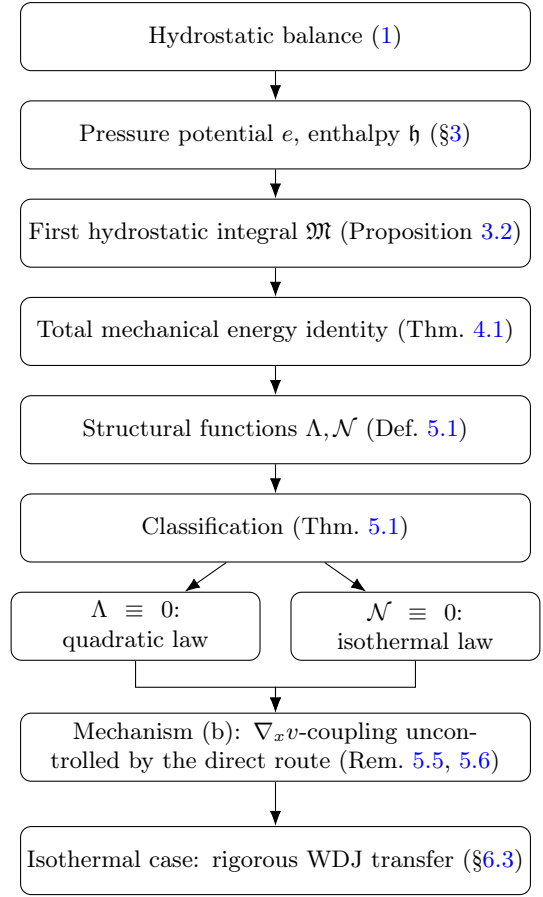
\begin{figure}[htbp]
\centering
\begin{tikzpicture}[
  box/.style={draw, rounded corners, align=center, text width=6.6cm,
    minimum height=9mm, font=\small, inner sep=2pt},
  small/.style={draw, rounded corners, align=center, text width=3.15cm,
    minimum height=9mm, font=\small, inner sep=2pt},
  arr/.style={-{Latex[length=2mm]}}
]
\node[box] (hb) at (0,0) {Hydrostatic balance \eqref{eq:intro-hydrostatic}};
\node[box] (pp) at (0,-1.3) {Pressure potential $e$, enthalpy $\mathfrak h$ (\S\ref{sec:hydrostatic-integral})};
\node[box] (hi) at (0,-2.6) {First hydrostatic integral $\mathfrak M$ (Proposition \ref{lem:first-integral})};
\node[box] (en) at (0,-3.9) {Total mechanical energy identity (Thm.\ \ref{thm:total-energy})};
\node[box] (ln) at (0,-5.2) {Structural functions $\Lambda,\mathcal{N}$ (Def.\ \ref{def:Lambda-N})};
\node[box] (cl) at (0,-6.5) {Classification (Thm.\ \ref{thm:Lambda-N-classification})};
\node[small] (quad) at (-1.85,-7.8) {$\Lambda\equiv0$:\\ quadratic law};
\node[small] (iso) at (1.85,-7.8) {$\mathcal{N}\equiv0$:\\ isothermal law};
\node[box] (bd) at (0,-9.4) {Mechanism (b): $\nabla_xv$-coupling uncontrolled by the direct route (Rem.\ \ref{prop:mechanism-b-obstruction}, \ref{rem:gravity-role})};
\node[box] (wdj) at (0,-10.9) {Isothermal case: rigorous WDJ transfer (\S\ref{subsec:wdj-transfer})};
\foreach \a/\b in {hb/pp, pp/hi, hi/en, en/ln, ln/cl}
  \draw[arr] (\a) -- (\b);
\draw[arr] (cl) -- (quad);
\draw[arr] (cl) -- (iso);
\draw[arr] (quad.south) -- ++(0,-0.35) -| (bd.north);
\draw[arr] (iso.south) -- ++(0,-0.35) -| (bd.north);
\draw[arr] (bd) -- (wdj);
\end{tikzpicture}
\caption{The logical chain of this note. The hydrostatic balance alone
produces the energetic structure of Sections
\ref{sec:hydrostatic-integral}--\ref{sec:Lambda-N}; the classification
of $\Lambda,\mathcal{N}$ singles out the quadratic and isothermal laws, but only
the isothermal branch is developed further, since mechanism (b)
remains uncontrolled by the direct augmented-velocity route at
\emph{every} barotropic law considered here
(Remark \ref{prop:mechanism-b-obstruction}).}
\label{fig:overview}
\end{figure}

Its isothermal specialization -- traced in \S\ref{subsec:wdj-transfer}
once the relevant notation is available -- follows the same chain
through the Ersoy--Ngom--Sy change of variables to the detailed
transfer of Theorem \ref{thm:wdj-transfer}.

\section{Preliminaries}
\label{sec:preliminaries}

This section fixes the setting used throughout: the two strands of
prior work the present note draws on, then the model, unknowns, and
standing hypotheses under which every subsequent result is stated.

\subsection{Related work}
\label{sec:related-work}

The mathematical analysis of the CPE has developed along several
largely independent lines, organized here by theme; the present note
draws on all of them but is subsumed by none.

\subsubsection{Derivation and strong solutions}
The hydrostatic approximation was introduced formally for the
incompressible primitive equations by Lions, Temam and Wang
\cite{LionsTemamWang1,LionsTemamWang2}, and placed on rigorous footing
by Cao and Titi \cite{CaoTiti}, who
obtained global well-posedness of strong solutions for the
three-dimensional viscous primitive equations of large-scale ocean and
atmosphere dynamics. In the compressible setting, Ersoy, Ngom and Sy
\cite{ErsoyNgomSy} gave the formal derivation from the compressible
Navier--Stokes equations and a first stability result for weak
solutions -- not existence: given a sequence of weak solutions
satisfying the requisite a priori bounds, they prove convergence, up to
a subsequence, to a weak solution of the same system, and note that
existence itself remains open. For their reduced
isothermal system, obtained after the change of variables discussed
below, their stability proof already deploys a Bresch--Desjardins-type
computation, testing the momentum equation against an augmented
velocity of the form $\bm{u}+2\bar\nu_1\nabla_x\log\xi$ to control the
density gradient. This is an early, isothermal-specific instance of the
Bresch--Desjardins entropy technique, applied directly to the
hydrostatic CPE well before the more flexible extensions of
Vasseur and Yu \cite{VasseurYu,VasseurYu2} and Bresch, Vasseur and Yu
\cite{BVY} discussed below, and a natural precursor of the
construction Section \ref{sec:Lambda-N} develops here for a
general barotropic law. The sign-indefinite integral that Theorem
\ref{thm:total-energy} below routes around is recorded explicitly, for
the original (untransformed) hydrostatic system, in the theses of
Ersoy \cite{ErsoyThese2010} and Ngom \cite{NgomThese2010}: testing the
momentum equation against $\bm u$ produces $\int_\Omega \rho g v\,dx\,dy$,
whose sign they state cannot be determined from the hydrostatic balance
alone. Both theses bypass the difficulty operationally, through the
same isothermal change of variables just recalled, rather than through
the structural, pressure-law-independent explanation the Montgomery
potential $\mathfrak M$ of Definition \ref{def:hydrostatic-potential}
supplies here. Liu and Titi \cite{LiuTiti1} established local
well-posedness of strong solutions for the three-dimensional CPE, and,
together with Xin \cite{LiuTiti2}, studied the asymptotic stability of
the hydrostatic equilibrium for the associated free-boundary problem
under \emph{physical vacuum} -- a degeneracy of the pressure at the
free boundary distinct from, but related to,
the vacuum discussed in Section \ref{sec:polytropic}. Very
recently, Hieber, Iida, Roy and Z\"ochling \cite{HieberIidaRoyZochling}
introduced a hydrostatic Lagrangian approach to the CPE based on the
compressible hydrostatic Lam\'e and Stokes operators, obtaining local
strong well-posedness for large data and global well-posedness
for small data, with and without gravity, under several
pressure laws. In parallel, Ba, Ersoy and Ngom
\cite{BaErsoyNgom2026} establish local-in-time existence and
uniqueness of strong solutions for the isothermal CPE with gravity,
away from vacuum, via a Danchin-type maximal $L^p$-regularity approach.
They work directly in the Ersoy--Ngom--Sy stratified
variables recalled in \S\ref{subsec:wdj-transfer} below -- the same
change of variables the transfer result of this note relies on.
Requiring the density to remain strictly positive and bounded, rather
than the degenerate viscosity used elsewhere to admit vacuum, this
functional-analytic route is structurally distinct from the
Bresch--Desjardins/entropy methods reviewed here.

\subsubsection{Weak solutions and degenerate viscosity}
The general theory of global weak solutions for compressible flows
originates with the constant-viscosity, barotropic compressible
Navier--Stokes equations, for which Lions \cite{Lions} constructed
global weak solutions for adiabatic exponent $\gamma\ge9/5$ in three
space dimensions, later extended by Feireisl, Novotn\'y and
Petzeltov\'a \cite{Feireisl} to $\gamma>3/2$. The CPE,
obtained from the compressible Navier--Stokes equations in the
hydrostatic limit, are not covered directly; the two
routes described below were developed specifically to accommodate,
respectively, gravity with vacuum and density-dependent
degenerate viscosity.

For viscosities that degenerate at vacuum, the basic energy estimate
alone does not control the density gradient. The two routes recalled in
the Introduction both postdate and rely on the general
theory of global weak solutions for compressible Navier--Stokes
equations with density-dependent, possibly degenerate, viscosity built
by Bresch and Desjardins \cite{BD1,BD2}, Mellet and Vasseur
\cite{MelletVasseur}, Vasseur and Yu \cite{VasseurYu,VasseurYu2}, and,
in its most flexible current form -- the $\kappa$-entropy method,
applicable to viscosities satisfying a
Bresch--Desjardins-type structural relation -- Bresch, Vasseur and Yu
\cite{BVY}. A global weak theory for the three-dimensional CPE with
gravity covering a general class of polytropic pressure laws -- rather
than the linear law used by \cite{WangDouJiu} -- does not appear to be
available within the frameworks considered here; the energetic
structure isolated in this note is intended as a foundation for closing
this gap.

The two closest precedents resolve the same degenerate-viscosity
difficulty by removing, in complementary ways, the hydrostatic density
coupling this note studies. Liu and Titi \cite{LiuTiti4} remove gravity
entirely, so that their hydrostatic balance reduces to
$\partial_zP(\rho)=0$: since $P$ is strictly increasing, this forces the
density itself to be independent of the vertical coordinate,
$\rho=\rho(t,x)$, eliminating the hydrostatic density stratification --
and with it any analogue of the structural functions $\Lambda,\mathcal{N}$ -- from
the outset. The vertical velocity still satisfies no momentum equation
of its own in their model either -- they identify this explicitly as
their main difficulty -- and they recover its compactness, needed to
pass to the limit in the nonlinear terms, through a Bresch--Desjardins
entropy and Mellet--Vasseur route of the same general degenerate-viscosity
family this note builds on; but their route acts on a density with no
vertical structure to begin
with, since $\rho=\rho(t,x)$ identically, whereas retaining gravity
reintroduces the genuine hydrostatic profile of Proposition
\ref{lem:first-integral} and, with it, the vertical density variation
their route never has to confront. Wang,
Dou and Jiu \cite{WangDouJiu} keep gravity but work at the one pressure law,
isothermal, for which the Ersoy--Ngom--Sy change of variables removes
the vertical coupling by a different route: after the transform,
$\partial_z\xi=0$ eliminates the pressure-driven density stratification
in the transformed vertical coordinate entirely, and the surviving
cross term -- already treated through the classical vertical
Bresch--Desjardins dissipation in Ersoy--Ngom--Sy's and Wang--Dou--Jiu's
own entropy computations \cite[eq.\ (40)]{ErsoyNgomSy}, \cite[eq.\
(4.3)]{WangDouJiu} -- is traced here (Lemma
\ref{lem:wdj-cross-term-sign}, Proposition
\ref{prop:wdj-mechanism-reduction}) back to mechanism (b) of Theorem
\ref{thm:kappa-entropy}. Neither precedent formulates the term as
mechanism (b) of a general-barotropic, original-variable augmented
identity, with gravity retained and no change of variables or
vertically independent density available to remove the coupling; but
their relation to it differs. Liu and Titi's route never encounters
the term at all, since $\rho=\rho(t,x)$ eliminates the vertical density
coupling from the outset. Wang, Dou and Jiu's route does encounter it
-- their own Bresch--Desjardins computation already controls its
transformed counterpart, as the identification above makes precise --
but only after passing through the isothermal change of variables, not
as mechanism (b) of a general law in the original variables. This is
why the obstruction of Remark \ref{prop:mechanism-b-obstruction}, for a
general barotropic law in the original variables, is not visible from
either existing route, and why classifying it explicitly is the
contribution of Section \ref{sec:Lambda-N}.

\subsubsection{Energy conservation and regularity criteria}
A separate line of work addresses the converse question: under what
regularity does a weak solution of the CPE conserve energy, rather than
merely satisfy an energy \emph{inequality}. Ne\v{c}asov\'a,
Rodr\'iguez-Bellido and Tang \cite{NecasovaRodriguezBellidoTang2023}
obtain sufficient regularity conditions for energy equality to hold for
weak solutions of the CPE with degenerate viscosity, even in the
presence of vacuum. This complements the structural identity of
Theorem \ref{thm:total-energy}: it asks when the induced weak
inequality \eqref{eq:weak-energy-ineq} can be upgraded to an equality,
a question not addressed here. Ne\v{c}asov\'a, Tang, Wiedemann and Zhu
\cite{NecasovaTangWiedemannZhu2025} obtain an analogous result for the
\emph{inviscid, incompressible} primitive equations -- a related but
distinct system, sharing the hydrostatic domain but neither compressible
nor viscous.

\subsubsection{Comparison with static energy formulations}
\label{subsec:tailleux-comparison}
The construction of $\mathfrak{M}$ here is conceptually related to, but
distinct from, the theory of available potential and
static energy for stratified fluids developed by Tailleux
\cite{Tailleux2013,Tailleux2018,TailleuxDubos2024,TailleuxRoullet2025}.
Both isolate a scalar potential that reorganizes the energy of a
stratified compressible fluid around its hydrostatic structure, but
Tailleux's static energy is defined relative to an adiabatically
rearranged reference state, chosen independently of the dynamics.
The vertical constancy of the hydrostatic potential $\mathfrak{M}$ of
Definition \ref{def:hydrostatic-potential} used here, by contrast, is an
exact, local, and unconditional consequence of the hydrostatic balance
\eqref{eq:intro-hydrostatic} itself, and does not require constructing
an adiabatically rearranged reference state of the kind Tailleux's
construction builds: shifting the reference density $\rho_\star$ fixing
the additive constant in $\mathfrak h$ (Definition
\ref{def:pressure-potential}) only shifts $\mathfrak{M}$ by that same
constant, and leaves its vertical constancy, the property used
throughout this note, unaffected.

\subsection{The model}
\label{sec:model}

Against this background, we now fix the equations, unknowns, and
standing hypotheses on which every result below rests.

We consider the cylindrical domain $\Omega=\mathbb T^2\times(0,H)$,
$H>0$, with $x=(x_1,x_2)\in\mathbb T^2$ the horizontal variable and
$y\in(0,H)$ the vertical variable. The unknowns are the density
$\rho=\rho(t,x,y)\ge0$, the horizontal velocity $\bm{u}\in\mathbb R^2$, and
the vertical velocity $v\in\mathbb R$. For a horizontal field $\bm{u}$, we
write $D_x\bm{u}=\tfrac12(\nabla_x\bm{u}+(\nabla_x\bm{u})^\top)$ for its symmetric
part. The compressible primitive equations read
\begin{equation}
\label{eq:CPE}
\left\{
\begin{aligned}
&\partial_t\rho+\operatorname{div}_x(\rho \bm{u})+\partial_y(\rho v)=0,\\
&\partial_t(\rho \bm{u})+\operatorname{div}_x(\rho \bm{u}\otimes \bm{u})+\partial_y(\rho v\bm{u})+\nabla_xp(\rho)
=2\operatorname{div}_x(\nu_1(\rho)D_x\bm{u})+\partial_y(\nu_2(\rho)\partial_y\bm{u}),\\
&\partial_yp(\rho)=-g\rho,
\end{aligned}
\right.
\end{equation}
with horizontal periodicity, impermeability $v=0$ and the Neumann
condition $\partial_y\bm{u}=0$ on $\{y=0,H\}$.

\begin{assumption}[Pressure law]
\label{ass:pressure}
$p:[0,\infty)\to[0,\infty)$ satisfies $p\in C([0,\infty))\cap
C^1((0,\infty))$, $p(0)=0$, and $p'(\rho)>0$ for every $\rho>0$.
\end{assumption}

\begin{assumption}[Viscosities]
\label{ass:viscosities}
$\nu_1,\nu_2:[0,\infty)\to[0,\infty)$ are continuous, $C^1$ on
$(0,\infty)$, with $\nu_1,\nu_2\ge0$.
\end{assumption}

No algebraic relation between $p,\nu_1,\nu_2$ is imposed at this stage;
Assumptions \ref{ass:pressure}--\ref{ass:viscosities} suffice for the
energetic constructions of Sections \ref{sec:hydrostatic-integral}--\ref{sec:total-energy}.
The structural analysis of Section \ref{sec:Lambda-N}, where the
functions $\Lambda$ and $\mathcal{N}$ involve $\mathfrak{h}''$ and hence $p''$,
requires the additional regularity $p\in C^2((0,\infty))$, imposed
explicitly at the beginning of that section. Only the closure attempt
at the end of Section \ref{sec:Lambda-N} requires, further, the
linear-viscosity hypothesis of Assumption \ref{ass:linear-viscosity},
introduced and flagged there explicitly.

Throughout, a \emph{regular solution} of \eqref{eq:CPE} on a time
interval $[0,T]$, $T>0$, means a
strictly positive solution $(\rho,\bm{u},v)$, periodic in $x$ and
satisfying the stated boundary conditions, which is sufficiently smooth
on $[0,T]\times\overline\Omega$ to justify all differentiations,
commutations of derivatives, traces, and integrations by parts used
below -- in particular the second-order horizontal, vertical, and mixed
derivatives of $\log\rho$ ($\nabla_x^2\log\rho$, $\partial_y^2\log\rho$,
and $\partial_y\nabla_x\log\rho$) entering
Sections~\ref{sec:Lambda-N}--\ref{sec:polytropic}, and the horizontal
differentiation of the continuity equation underlying Remark
\ref{prop:mechanism-b-obstruction}. No claim of a fixed, minimal
regularity class is made here: this is a standing convention, not a
theorem, and this note makes no claim that solutions of this regularity
exist. Whenever a specific computation requires regularity beyond this
baseline -- as for instance in Lemma \ref{lem:wdj-cross-term-sign},
where two $z$-derivatives of $\xi\bm{u}$ and $w$ are used -- this is
stated explicitly at the point of use. Where a genuine weak solution
is needed (Definition \ref{def:weak-solution} on), the regularity class
required is stated explicitly and independently of this convention.
With the model fixed, we turn to the object organizing its energy
structure.

\section{Hydrostatic first integral}
\label{sec:hydrostatic-integral}

Locating the structure announced in the Introduction requires first
constructing the object it is a statement about. This section
assembles it in two steps: a pressure potential and its enthalpy
derivative, standard constructions attached to the pressure law alone
and used throughout the mathematical theory of compressible
Navier--Stokes flows \cite{Lions,Feireisl}
(Definition \ref{def:pressure-potential}); and the
geopotential, fixed by gravity alone (Definition
\ref{def:hydrostatic-potential}). Neither step involves the
hydrostatic balance \eqref{eq:intro-hydrostatic}. What does -- and what
singles out their particular combination, the Montgomery potential
$\mathfrak M:=\mathfrak h+\Phi$ -- is Proposition \ref{lem:first-integral}
below: \eqref{eq:intro-hydrostatic} forces this sum to be independent of
the vertical coordinate. $\mathfrak M$ is therefore
not chosen for convenience; among combinations obtained by adding a
function of $\rho$ alone to a function of $y$ alone, it is, up to the
single affine reparametrization made precise in Proposition
\ref{prop:montgomery-additive-uniqueness} below, the only one the
hydrostatic balance renders vertically constant -- a qualification that
matters, since a nonlinear function $F(\mathfrak M)$ of $\mathfrak M$
alone is also vertically constant without being of this additive form,
so vertical constancy alone does not single out $\mathfrak M$ among
\emph{all} functions of $\rho$ and $y$. Its
constancy organizes every result of this note. All computations in
this section are carried out for regular solutions with $\rho>0$.

\begin{proposition}[Additive uniqueness of the Montgomery potential]
\label{prop:montgomery-additive-uniqueness}
Let $A\in C^1((0,\infty))$ and $B\in C^1((0,H))$. Assume that, for
every open subinterval $I\subset(0,H)$ and every positive $C^1$
hydrostatic profile $\rho:I\to(0,\infty)$ satisfying
$\partial_yp(\rho)=-g\rho$ on $I$, the quantity $A(\rho(y))+B(y)$ is
independent of $y\in I$ -- a hypothesis on local profiles, imposed on
every subinterval rather than only on profiles spanning the full
height $(0,H)$, so that it is never vacuous regardless of how
restricted the range of $\mathfrak h$ is under Assumption
\ref{ass:pressure} (Remark \ref{rem:montgomery-additive-local}). Then
there exist constants $c,d_1,d_2\in\mathbb R$
such that
\[
A(\rho)=c\,\mathfrak h(\rho)+d_1,
\qquad
B(y)=c\,\Phi(y)+d_2,
\]
and consequently $A(\rho)+B(y)=c\,\mathfrak M+C$ with $C:=d_1+d_2$. The
Montgomery potential is thus the unique vertically-constant combination
of a function of $\rho$ alone and a function of $y$ alone, up to the
multiplicative constant $c$ (fixed to $1$ by the normalization
$\mathfrak M=\mathfrak h+\Phi$) and an additive constant $C$ left
entirely free by the hypothesis -- the hypothesis constrains only the
sum $d_1+d_2$, not $d_1,d_2$ individually. This does not extend to
nonlinear functions of $\mathfrak M$: $F(\mathfrak M)$ for nonlinear $F$
is vertically constant too, but is generally not expressible as
$A(\rho)+B(y)$ for any choice of $A,B$, so it is not a counterexample to
the statement above -- only to the stronger, unqualified claim that
$\mathfrak M$ is the unique function of $\rho$ and $y$ with this
property.
\end{proposition}
\begin{proof}
Fix $y_0\in(0,H)$ and $\rho_\star\in(0,\infty)$; we produce a local
profile through $y_0$ with $\rho(y_0)=\rho_\star$. Under Assumption
\ref{ass:pressure}, $\mathfrak h'=p'/\rho>0$ on $(0,\infty)$, so
$\mathfrak h$ is a $C^1$ strictly increasing bijection of $(0,\infty)$
onto its range $\operatorname{ran}\mathfrak h$, an open interval of
$\mathbb R$ (the continuous, strictly monotonic image of an open
interval), with $C^1$ inverse $\mathfrak h^{-1}$ on
$\operatorname{ran}\mathfrak h$ by the inverse function theorem. Set
$m:=\mathfrak h(\rho_\star)+\Phi(y_0)$, so that $m-\Phi(y_0)=\mathfrak
h(\rho_\star)$ is an interior point of $\operatorname{ran}\mathfrak h$;
by continuity of $\Phi$, there is $\varepsilon>0$ such that
$I:=(y_0-\varepsilon,y_0+\varepsilon)\cap(0,H)$ is a neighborhood of
$y_0$ in $(0,H)$ on which $m-\Phi(y)\in\operatorname{ran}\mathfrak h$
throughout. Then
\[
\rho(y):=\mathfrak h^{-1}\bigl(m-\Phi(y)\bigr),\qquad y\in I,
\]
is a well-defined, positive, $C^1$ profile on $I$ with
$\mathfrak h(\rho(y))+\Phi(y)\equiv m$ on $I$. Differentiating this
identity in $y$ gives $\mathfrak h'(\rho(y))\rho'(y)+g=0$; since
$\mathfrak h'(\rho)=p'(\rho)/\rho$ (Definition
\ref{def:pressure-potential}), this reads
$p'(\rho(y))\rho'(y)/\rho(y)=-g$, i.e.\ $\partial_yp(\rho(y))=-g\rho(y)$
on $I$ -- the converse of Proposition \ref{lem:first-integral},
obtained here directly rather than by citing that proposition, whose
statement is the forward implication only. Also $\rho(y_0)=\rho_\star$
by construction: unlike a profile required to span the full height
$(0,H)$, this local one exists for \emph{every} $\rho_\star\in(0,\infty)$
and every $y_0\in(0,H)$, with no constraint relating the two.

Differentiating $A(\rho(y))+B(y)=\mathrm{const}$ on $I$ and using
$\rho'(y)=-g/\mathfrak h'(\rho(y))$, which follows from
$\mathfrak h(\rho(y))+\Phi(y)\equiv m$ exactly as in the proof of
Proposition \ref{lem:first-integral}, evaluation at $y=y_0$ gives
\begin{equation}
\label{eq:additive-uniqueness-ode}
g\,\frac{A'(\rho_\star)}{\mathfrak h'(\rho_\star)}=B'(y_0).
\end{equation}
Since $y_0\in(0,H)$ and $\rho_\star\in(0,\infty)$ were arbitrary and
mutually independent, the left side of \eqref{eq:additive-uniqueness-ode}
depends only on $\rho_\star$ and the right side only on $y_0$; equality
for every pair $(\rho_\star,y_0)$ forces both to equal one and the same
constant $c\in\mathbb R$:
\[
\frac{A'(\rho)}{\mathfrak h'(\rho)}=c\ \text{ for every }\rho\in(0,\infty),
\qquad
B'(y)=gc=c\,\Phi'(y)\ \text{ for every }y\in(0,H).
\]
Hence $A=c\,\mathfrak h+d_1$ and $B=c\,\Phi+d_2$ for constants
$d_1,d_2\in\mathbb R$. Conversely, for any $d_1,d_2$,
$A(\rho)+B(y)=c\,\mathfrak M+(d_1+d_2)$ is independent of $y$ on every
subinterval $I$, since $\mathfrak M$ itself is (Proposition
\ref{lem:first-integral}, applied on $I$): the hypothesis constrains
only $c$ and the sum $d_1+d_2$, not $d_1,d_2$ individually.
\end{proof}

\begin{remark}[Why the hypothesis is stated locally]
\label{rem:montgomery-additive-local}
Under the general Assumption \ref{ass:pressure} alone, $\mathfrak h$
need not be surjective onto $\mathbb R$, and its range could in
principle be an interval narrower than $[0,gH]$: were the hypothesis of
Proposition \ref{prop:montgomery-additive-uniqueness} instead imposed
only on profiles $\rho:(0,H)\to(0,\infty)$ defined on the \emph{full}
height, it would require $m-\Phi(y)\in\operatorname{ran}\mathfrak h$
simultaneously for every $y\in(0,H)$, which can fail to hold for any
$m$ once $g H$ exceeds the length of $\operatorname{ran}\mathfrak h$ --
making the hypothesis vacuous, and the proposition's conclusion
unsupported, for such a law. The local formulation used above sidesteps
this entirely: the profile constructed in the proof only has to be
defined on an arbitrarily small neighborhood of the point $y_0$ at
which it is used, which is always possible, for every target density
$\rho_\star\in(0,\infty)$, since $\operatorname{ran}\mathfrak h$ is open
and $\Phi$ is continuous. The two formulations agree whenever
$\operatorname{ran}\mathfrak h$ is wide enough relative to $gH$ (in
particular whenever $\mathfrak h$ is a bijection onto all of $\mathbb
R$, as at the isothermal law); the local formulation is the one valid
under Assumption \ref{ass:pressure} alone, with no such restriction.
\end{remark}

\begin{definition}[Pressure potential and barotropic enthalpy]
\label{def:pressure-potential}
For a reference density $\rho_\star\ge0$ such that the integral
converges,
\begin{equation}
\label{eq:pressure-potential}
e(\rho)=\rho\int_{\rho_\star}^\rho\frac{p(\varrho)}{\varrho^2}\,d\varrho,\qquad \rho>0.
\end{equation}
Following the customary notation for the specific enthalpy, but set in
Fraktur to reserve the roman $h$ for the height of the transformed
domain introduced in \S\ref{subsec:wdj-transfer},
$\mathfrak{h}(\rho):=e'(\rho)$, so that $\mathfrak{h}'(\rho)=p'(\rho)/\rho$ and
\begin{equation}
\label{eq:pressure-enthalpy-duality}
e(\rho)=\rho \mathfrak{h}(\rho)-p(\rho).
\end{equation}
\label{def:enthalpy}
\end{definition}

If $p\in C^1((0,\infty))$, then $e\in C^2((0,\infty))$, with
\begin{equation}
\label{eq:e-identities}
\rho e'(\rho)-e(\rho)=p(\rho),\qquad e''(\rho)=\frac{p'(\rho)}\rho:
\end{equation}
indeed, writing $I(\rho)=\int_{\rho_\star}^\rho p(\varrho)/\varrho^2\,d\varrho$ so that
$e=\rho I$ and $I'=p(\rho)/\rho^2$, $e'=I+p(\rho)/\rho$ gives
$\rho e'-e=p(\rho)$, and differentiating again gives $e''(\rho)=p'(\rho)/\rho$.
In particular $e$ is strictly convex on $(0,\infty)$ under Assumption
\ref{ass:pressure}.

\begin{definition}[Geopotential and hydrostatic potential]
\label{def:hydrostatic-potential}
Let
\begin{equation}
\label{eq:geopotential-def}
\Phi(y):=gy
\end{equation}
denote the \emph{geopotential} -- the gravitational potential energy
per unit mass at height $y$ -- and set
\begin{equation}
\label{eq:hydrostatic-potential-def}
\mathfrak{M}(t,x,y):=\mathfrak{h}(\rho(t,x,y))+\Phi(y).
\end{equation}
Under the hydrostatic balance \eqref{eq:intro-hydrostatic}, $\mathfrak M$
coincides with the classical \emph{Montgomery potential} of geophysical
fluid dynamics -- specific enthalpy plus geopotential -- used since
Montgomery's original work \cite{Montgomery1937} in isentropic and
isopycnal-coordinate formulations of atmosphere and ocean dynamics; we
denote it $\mathfrak M$, in Fraktur and matching $\mathfrak h$, rather
than the customary roman $M$, to avoid any collision with the total
mass $M$ of Section \ref{sec:total-energy} and \S\ref{subsec:wdj-transfer}.
\end{definition}

Definition \ref{def:pressure-potential} fixes $e$, and hence $\mathfrak
h=e'$ and $\mathfrak M=\mathfrak h+\Phi$, only once a reference density
$\rho_\star$ is chosen; changing $\rho_\star$ shifts $\mathfrak h$, and
therefore $\mathfrak M$, by an additive constant (the proof of Lemma
\ref{lem:relative-coercive} below records this shift explicitly). Like
any potential, $\mathfrak M$ is thus pinned down up to one additive
constant, not left ambiguous among an arbitrary family of unrelated
choices, and every structural statement of this note is insensitive to
that constant: the vertical constancy and horizontal-gradient
identities of Proposition \ref{lem:first-integral} below and the
structural functions of Definition \ref{def:Lambda-N} are phrased
through differences or derivatives of $\mathfrak M,\mathfrak h$ alone,
and the total mechanical energy identity of Theorem
\ref{thm:total-energy} below shifts, under a change of $\rho_\star$,
only by a time-independent multiple of the conserved total mass, which
cancels identically between its two sides. What does depend on
$\rho_\star$ is any formula recovering $\rho$ explicitly from the
numerical value of $\mathfrak M$ rather than from its differences alone:
the polytropic and isothermal profile formulas of Propositions
\ref{prop:polytropic-profile} and \ref{prop:isothermal-profile} below,
and the identity $\xi=\exp(\mathfrak M/c^2)$ of Proposition
\ref{prop:ens-montgomery} below, are each stated for the one $\rho_\star$
fixed within their respective subsections -- $\rho_\star=0$ for the
polytropic law, $\rho_\star=\mathrm e$ for the isothermal one
(Proposition \ref{prop:isothermal-potential}) -- and each would acquire
an explicit multiplicative or additive constant under a different
choice.

\begin{proposition}[Vertical constancy and horizontal gradient]
\label{lem:first-integral}
\label{lem:horizontal-gradient}
If $\rho>0$ satisfies \eqref{eq:intro-hydrostatic}, then $\partial_y\mathfrak{M}=0$:
there is a function, still denoted $\mathfrak{M}=\mathfrak{M}(t,x)$, independent of $y$
-- the first hydrostatic integral -- such that
\begin{equation}
\label{eq:first-integral}
\mathfrak{h}(\rho(t,x,y))+\Phi(y)=\mathfrak{M}(t,x).
\end{equation}
Moreover
\begin{equation}
\label{eq:horizontal-gradient}
\rho\nabla_x\mathfrak{M}=\nabla_xp(\rho).
\end{equation}
\end{proposition}
\begin{proof}
Using $\mathfrak{h}'(\rho)=p'(\rho)/\rho$ and \eqref{eq:intro-hydrostatic},
\[
\partial_y\mathfrak{M}=\mathfrak{h}'(\rho)\partial_y\rho+g
=\frac{p'(\rho)}\rho\,\partial_y\rho+g
=\frac{\partial_yp(\rho)}\rho+g
=\frac{-g\rho}\rho+g=0,
\]
giving \eqref{eq:first-integral}. Since $\Phi(y)$ is independent of $x$,
$\nabla_x\mathfrak{M}=\mathfrak{h}'(\rho)\nabla_x\rho=p'(\rho)\nabla_x\rho/\rho$;
multiplying by $\rho$ gives \eqref{eq:horizontal-gradient}.
\end{proof}

\begin{remark}[Weak formulation of the hydrostatic relation]
\label{rem:hydrostatic-weak}
The algebraic relation \eqref{eq:intro-hydrostatic}, having no time
derivative, is not tested against a space-time function $\varphi$ the
way the continuity and momentum equations are; this does not mean it
has no distributional formulation of its own. For fixed $t$, viewing
$\rho(t,\cdot,\cdot)$ and $p(\rho(t,\cdot,\cdot))$ as elements of
$\mathcal D'(\Omega)$, \eqref{eq:intro-hydrostatic} states precisely
that
\begin{equation}
\label{eq:HB-weak}
\int_\Omega p(\rho)\,\partial_y\chi\,dx\,dy=g\int_\Omega\rho\,\chi\,dx\,dy
\qquad\text{for every }\chi\in C_c^\infty(\Omega),
\tag{HB-weak}
\end{equation}
which makes sense as soon as $p(\rho)\in L^1_{loc}(\Omega)$ and
$\rho\in L^1_{loc}(\Omega)$, with no further regularity of $\rho$
assumed. Away from vacuum -- more precisely, on any vertical fiber
$\{(t,x)\}\times(0,H)$ on which $\rho$ is bounded away from $0$ and
$+\infty$ -- \eqref{eq:HB-weak} is equivalent to the pointwise statement
$\rho(t,x,y)=\mathfrak{h}^{-1}(\mathfrak{M}(t,x)-\Phi(y))$ for a.e.\
$(t,x,y)$, for some measurable $\mathfrak M(t,\cdot)$ independent of
$y$. This equivalence is not a formal restatement -- it moves from a
constraint on $p(\rho)$ to one on the differently nonlinear
$\mathfrak h(\rho)$ -- and is proved in full, by an explicit
Lipschitz-composition argument rather than by invoking Sobolev
embedding formally, in Lemma \ref{lem:HB-weak-pointwise} below. The
non-degeneracy hypothesis it requires is exactly what can fail at a
vacuum region: there,
\eqref{eq:HB-weak} continues to hold verbatim (both sides make sense, and
vanish, on $\{\rho=0\}$, since $p(0)=0$), but its equivalence with the
pointwise formula fails, and does so differently at the two laws treated
in this note. At the polytropic law, $\mathfrak h$ maps $(0,\infty)$ onto
$(0,\infty)$ only, so a vacuum point corresponds to $\mathfrak M-\Phi$
reaching the bottom of this range, and the pointwise formula extends
continuously through it (Proposition \ref{prop:polytropic-profile}). At
the isothermal law, by contrast, $\mathfrak h$ is a bijection of
$(0,\infty)$ onto all of $\mathbb R$ (Proposition
\ref{prop:isothermal-profile}), so no finite value of
$\mathfrak M(t,x)-\Phi(y)$ can represent a vacuum point through this
formula: on a vacuum region at the isothermal law, \eqref{eq:HB-weak}
itself, not any pointwise reformulation of it, is the operative
constraint. This is why Definition \ref{def:weak-solution}(i) below
takes \eqref{eq:HB-weak} directly, rather than a pointwise formula for
$\rho$, as the defining constraint on $\rho$ at the level of a weak
solution -- a formulation meaningful, and automatically satisfied, on
vacuum regions for every pressure law under Assumption
\ref{ass:pressure}, with no law-specific convention required.
\end{remark}

\begin{lemma}[Distributional--pointwise equivalence away from vacuum]
\label{lem:HB-weak-pointwise}
Let $\rho\ge0$ satisfy $\rho,p(\rho)\in L^1_{loc}(\Omega)$ and
\eqref{eq:HB-weak} for a.e.\ $t\in(0,T)$, and suppose that, for a.e.\
$(t,x)\in(0,T)\times\mathbb T^2$, there exist $0<m(t,x)\le
M(t,x)<\infty$ with
\begin{equation}
\label{eq:non-degenerate-fiber}
m(t,x)\le\rho(t,x,y)\le M(t,x)\qquad\text{for a.e.\ }y\in(0,H).
\end{equation}
Then there exists a measurable $\mathfrak M:(0,T)\times\mathbb
T^2\to\mathbb R$ such that
\begin{equation}
\label{eq:pointwise-from-lemma}
\rho(t,x,y)=\mathfrak h^{-1}\bigl(\mathfrak M(t,x)-\Phi(y)\bigr)
\qquad\text{for a.e.\ }(t,x,y)\in(0,T)\times\Omega.
\end{equation}
Conversely, if $\rho$ has the form \eqref{eq:pointwise-from-lemma} for
some measurable $\mathfrak M$ with $\mathfrak M(t,x)-\Phi(\cdot)$
ranging in the domain of $\mathfrak h^{-1}$, then $\rho$ satisfies
\eqref{eq:HB-weak}.
\end{lemma}
\begin{proof}
\emph{Reduction to a fixed vertical fiber.} Fix $t$ in the full-measure
set where \eqref{eq:HB-weak} and \eqref{eq:non-degenerate-fiber} hold.
Testing \eqref{eq:HB-weak} against $\chi(x,y)=\varphi(x)\psi(y)$ for
$\varphi\in C_c^\infty(\mathbb T^2)$, $\psi\in C_c^\infty(0,H)$, and
using Fubini (valid since $\rho,p(\rho)\in L^1_{loc}(\Omega)$),
\[
\int_{\mathbb T^2}\varphi(x)\Bigl[\int_0^Hp(\rho(x,y))\psi'(y)\,dy
-g\int_0^H\rho(x,y)\psi(y)\,dy\Bigr]dx=0
\]
for every $\varphi$; the bracket, for each fixed $\psi$, is an
$L^1_{loc}(\mathbb T^2)$ function of $x$ and therefore vanishes for
a.e.\ $x$, and running $\psi$ over a countable set dense in
$C_c^\infty(0,H)$ produces a single full-measure set of $x$ on which
\begin{equation}
\label{eq:HB-weak-fiber}
\int_0^Hp(\rho(x,y))\,\psi'(y)\,dy=g\int_0^H\rho(x,y)\,\psi(y)\,dy
\qquad\text{for every }\psi\in C_c^\infty(0,H),
\end{equation}
i.e.\ $\partial_yp(\rho(x,\cdot))=-g\rho(x,\cdot)$ in
$\mathcal D'(0,H)$. Fix such an $x$, also satisfying
\eqref{eq:non-degenerate-fiber} with constants $m,M$.

\emph{From $p(\rho)$ to a Lipschitz representative of $\rho$.} By
\eqref{eq:HB-weak-fiber} and $\rho(x,\cdot)\le M$ a.e.,
$\partial_yp(\rho(x,\cdot))\in L^\infty(0,H)$, so
$p(\rho(x,\cdot))\in W^{1,\infty}(0,H)$ and admits a Lipschitz
representative $P$ on $(0,H)$ with $P'=-g\rho(x,\cdot)$ a.e. Since
$P=p(\rho(x,\cdot))$ a.e.\ and $\rho(x,\cdot)\in[m,M]$ a.e., continuity
of $P$ forces $P(y)\in[p(m),p(M)]$ for \emph{every} $y$. Under
Assumption \ref{ass:pressure}, $p$ is a $C^1$ strictly increasing
bijection of $[m,M]$ onto $[p(m),p(M)]$ with $p'$ continuous and
bounded below by some $\delta>0$ on the compact $[m,M]$, so
$p^{-1}:[p(m),p(M)]\to[m,M]$ is Lipschitz. Hence $R:=p^{-1}\circ P$ is
a Lipschitz function on $(0,H)$, equal to $\rho(x,\cdot)$ a.e.

\emph{From $\rho$ to $\mathfrak h(\rho)$.} By \eqref{eq:e-identities},
$\mathfrak h'=p'/\rho$ is continuous and bounded on $[m,M]$ (as $m>0$),
so $\mathfrak h$ is Lipschitz on $[m,M]$, and $\widetilde H:=\mathfrak
h\circ R$ is Lipschitz on $(0,H)$, equal to $\mathfrak h(\rho(x,\cdot))$
a.e. At every point where both $R$ and $P$ are classically
differentiable -- a full-measure set by Rademacher's theorem, on which
the classical derivative agrees with the a.e.\ weak derivative already
computed -- the chain rule gives $P'=p'(R)R'$, and since
$R(y)\in[m,M]$ throughout,
\[
\widetilde H'=\mathfrak h'(R)R'=\frac{p'(R)}{R}\,R'
=\frac1R\,P'=\frac1R\,(-gR)=-g.
\]
Thus $(\widetilde H+\Phi)'=0$ a.e.\ on $(0,H)$; since
$\widetilde H+\Phi$ is Lipschitz, hence absolutely continuous, a
vanishing a.e.\ derivative forces it to be constant. Set $\mathfrak
M(t,x):=\widetilde H(y)+\Phi(y)$ for any (equivalently, every)
$y\in(0,H)$; then $\mathfrak h(\rho(x,y))+\Phi(y)=\mathfrak M(t,x)$ for
every $y$, recovering \eqref{eq:pointwise-from-lemma} for a.e.\ $y$
since $R=\mathfrak h^{-1}(\widetilde H)$ equals $\rho(x,\cdot)$ a.e.
Measurability of $\mathfrak M$ in $(t,x)$ follows by writing $\mathfrak
M(t,x)=\mathfrak h(\rho(t,x,y_0))+\Phi(y_0)$ at a fixed $y_0$ for which
$\rho(\cdot,\cdot,y_0)$ is jointly measurable, which holds for a.e.\
$y_0$ by Fubini.

\emph{Converse.} Suppose $\rho(t,x,y)=\mathfrak h^{-1}(\mathfrak
M(t,x)-\Phi(y))$ a.e.\ with $\mathfrak M(t,x)-\Phi(\cdot)$ valued in
the range of $\mathfrak h$. Set $Q:=p\circ\mathfrak h^{-1}$, of class
$C^1$ wherever $\mathfrak h^{-1}$ is, since $\mathfrak h^{-1}$ is $C^1$
by the inverse function theorem ($\mathfrak h'>0$). By $\mathfrak
h'(\rho)=p'(\rho)/\rho$ and the inverse-derivative rule,
$Q'(z)=p'(\mathfrak h^{-1}(z))(\mathfrak h^{-1})'(z)=p'(\rho)/\mathfrak
h'(\rho)=\rho$ at $\rho=\mathfrak h^{-1}(z)$, so
$p(\rho(t,x,y))=Q(\mathfrak M(t,x)-\Phi(y))$ is, for fixed $(t,x)$, a
$C^1$ function of $y$ with $\partial_yp(\rho(t,x,y))=-Q'(\mathfrak
M(t,x)-\Phi(y))\,\Phi'(y)=-g\rho(t,x,y)$ classically; integrating by
parts against $\chi\in C_c^\infty(\Omega)$ and integrating in $x$
recovers \eqref{eq:HB-weak}.
\end{proof}

\begin{proposition}[Integrated static energy balance]
\label{prop:integrated-balance}
For a regular solution satisfying the boundary conditions of Section
\ref{sec:model},
\begin{equation}
\label{eq:integrated-balance}
\frac{d}{dt}\int_\Omega\bigl(e(\rho)+\Phi(y)\rho\bigr)\,dx\,dy
=\int_\Omega \bm{u}\cdot\nabla_xp(\rho)\,dx\,dy.
\end{equation}
\end{proposition}
\begin{proof}
Multiply the continuity equation by $\mathfrak{M}=\mathfrak{h}(\rho)+\Phi(y)$ and integrate
over $\Omega$. Since $y$ is independent of $t$,
$\mathfrak{M}\partial_t\rho=\partial_t(e(\rho)+\Phi(y)\rho)$. Horizontally,
$\mathfrak{M}\operatorname{div}_x(\rho
\bm{u})=\operatorname{div}_x(\rho\bm{u}\mathfrak{M})-\bm{u}\cdot\nabla_xp(\rho)$ by Proposition
\ref{lem:horizontal-gradient}, and the flux $\operatorname{div}_x(\rho\bm{u}\mathfrak{M})$
integrates to zero by horizontal periodicity. Vertically,
$\mathfrak{M}\partial_y(\rho v)=\partial_y(\rho v\mathfrak{M})$ since $\partial_y\mathfrak{M}=0$
(Proposition \ref{lem:first-integral}), and this flux integrates to zero by
$v|_{y=0,H}=0$. Summing the three contributions and integrating over
$\Omega$ gives \eqref{eq:integrated-balance}: unlike a direct energy
computation, no term $v\,\partial_yp(\rho)$ is estimated on its own
here, since multiplying by $\mathfrak{M}$, rather than by $\mathfrak{h}(\rho)$ alone, routed
the entire vertical contribution into the exact flux
$\partial_y(\rho v\mathfrak{M})$ just annihilated by impermeability alone.
\end{proof}

\section{Intrinsic hydrostatic energy}
\label{sec:total-energy}

Having fixed the first integral $\mathfrak M$, we now use it as a
multiplier: this is what turns the vertical constancy of the previous
section into an energy identity.

What Proposition \ref{prop:integrated-balance} isolated is exactly the
quantity a direct energy computation still has to account for; we
record that computation now, and let the two combine.

\begin{lemma}[Kinetic energy identity]
\label{lem:kinetic}
For a regular solution of \eqref{eq:CPE},
\begin{equation}
\label{eq:kinetic}
\frac{d}{dt}\int_\Omega\tfrac12\rho|\bm{u}|^2\,dx\,dy
+\int_\Omega \bm{u}\cdot\nabla_xp(\rho)\,dx\,dy
+2\int_\Omega\nu_1(\rho)|D_x\bm{u}|^2\,dx\,dy
+\int_\Omega\nu_2(\rho)|\partial_y\bm{u}|^2\,dx\,dy=0.
\end{equation}
\end{lemma}
\begin{proof}
Test the horizontal momentum equation against $\bm{u}$ and integrate. The
inertial terms combine, via the continuity equation, into a pure time
derivative; the divergence terms do not contribute by periodicity and
$v|_{y=0,H}=0$. Since $D_x\bm{u}:A_x\bm{u}=0$,
$2\int\operatorname{div}_x(\nu_1(\rho)D_x\bm{u})\cdot
\bm{u}\,dx\,dy=-2\int\nu_1(\rho)|D_x\bm{u}|^2\,dx\,dy$, with no term in $\nabla_x\nu_1(\rho)$
appearing, the integration by parts acting on the full flux
$\nu_1(\rho)D_x\bm{u}$. Vertically, $\partial_y\bm{u}|_{y=0,H}=0$ gives
$\int\partial_y(\nu_2(\rho)\partial_y\bm{u})\cdot
\bm{u}\,dx\,dy=-\int\nu_2(\rho)|\partial_y\bm{u}|^2\,dx\,dy$.
\end{proof}

\begin{definition}[Total mechanical energy]
\label{def:total-energy}
The total mechanical energy is the sum of exactly three energies --
kinetic, internal, and gravitational:
\begin{equation}
\label{eq:total-energy-def}
\mathcal E(t)=
\underbrace{\int_\Omega\tfrac12\rho|\bm{u}|^2\,dx\,dy}_{\text{kinetic}}
\;+\;
\underbrace{\int_\Omega e(\rho)\,dx\,dy}_{\text{internal}}
\;+\;
\underbrace{\int_\Omega \Phi(y)\rho\,dx\,dy}_{\text{gravitational}},
\end{equation}
\begin{equation}
\label{eq:dissipation-def}
\mathcal D(t)=2\int_\Omega\nu_1(\rho)|D_x\bm{u}|^2\,dx\,dy
+\int_\Omega\nu_2(\rho)|\partial_y\bm{u}|^2\,dx\,dy.
\end{equation}
\end{definition}

The kinetic identity of Lemma \ref{lem:kinetic} and the static balance
of Proposition \ref{prop:integrated-balance} share the same pressure
term $\int_\Omega \bm{u}\cdot\nabla_xp(\rho)\,dx\,dy$, with opposite sign; adding
them cancels it exactly, leaving a single dissipative identity for the
sum of all three energies -- the first sign that $\mathfrak M$, rather
than $\mathfrak h$ alone, is the correct multiplier for this system.

\begin{theorem}[Total mechanical energy identity]
\label{thm:total-energy}
For a regular solution of \eqref{eq:CPE} satisfying the boundary
conditions of Section \ref{sec:model},
\begin{equation}
\label{eq:total-energy-identity}
\frac{d}{dt}\mathcal E(t)+\mathcal D(t)=0,\qquad\text{i.e.}\qquad
\mathcal E(t)+\int_0^t\mathcal D(s)\,ds=\mathcal E(0),\quad t\ge0.
\end{equation}
\end{theorem}
\begin{proof}
By Lemma \ref{lem:kinetic} and Proposition \ref{prop:integrated-balance},
$\int_\Omega \bm{u}\cdot\nabla_xp(\rho)\,dx\,dy=\frac{d}{dt}\int_\Omega(e(\rho)+\Phi(y)
\rho)\,dx\,dy$; substituting into the kinetic balance gives
$\frac{d}{dt}\mathcal E(t)+\mathcal D(t)=0$, then integrate in time.
\end{proof}

Both closest precedents establish the basic energy estimate by direct
computation rather than through a first hydrostatic integral. Liu and
Titi \cite{LiuTiti4} test their momentum equation against the
horizontal velocity and manipulate powers $\rho^q$ directly, tailored
to their approximation scheme; since their model omits gravity
entirely, no gravitational potential term appears. Wang, Dou and Jiu
\cite{WangDouJiu}, after the change of variables of
\cite{ErsoyNgomSy} which renders their transformed density
independent of the vertical coordinate by construction, likewise obtain
an energy identity with no explicit gravitational term: gravity is
absorbed implicitly by the change of variables itself, which exists
only because the pressure law is linear. Neither precedent constructs
an object playing the role of $\mathfrak{M}$; the route taken here, through the
first hydrostatic integral, is valid for a general barotropic law and
keeps the gravitational potential explicit throughout
\eqref{eq:total-energy-def}.

\begin{proposition}[Two equivalent representations of the energy]
\label{prop:representations}
The total energy satisfies
\begin{equation}
\label{eq:pi-representation}
\mathcal E(t)=\int_\Omega\Bigl(\tfrac12\rho|\bm{u}|^2+\rho\mathfrak{M}(t,x)-p(\rho)\Bigr)dx\,dy,
\end{equation}
and, since
\begin{equation}
\label{eq:phi-rho-identity}
\int_\Omega \Phi(y)\rho\,dx\,dy=\int_\Omega p(\rho)\,dx\,dy-H\int_{\mathbb
T^2}p(\rho(t,x,H))\,dx
\end{equation}
(obtained by multiplying \eqref{eq:intro-hydrostatic} by $y$ and
integrating in $y$),
\begin{equation}
\label{eq:enthalpic-representation}
\mathcal E(t)=\int_\Omega\Bigl(\tfrac12\rho|\bm{u}|^2+\rho
\mathfrak{h}(\rho)\Bigr)dx\,dy-H\int_{\mathbb T^2}p(\rho(t,x,H))\,dx.
\end{equation}
\end{proposition}
\begin{proof}
By \eqref{eq:pressure-enthalpy-duality}, $e(\rho)=\rho \mathfrak{h}(\rho)-p(\rho)$;
adding $\Phi(y)\rho$ and using \eqref{eq:first-integral} gives
$e(\rho)+\Phi(y)\rho=\rho\mathfrak{M}(t,x)-p(\rho)$, which is
\eqref{eq:pi-representation}. For \eqref{eq:enthalpic-representation},
combine \eqref{eq:phi-rho-identity} with
$e(\rho)+p(\rho)=\rho \mathfrak{h}(\rho)$.
\end{proof}
The condition $v|_{y=0,H}=0$ annihilates vertical fluxes in the energy
balance but imposes no condition on $p(\rho(t,x,H))$: the trace term in
\eqref{eq:enthalpic-representation} must therefore be retained in
general, unless $p(\rho(t,x,H))\equiv p_{\mathrm{top}}$ is constant, in
which case it reduces to the time-independent constant
$H|\mathbb T^2|p_{\mathrm{top}}$.

\begin{lemma}[Coercivity, and independence from $\rho_\star$]
\label{lem:relative-coercive}
Fix a background density $\bar\rho>0$ and set
\begin{equation}
\label{eq:relative-potential}
e_{\bar\rho}(\rho):=e(\rho)-e(\bar\rho)-\mathfrak{h}(\bar\rho)(\rho-\bar\rho),
\qquad\rho>0,
\end{equation}
the pressure potential of Definition \ref{def:pressure-potential},
harmonized relative to $\bar\rho$. Then $e_{\bar\rho}$ does not depend
on the reference density $\rho_\star$ used to define $e$. Under
Assumption \ref{ass:pressure}, $e_{\bar\rho}$ is strictly convex on
$(0,\infty)$, $e_{\bar\rho}(\bar\rho)=e_{\bar\rho}'(\bar\rho)=0$, and
\begin{equation}
\label{eq:relative-coercive}
e_{\bar\rho}(\rho)\ge0\ \text{ for every }\rho>0\text{, with equality
iff }\rho=\bar\rho.
\end{equation}
\end{lemma}
\begin{proof}
Replacing $\rho_\star$ by another admissible reference density
$\rho_\star'$ in Definition \ref{def:pressure-potential} gives
$\tilde e(\rho)-e(\rho)=\rho\int_{\rho_\star'}^{\rho_\star}p(\varrho)/\varrho^2\,d\varrho=:C\rho$,
a constant $C$ independent of $\rho$; so it suffices to treat
$\tilde e=e+C\rho$ for an arbitrary constant $C$. Then
$\widetilde{\mathfrak h}=\tilde e'=\mathfrak{h}+C$ and
\[
\tilde e(\rho)-\tilde e(\bar\rho)-\widetilde{\mathfrak h}(\bar\rho)(\rho-\bar\rho)
=e(\rho)-e(\bar\rho)-\mathfrak{h}(\bar\rho)(\rho-\bar\rho)+C\bigl[\rho-\bar\rho-(\rho-\bar\rho)\bigr]
=e_{\bar\rho}(\rho):
\]
the constant $C$ cancels identically, so $e_{\bar\rho}$ does not depend
on $\rho_\star$. By \eqref{eq:e-identities} and Assumption
\ref{ass:pressure}, $e_{\bar\rho}''=e''=p'(\rho)/\rho>0$, so
$e_{\bar\rho}$ is strictly convex; by construction
$e_{\bar\rho}(\bar\rho)=0$ and $e_{\bar\rho}'(\rho)=\mathfrak{h}(\rho)-\mathfrak{h}(\bar\rho)$
vanishes at $\rho=\bar\rho$. Strict convexity then gives
$e_{\bar\rho}(\rho)>e_{\bar\rho}(\bar\rho)+e_{\bar\rho}'(\bar\rho)(\rho-\bar\rho)=0$
for every $\rho\neq\bar\rho$.
\end{proof}

Potentials of the form \eqref{eq:relative-potential}, obtained by
subtracting from $e$ its tangent line at a background density, are the
standard device for producing a coercive, sign-definite internal
energy in the mathematical theory of compressible Navier--Stokes-type
equations, following Lions \cite{Lions}; the same device underlies the
relative-energy methods developed since. Replacing $e$ by
$e_{\bar\rho}$ in Definition \ref{def:total-energy} defines the
harmonized total energy
\begin{equation}
\label{eq:harmonized-energy}
\widehat{\mathcal E}(t)=
\underbrace{\int_\Omega\tfrac12\rho|\bm{u}|^2\,dx\,dy}_{\text{kinetic}}
\;+\;
\underbrace{\int_\Omega e_{\bar\rho}(\rho)\,dx\,dy}_{\text{internal}}
\;+\;
\underbrace{\int_\Omega \Phi(y)\rho\,dx\,dy}_{\text{gravitational}},
\end{equation}
the same sum of three energies as \eqref{eq:total-energy-def}, with a
coercive internal-energy term in place of $e$. Since $e_{\bar\rho}-e$
is an affine function of $\rho$ (\eqref{eq:relative-potential}) and the
total mass $M:=\int_\Omega\rho(t)\,dx\,dy$ is conserved by the
continuity equation and the boundary conditions of Section
\ref{sec:model},
\[
\widehat{\mathcal E}(t)-\mathcal E(t)
=-\mathfrak{h}(\bar\rho)M-\bigl(e(\bar\rho)-\mathfrak{h}(\bar\rho)\bar\rho\bigr)|\Omega|
\]
is independent of $t$; consequently $\widehat{\mathcal E}$ satisfies
the identical identity \eqref{eq:total-energy-identity} of Theorem
\ref{thm:total-energy}, and every estimate of this note stated as a
bound on $\mathcal E$ or on $d\mathcal E/dt$ transfers unchanged to
$\widehat{\mathcal E}$. Unlike $e$, however, $e_{\bar\rho}$ is
nonnegative and vanishes only at the background state $\rho=\bar\rho$
(Lemma \ref{lem:relative-coercive}): it is $\widehat{\mathcal E}$,
rather than $\mathcal E$, that carries the coercive control of
$\rho\log\rho$-type quantities used in the weak-solution existence
literature. Taking $\bar\rho=1$ recovers, in the isothermal case of
\S\ref{subsec:isothermal}, the normalization
\eqref{eq:isothermal-potential} below; in the polytropic case of
\S\ref{subsec:polytropic}, it recovers the classical form
$e_1(\rho)=\frac a{\gamma-1}(\rho^\gamma-\gamma\rho+\gamma-1)$. There,
however, $\rho_\star=0$ already gives
$e(\rho)=a\rho^\gamma/(\gamma-1)\ge0$, vanishing at the vacuum
$\rho=0$: harmonization is therefore only genuinely needed in the
isothermal case, where no density plays the distinguished role that
$\rho=0$ plays for $\gamma>1$.

We now record the class of weak solutions on which the energetic
structure above is intended to act. The distributional formulation
(iii)--(iv) below follows the standard convention for the barotropic
compressible Navier--Stokes equations, due to Lions \cite{Lions} and
extended by Feireisl, Novotn\'y and Petzeltov\'a \cite{Feireisl}; for
the CPE specifically, it is modeled on the two existing precedents,
the stability class of Ersoy, Ngom and Sy \cite{ErsoyNgomSy} and the
existence class of Wang, Dou and Jiu \cite{WangDouJiu}. The
regularity class (ii) pairs $\nu_i(\rho)\in L^1((0,T)\times\Omega)$
with $\sqrt{\nu_1(\rho)}\,D_x\bm{u}$, $\sqrt{\nu_2(\rho)}\,\partial_y\bm{u}\in
L^2((0,T)\times\Omega)$, rather than a bound on $\nabla_x\bm{u}$ itself: this
is the admissibility class standard in the degenerate-viscosity theory
of Bresch and Desjardins \cite{BD1,BD2}, Mellet and Vasseur
\cite{MelletVasseur}, and Vasseur and Yu \cite{VasseurYu,VasseurYu2},
also adopted by Liu and Titi \cite{LiuTiti4} for the polytropic CPE
without gravity. It is dictated by $\nu_1,\nu_2$ possibly vanishing at
vacuum (Assumption \ref{ass:viscosities}), under which a bound on
$\nabla_x\bm{u}$ itself is not available.

\begin{definition}[Weak solution of the CPE]
\label{def:weak-solution}
Let $T>0$ and let $\rho_0\ge0$, $m_0$ be initial data with
$\rho_0\in L^1(\Omega)$, $e(\rho_0)\in L^1(\Omega)$, and
$|m_0|^2/\rho_0\in L^1(\Omega)$ (with the convention $:=0$ on
$\{\rho_0=0\}$, where also $m_0=0$). We say that $(\rho,\bm{u},v)$ is a weak
solution of \eqref{eq:CPE} on $[0,T]$ if:
\begin{enumerate}[label=(\roman*)]
\item $\rho\ge0$ a.e.\ on $(0,T)\times\Omega$, with $\rho\in
L^1_{loc}(\Omega)$ and $p(\rho)\in L^1_{loc}(\Omega)$ for a.e.\ $t$, and
the hydrostatic relation holds in the sense of distributions: for a.e.\
$t\in(0,T)$,
\begin{equation}
\label{eq:weak-hydrostatic-def}
\int_\Omega p(\rho(t))\,\partial_y\chi\,dx\,dy=g\int_\Omega\rho(t)\,\chi\,dx\,dy
\qquad\text{for every }\chi\in C_c^\infty(\Omega),
\end{equation}
i.e.\ \eqref{eq:HB-weak} of Remark \ref{rem:hydrostatic-weak}. This is
how the third equation of \eqref{eq:CPE} -- the algebraic hydrostatic
relation \eqref{eq:intro-hydrostatic} -- is encoded at the level of a
weak solution: not as a further equation tested against $\varphi$ or
$\psi$, but as a structural constraint on $\rho$ built directly into
this definition. \eqref{eq:weak-hydrostatic-def} is taken as the
defining constraint, rather than a pointwise formula for $\rho$, for
exactly the reason isolated in Remark \ref{rem:hydrostatic-weak}: it
makes sense under no regularity of $\rho$ beyond $\rho,p(\rho)\in
L^1_{loc}(\Omega)$, and both of its sides vanish identically wherever
$\rho\equiv0$, since $p(0)=0$ -- so it is compatible with vacuum,
including an entire vacuum column ($\rho(t,x,\cdot)\equiv0$ on $(0,H)$
for a.e.\ $x$), for \emph{every} pressure law under Assumption
\ref{ass:pressure}, with no law-specific convention required. On the
non-degenerate set where $\rho(t,\cdot,\cdot)$ is, for a.e.\ $x$, bounded
away from $0$ and $+\infty$ on $(0,H)$, Lemma \ref{lem:HB-weak-pointwise}
shows that \eqref{eq:weak-hydrostatic-def} is equivalent to the
pointwise formula $\rho(t,x,y)=\mathfrak h^{-1}(\mathfrak
M(t,x)-\Phi(y))$ for some $\mathfrak M(t,\cdot)$ independent of $y$; this
is the regime in which Propositions \ref{prop:polytropic-profile} and
\ref{prop:isothermal-profile} apply, respectively extending the formula
continuously through vacuum at the polytropic law via
$\rho=[\cdots]_+^{1/(\gamma-1)}$, and excluding vacuum a posteriori for
regular, strictly positive solutions at the isothermal law. At the
isothermal law specifically, $\mathfrak h$ is a bijection of
$(0,\infty)$ onto all of $\mathbb R$, so no finite $\mathfrak
M(t,x)-\Phi(y)$ can represent a vacuum point through this pointwise
formula: a vacuum column is admissible at the isothermal law only
through \eqref{eq:weak-hydrostatic-def} directly, on the degenerate set
where the pointwise formula does not apply;
\item the solution has finite regularity class:
\[
\rho\in L^\infty(0,T;L^1(\Omega)),\quad
e(\rho)\in L^\infty(0,T;L^1(\Omega)),\quad
p(\rho)\in L^1((0,T)\times\Omega),
\]
\[
\sqrt\rho\,\bm{u}\in L^\infty(0,T;L^2(\Omega)),\quad
\sqrt\rho\,v\in L^2((0,T)\times\Omega),\quad
\nu_1(\rho),\nu_2(\rho)\in L^1((0,T)\times\Omega),
\]
\[
\sqrt{\nu_1(\rho)}\,D_x\bm{u},\ \sqrt{\nu_2(\rho)}\,\partial_y\bm{u}\in L^2((0,T)\times\Omega);
\]
\item the continuity equation holds in the sense of distributions: for
every $\varphi\in C_c^\infty([0,T)\times\overline\Omega)$ -- test
functions compactly supported in $t$ but, through the closure
$\overline\Omega=\mathbb T^2\times[0,H]$, allowed to be nonzero up to and
including $y=0,H$, precisely so that no boundary term is discarded by
the choice of test space itself -- periodic in
$x$,
\begin{equation}
\label{eq:weak-continuity}
\int_0^T\!\!\int_\Omega\rho\bigl(\partial_t\varphi+\bm{u}\cdot\nabla_x\varphi+v\partial_y\varphi\bigr)\,dx\,dy\,dt
+\int_\Omega\rho_0\,\varphi(0,\cdot)\,dx\,dy=0;
\end{equation}
\item the momentum equation holds in the sense of distributions: for
every $\psi\in C_c^\infty([0,T)\times\overline\Omega;\mathbb R^2)$,
periodic in $x$,
\begin{multline}
\label{eq:weak-momentum}
\int_0^T\!\!\int_\Omega\Bigl[\rho \bm{u}\cdot\partial_t\psi+\rho \bm{u}\otimes
\bm{u}:\nabla_x\psi+\rho v\bm{u}\cdot\partial_y\psi+p(\rho)\operatorname{div}_x\psi\Bigr]\,dx\,dy\,dt\\
-2\int_0^T\!\!\int_\Omega\nu_1(\rho)D_x\bm{u}:D_x\psi\,dx\,dy\,dt
-\int_0^T\!\!\int_\Omega\nu_2(\rho)\partial_y\bm{u}\cdot\partial_y\psi\,dx\,dy\,dt
+\int_\Omega m_0\cdot\psi(0,\cdot)\,dx\,dy=0.
\end{multline}
\end{enumerate}
\end{definition}

No term at $y=0,H$ appears in, or is required for, either identity, and
no trace regularity of $\bm{u}$ or $v$ at $y=0,H$ is presupposed: this is
precisely how the two boundary conditions of Section \ref{sec:model}
are encoded weakly. In \eqref{eq:weak-continuity}, the absence of a
boundary term is the weak form of impermeability $v=0$ on
$\{y=0,H\}$ -- of Dirichlet type for this first-order equation, encoded
as a vanishing flux across a characteristic boundary -- obtained by
testing against $\varphi$ arbitrary (not required to vanish at the
boundary).
In \eqref{eq:weak-momentum}, the absence of a boundary term is
likewise the weak form of the Neumann condition $\partial_y\bm{u}=0$, exactly
as for any natural condition in a variational formulation.

\begin{definition}[Finite-energy weak solution]
\label{def:finite-energy}
A weak solution $(\rho,\bm{u},v)$ in the sense of Definition
\ref{def:weak-solution} is of finite energy if $\mathcal E(0)<\infty$
and, for almost every $t\ge0$,
\begin{equation}
\label{eq:weak-energy-ineq}
\mathcal E(t)+\int_0^t\mathcal D(s)\,ds\le\mathcal E(0).
\end{equation}
\end{definition}

Since the total mass $M=\int_\Omega\rho\,dx\,dy$ is conserved (immediate
by integrating the continuity equation, using periodicity and
$v|_{y=0,H}=0$), the identity preceding \eqref{eq:harmonized-energy}
shows that $\widehat{\mathcal E}(t)-\mathcal E(t)$ is a time-independent
affine function of $M$; consequently the finite-energy inequality
\eqref{eq:weak-energy-ineq}, stated for $\mathcal E$, is equivalent to
the same inequality stated for $\widehat{\mathcal E}$. Because
$e_{\bar\rho}\ge0$, and vanishes only at $\bar\rho$ (Lemma
\ref{lem:relative-coercive}), unlike $e$ itself, it is this harmonized
form of the inequality that yields the coercive bound
$\sup_t\int_\Omega(\tfrac12\rho|\bm{u}|^2+e_{\bar\rho}(\rho))\,dx\,dy$ and
the integrated viscous dissipation, both controlled by
$\widehat{\mathcal E}(0)$ alone; the gravitational term, in either form,
is in addition controlled uniformly by
$0\le\int_\Omega\Phi(y)\rho\,dx\,dy\le gHM$.

Regarding the scope of Definition \ref{def:weak-solution}: point (i) is
the only place where the hydrostatic relation enters; it
is stated directly in distributional form,
\eqref{eq:weak-hydrostatic-def}, precisely so as not to presuppose
differentiability or strict positivity of the vertical density profile.
Away from vacuum, Lemma \ref{lem:HB-weak-pointwise} recovers the
Montgomery first-integral representation of Proposition
\ref{lem:first-integral} under the non-degeneracy assumptions stated
there. The regularity class (ii)
is an admissible class under which every term of
\eqref{eq:weak-continuity}--\eqref{eq:weak-momentum} is well defined;
no minimality is claimed or needed here, and finer
spaces (for instance $\rho\in L^\infty(0,T;L^\gamma(\Omega))$ for the
polytropic law) are not fixed at this level of generality, consistently
with Assumptions \ref{ass:pressure}--\ref{ass:viscosities}. This note
does not develop an existence theory; Definition
\ref{def:weak-solution} and Definition \ref{def:finite-energy} are
recorded only because Remark \ref{prop:mechanism-b-obstruction}
and the isothermal application of Section \ref{sec:polytropic} are
stated at this level of regularity.

\section{Structural classification}
\label{sec:Lambda-N}

We now differentiate the same constancy a second time. Testing the
momentum equation against a Bresch--Desjardins-type augmented velocity
produces cross terms governed by two structural functions of the
pressure law alone; classifying them is the goal of this section.

The total energy identity of Theorem \ref{thm:total-energy} controls
$\rho$ only through $e(\rho)$, and yields no information on
$\nabla_x\rho$; for a viscosity that degenerates at vacuum, this is
insufficient for compactness, and one is led, as in the classical
Bresch--Desjardins theory \cite{BD1,BD2} and its degenerate-viscosity
extensions \cite{MelletVasseur,VasseurYu,VasseurYu2,BVY}, to test the
momentum equation against an augmented velocity. Following Bresch,
Vasseur and Yu \cite[eq.\ (1.6)]{BVY}, we couple this drift to the
viscosity itself rather than to $\log\rho$ directly: for a reference
density $\rho_\star>0$, set
\begin{equation}
\label{eq:s-def}
s(\rho):=\int_{\rho_\star}^\rho\frac{\nu_1'(\varrho)}\varrho\,d\varrho,
\end{equation}
and test the momentum equation against
\begin{equation}
\label{eq:augmented-velocity-def}
U:=\bm{u}+2\kappa\nabla_xs(\rho)
\end{equation}
for a dimensionless constant $\kappa\ge0$: $\kappa$ and $\nu_1$ thus enter
through two different objects, $s(\rho)$ carrying all the dependence on
the viscosity profile. As Assumption \ref{ass:linear-viscosity} below
records, under the linear-viscosity hypothesis $\nu_1(\rho)=\bar\nu_1\rho$
used there, $s(\rho)$ reduces to a multiple of $\log\rho$, which is the
only case this note develops in full;
Remark \ref{rem:wdj-mechanism} further compares the resulting $\kappa$
with the relaxation parameter of the same name used by
\cite{BVY}. Constructing the corresponding
augmented equation requires differentiating the hydrostatic relation
\eqref{eq:first-integral} a second time, both horizontally and
vertically -- a step forced by the augmented-velocity construction
itself, not a choice made for convenience. This differentiation cannot
avoid producing two algebraic functions of $\rho$, encoding respectively
its vertical and its mixed vertical-horizontal second-order behavior; we
isolate them here, for a general barotropic law, independently of any
further estimate, because they determine exactly which pressure-dependent
hydrostatic cross terms simplify or vanish in the resulting augmented
identity.
Since $\Lambda,\mathcal{N}$ below involve $\mathfrak{h}''$, and hence $p''$, through
$\mathfrak{h}'(\rho)=p'(\rho)/\rho$, this section implicitly strengthens the
standing regularity of Assumption \ref{ass:pressure} from $p\in
C^1((0,\infty))$ to $p\in C^2((0,\infty))$; this is used only in the
remainder of \S\ref{sec:Lambda-N} and in \S\ref{sec:polytropic}; the definition
and use of $\Lambda,\mathcal{N}$ throughout are understood under this
strengthened hypothesis.

\begin{definition}[Structural functions]
\label{def:Lambda-N}
Assume in addition $p\in C^2((0,\infty))$. Set
\begin{equation}
\label{eq:Lambda-N-def}
\Lambda(\rho):=\frac{g\,\mathfrak{h}''(\rho)}{\mathfrak{h}'(\rho)^2},
\qquad
\mathcal{N}(\rho):=\frac{\Lambda(\rho)}\rho+\frac g{\mathfrak{h}'(\rho)\rho^2}.
\end{equation}
\end{definition}

Both functions are defined here purely algebraically, from $\mathfrak{h}',\mathfrak{h}''$
alone; Proposition \ref{lem:Lambda-N-geometric} below shows that each is in
fact already present, unnamed, in the vertical density profile itself,
and that each encodes a distinct hydrostatic obstruction:
$\Lambda=\partial_y\log|\partial_y\rho|$, the logarithmic vertical
variation of $\partial_y\rho$, is the first -- it measures how far the
vertical profile is from affine, the degeneracy needed for the
horizontal--vertical commutator of Lemma \ref{lem:N} to vanish at the
level of $\nabla_x\rho$ itself; $\mathcal{N}=\partial_y^2\log\rho/\partial_y\rho$,
the vertical curvature of $\log\rho$ rescaled by $\partial_y\rho$, is the
second, and finer -- it measures the same degeneracy one logarithmic
derivative down, vanishing exactly when the profile is a pure
exponential in $y$. This geometric reading, rather than the algebraic definition above, is
what drives the classification of Theorem
\ref{thm:Lambda-N-classification}: it identifies exactly which
barotropic pressure laws eliminate the two hydrostatic coefficients
generated by the augmented identity of Theorem \ref{thm:kappa-entropy}
below, and, in doing so, turns the isothermal law from one example
among others into
the only pressure law admitting the multiplicative Ersoy--Ngom--Sy
factorization considered here (Remark
\ref{rem:classification-geometric}), and hence the only law reached by
the application of \S\ref{subsec:wdj-transfer};
the resulting simplification is partial, not a closed entropy, as
Remark \ref{prop:mechanism-b-obstruction} below makes precise.

\begin{lemma}[Vertical--horizontal commutation identities]
\label{lem:N}
For a regular solution with $\rho>0$ satisfying
\eqref{eq:intro-hydrostatic},
\begin{equation}
\label{eq:Lambda-identity}
\partial_y\nabla_x\rho=\Lambda(\rho)\,\nabla_x\rho,
\qquad
\partial_y\bigl(\nabla_x\log\rho\bigr)=\mathcal{N}(\rho)\,\nabla_x\rho.
\end{equation}
\end{lemma}
\begin{proof}
By Proposition \ref{lem:horizontal-gradient}, $\nabla_x\mathfrak{M}=\mathfrak{h}'(\rho)\nabla_x\rho$,
and $\nabla_x\mathfrak{M}$ is independent of $y$ since $\mathfrak{M}$ is (Proposition
\ref{lem:first-integral}); hence, using $\partial_y\rho=-g/\mathfrak{h}'(\rho)$
from the proof of Proposition \ref{lem:first-integral},
\[
0=\partial_y\bigl(\mathfrak{h}'(\rho)\nabla_x\rho\bigr)
=\mathfrak{h}''(\rho)\,\partial_y\rho\,\nabla_x\rho+\mathfrak{h}'(\rho)\,\partial_y\nabla_x\rho,
\]
and dividing by $\mathfrak{h}'(\rho)$ gives the first identity. For the second,
write
\[
\partial_y\Bigl(\frac{\nabla_x\rho}\rho\Bigr)
=\frac{\partial_y\nabla_x\rho}\rho-\frac{(\nabla_x\rho)(\partial_y\rho)}{\rho^2};
\]
substituting the first identity just proved and $\partial_y\rho=-g/\mathfrak{h}'(\rho)$
gives
\[
\partial_y\Bigl(\frac{\nabla_x\rho}\rho\Bigr)
=\Bigl[\frac{\Lambda(\rho)}\rho+\frac g{\mathfrak{h}'(\rho)\rho^2}\Bigr]\nabla_x\rho
=\mathcal{N}(\rho)\nabla_x\rho.
\]
\end{proof}

Both identities are purely algebraic consequences of the hydrostatic
relation \eqref{eq:first-integral} and hold independently of any
regularity of $\bm{u},v$ or of the momentum equation itself. Once $\nu_1$ is
linear, so that $s(\rho)$ is a multiple of $\log\rho$ and the augmented
velocity \eqref{eq:augmented-velocity-def} takes the form
$U=\bm{u}+2\kappa\bar\nu_1\nabla_x\log\rho$ (Assumption \ref{ass:linear-viscosity}
below), $\Lambda$ and $\mathcal{N}$ govern, respectively, the two cross terms
coupling the vertical velocity to the horizontal
density gradient that arise from differentiating the continuity
equation: $\Lambda$ enters through the second-order vertical variation
of $\nabla_x\rho$ used in constructing the augmented momentum equation
itself, and $\mathcal{N}$ enters directly through the commutator between
$\partial_y$ and $\nabla_x\log\rho$ that appears once that equation is
tested. Whether, and under what condition on $\nu_1,\nu_2$, these two
mechanisms can be absorbed into a closed Bresch--Desjardins-type entropy
estimate is a question specific to the choice of pressure law and
viscosity profile; we now carry this out far enough to identify
precisely where it succeeds and where it does not.

This is only one of two distinct consequences of the same hydrostatic
reduction \eqref{eq:intro-hydrostatic}. $\Lambda$ and $\mathcal{N}$ obstruct the
augmented-velocity construction \emph{kinematically}, through the
vertical density profile alone, independently of the momentum equation
or of $v$. A second, \emph{dynamical} obstruction, examined starting
from the reconstruction formula \eqref{eq:v-reconstruction} below,
arises because, in the hydrostatic regime, the vertical momentum balance
is itself replaced by the algebraic relation \eqref{eq:intro-hydrostatic}:
the vertical velocity is left with no viscous equation of its
own to test against the $\nabla_xv$ term of mechanism (b) (Remark
\ref{prop:mechanism-b-obstruction}, Remark \ref{rem:gravity-role}). Both
obstructions trace to \eqref{eq:intro-hydrostatic}, but neither reduces
to the other within the framework considered here: $\Lambda,\mathcal{N}$
encode the pressure-dependent kinematic cross terms generated by the
hydrostatic density profile, while mechanism (b) reflects the absence
of a vertical momentum equation for $v$ in this hydrostatic
formulation, and persists even at the isothermal law, where
$\mathcal{N}\equiv0$ already removes mechanism (a) entirely.

Before doing so, we record two further facts about $\Lambda$ and $\mathcal{N}$
that hold independently of the augmented-velocity construction and
justify treating them as the central objects of this section: an
intrinsic analytic meaning, in terms of the vertical density profile
alone, and a complete classification of the barotropic laws that
annihilate each of them.

\begin{proposition}[Analytic meaning of $\Lambda$ and $\mathcal{N}$]
\label{lem:Lambda-N-geometric}
For a regular solution with $\rho>0$ satisfying
\eqref{eq:intro-hydrostatic}, $\partial_y\rho<0$ everywhere, and
\begin{equation}
\label{eq:Lambda-N-geometric}
\Lambda(\rho)=\partial_y\log|\partial_y\rho|,
\qquad
\mathcal{N}(\rho)=\frac{\partial_y^2\log\rho}{\partial_y\rho}.
\end{equation}
Consequently $\Lambda\equiv0$ on the range of a solution if and only if
$y\mapsto\rho(t,x,y)$ is affine there, and $\mathcal{N}\equiv0$ if and only if
$y\mapsto\log\rho(t,x,y)$ is affine there, i.e.\ the vertical density
profile is a pure exponential in $y$.
\end{proposition}
\begin{proof}
By the proof of Proposition \ref{lem:first-integral}, $\partial_y\rho=-g/\mathfrak{h}'(\rho)$,
strictly negative since $\mathfrak{h}'(\rho)=p'(\rho)/\rho>0$ by Assumption
\ref{ass:pressure}; hence $\log|\partial_y\rho|=\log g-\log \mathfrak{h}'(\rho)$,
and
\[
\partial_y\log|\partial_y\rho|
=-\frac{\mathfrak{h}''(\rho)}{\mathfrak{h}'(\rho)}\,\partial_y\rho
=-\frac{\mathfrak{h}''(\rho)}{\mathfrak{h}'(\rho)}\cdot\Bigl(-\frac g{\mathfrak{h}'(\rho)}\Bigr)
=\frac{g\mathfrak{h}''(\rho)}{\mathfrak{h}'(\rho)^2}=\Lambda(\rho).
\]
For $\mathcal{N}$, set
$\phi(\rho):=\partial_y\log\rho=\partial_y\rho/\rho=-g/(\mathfrak{h}'(\rho)\rho)$;
by the chain rule $\partial_y^2\log\rho=\phi'(\rho)\,\partial_y\rho$, and
\[
\phi'(\rho)=-g\,\frac{d}{d\rho}\Bigl[\frac1{\mathfrak{h}'(\rho)\rho}\Bigr]
=g\,\frac{\mathfrak{h}''(\rho)\rho+\mathfrak{h}'(\rho)}{(\mathfrak{h}'(\rho)\rho)^2}
=\frac{g\mathfrak{h}''(\rho)}{\mathfrak{h}'(\rho)^2\rho}+\frac g{\mathfrak{h}'(\rho)\rho^2}
=\frac{\Lambda(\rho)}\rho+\frac g{\mathfrak{h}'(\rho)\rho^2}=\mathcal{N}(\rho),
\]
so $\partial_y^2\log\rho=\mathcal{N}(\rho)\,\partial_y\rho$, and dividing by
$\partial_y\rho\neq0$ gives the stated identity. The two equivalences
follow at once: $\Lambda\equiv0$ means $\partial_y\rho$ has constant
sign and constant logarithm of its magnitude, i.e.\ is itself constant,
so $\rho$ is affine in $y$; $\mathcal{N}\equiv0$ means $\partial_y\log\rho$ is
constant, so $\log\rho$ is affine in $y$.
\end{proof}

\begin{theorem}[Classification of the structural functions]
\label{thm:Lambda-N-classification}
Under Assumption \ref{ass:pressure}, with the additional regularity
$p\in C^2((0,\infty))$ required in Definition \ref{def:Lambda-N} to
define $\Lambda,\mathcal{N}$:
\begin{enumerate}[label=(\roman*)]
\item $\Lambda\equiv0$ if and only if $p(\rho)=a\rho^2$ for some
constant $a>0$;
\item $\mathcal{N}\equiv0$ if and only if $p(\rho)=c^2\rho$ for some constant
$c^2=R\mathcal{T}>0$ (the isothermal law; $R>0$ the specific gas
constant and $\mathcal{T}>0$ the fixed absolute temperature, see
\S\ref{subsec:isothermal}).
\end{enumerate}
In particular, under these assumptions, no barotropic
pressure law annihilates both $\Lambda$ and $\mathcal{N}$ simultaneously.
\end{theorem}
\begin{proof}
(i) Since $\mathfrak{h}'(\rho)=p'(\rho)/\rho>0$ for every $\rho>0$ by Assumption
\ref{ass:pressure}, $\Lambda(\rho)=g\mathfrak{h}''(\rho)/\mathfrak{h}'(\rho)^2$ vanishes at
$\rho$ if and only if $\mathfrak{h}''(\rho)=0$. If $\Lambda\equiv0$ on
$(0,\infty)$, then $\mathfrak{h}''\equiv0$, so $\mathfrak{h}(\rho)=A\rho+B$ for constants
$A,B$, with $\mathfrak{h}'\equiv A$; since $\mathfrak{h}'(\rho)=p'(\rho)/\rho>0$, $A>0$, and
$p'(\rho)=A\rho$ gives $p(\rho)=\tfrac A2\rho^2+C$ for a constant $C$;
$p(0)=0$ (Assumption \ref{ass:pressure}) forces $C=0$, so
$p(\rho)=a\rho^2$ with $a:=A/2>0$. Conversely, if $p(\rho)=a\rho^2$,
$a>0$, then $\mathfrak{h}'(\rho)=p'(\rho)/\rho=2a$ is constant, so $\mathfrak{h}''\equiv0$ and
$\Lambda\equiv0$.

(ii) A direct computation gives
\[
\mathfrak{h}'(\rho)\rho^2\mathcal{N}(\rho)=\mathfrak{h}'(\rho)\rho\Lambda(\rho)+g
=\frac{g\rho \mathfrak{h}''(\rho)}{\mathfrak{h}'(\rho)}+g
=g\Bigl[\frac{\rho \mathfrak{h}''(\rho)}{\mathfrak{h}'(\rho)}+1\Bigr],
\]
so, since $\mathfrak{h}'(\rho)\rho^2>0$, $\mathcal{N}(\rho)=0$ if and only if
$\rho \mathfrak{h}''(\rho)/\mathfrak{h}'(\rho)=-1$, i.e.\
$\tfrac{d}{d\rho}\log \mathfrak{h}'(\rho)=\mathfrak{h}''(\rho)/\mathfrak{h}'(\rho)=-1/\rho$. If
$\mathcal{N}\equiv0$ on $(0,\infty)$, integrating this identity gives
$\log \mathfrak{h}'(\rho)=-\log\rho+C_0$ for a constant $C_0$, i.e.\
$\mathfrak{h}'(\rho)=K/\rho$ with $K:=e^{C_0}>0$; then $p'(\rho)=\mathfrak{h}'(\rho)\rho=K$,
so $p(\rho)=K\rho+D$, and $p(0)=0$ forces $D=0$, giving $p(\rho)=c^2\rho$
with $c^2:=K>0$. Conversely, if $p(\rho)=c^2\rho$, then
$\mathfrak{h}'(\rho)=c^2/\rho$ and $\mathfrak{h}''(\rho)=-c^2/\rho^2$, so
$\rho \mathfrak{h}''(\rho)/\mathfrak{h}'(\rho)=\rho\cdot(-c^2/\rho^2)/(c^2/\rho)=-1$, giving
$\mathcal{N}\equiv0$.

For the final statement: by (i)--(ii), $\Lambda\equiv0$ forces
$p(\rho)=a\rho^2$ and $\mathcal{N}\equiv0$ forces $p(\rho)=c^2\rho$; since
$\rho\mapsto a\rho^2$ and $\rho\mapsto c^2\rho$, for any $a,c^2>0$, are
distinct functions on $(0,\infty)$ -- they agree only at the single
point $\rho=c^2/a$, not identically -- no single pressure law satisfying
Assumption \ref{ass:pressure} can satisfy both.
\end{proof}

\begin{remark}[Reading the classification geometrically]
\label{rem:classification-geometric}
Theorem \ref{thm:Lambda-N-classification} and Proposition
\ref{lem:Lambda-N-geometric} together say that the two hydrostatic
profiles singled out by $\Lambda$ and $\mathcal{N}$ -- affine and exponential in
$y$, respectively -- occur only for the quadratic and the isothermal
pressure laws, and for no others. This is more than an analogy: the
change of variables of Ersoy--Ngom--Sy \cite{ErsoyNgomSy}, recalled in
\S\ref{subsec:wdj-transfer} below, is the operation of
factoring out the exponential-in-$y$ part of the density profile,
$\xi:=\rho\,e^{y/H_s}$ for the stratification height $H_s$ made precise
below; by Proposition \ref{lem:Lambda-N-geometric} this factorization
produces a genuinely $y$-independent $\xi$ if and only if $\rho(t,x,\cdot)$
\emph{is} a pure exponential in $y$, that is, if and only if $\mathcal{N}\equiv0$,
that is (Theorem \ref{thm:Lambda-N-classification}) if and only if the
pressure law is isothermal. This gives a complete, rigorous answer to
the question of why the change of variables underlying
\cite{ErsoyNgomSy,WangDouJiu} is available only at the isothermal law,
in the following precise sense:
the isothermal law is the only pressure law for which a factorization
$\rho=\xi(t,x)\varphi(y)$, with $\xi$ independent of the vertical
variable, arises from the hydrostatic first integral -- available if,
and only if, $\mathcal{N}\equiv0$. Proposition \ref{prop:ens-montgomery} below
sharpens this further: the Ersoy--Ngom--Sy variable $\xi$ is not merely
enabled by $\mathcal{N}\equiv0$, but is, at the isothermal law, literally
$\mathfrak M$ written in exponential form.
\end{remark}

\begin{proposition}[A structural characterization of constant $\Lambda$]
\label{prop:theta-lambda-const}
$\Lambda$ is a (not necessarily zero) constant on $(0,\infty)$ if and
only if
\[
\frac1{\mathfrak h'(\rho)}=a_0+a_1\rho
\]
for constants $a_0,a_1\in\mathbb R$, in which case $\Lambda\equiv-ga_1$.
The quadratic law of Theorem \ref{thm:Lambda-N-classification}(i) is
the sub-case $a_1=0$; the isothermal law of Theorem
\ref{thm:Lambda-N-classification}(ii) is the sub-case $a_0=0$.
\end{proposition}
\begin{proof}
Set $u(\rho):=1/\mathfrak h'(\rho)$, well defined and $C^1$ on
$(0,\infty)$ since $\mathfrak h'>0$ by Assumption \ref{ass:pressure}. By
the quotient rule, $u'(\rho)=-\mathfrak h''(\rho)/\mathfrak
h'(\rho)^2=-\Lambda(\rho)/g$. If $\Lambda\equiv\Lambda_0$ on
$(0,\infty)$, then $u'\equiv-\Lambda_0/g$ is constant, so
$u(\rho)=a_0-(\Lambda_0/g)\rho$ for a constant $a_0$, i.e.\
$1/\mathfrak h'(\rho)=a_0+a_1\rho$ with $a_1:=-\Lambda_0/g$. Conversely,
if $1/\mathfrak h'(\rho)=a_0+a_1\rho$, then $u'\equiv a_1$, so
$\Lambda=-gu'\equiv-ga_1$ is constant.
\end{proof}

\begin{remark}[Motivation, and limits of this reformulation]
\label{rem:theta-lambda-const}
Proposition \ref{prop:theta-lambda-const} has a motivating reading in
terms of the vertical mass flux $\rho v$: reconstructing it from the
continuity equation by integrating in $y$ against the horizontal mass
flux $\rho u$ produces, in general, a horizontal-divergence term plus a
residual coupled to $\nabla_x\mathfrak M$; on the family
$1/\mathfrak h'(\rho)=a_0+a_1\rho$ isolated above, this residual is
absent and the reconstruction reduces to a pure horizontal divergence.
This purely algebraic reformulation does not, however, improve the
closure of mechanism (b) in Theorem \ref{thm:kappa-entropy} at the
level of the present augmented identity: its first cross term is
recovered, term for term, by a three-line horizontal integration by
parts that never invokes this reconstruction, and its second cross term
becomes strictly more complicated under it, introducing a second
horizontal derivative of $\log\rho$ absent elsewhere in this note. This
alternative representation therefore gives an additional structural
interpretation of constant-$\Lambda$ laws, but does not by itself
supply the missing control on $\nabla_xv$.
\end{remark}

This requires one further hypothesis, additional to Assumption
\ref{ass:viscosities} and used only for the remainder of this section:

\begin{assumption}[Linear horizontal viscosity]
\label{ass:linear-viscosity}
$\nu_1(\rho)=\bar\nu_1\rho$ for a constant $\bar\nu_1>0$. Then \eqref{eq:s-def}
gives $s(\rho)=\bar\nu_1\log\rho$ (up to an additive constant, immaterial
since only $\nabla_xs$ enters \eqref{eq:augmented-velocity-def}), so
that the augmented velocity of Section \ref{sec:Lambda-N} is
$U=\bm{u}+2\kappa\bar\nu_1\nabla_x\log\rho$. We do \emph{not} fix $\kappa$: it is
left free in $[0,1]$ throughout the remainder of this section, and every
statement below holds for every such $\kappa$. Its two endpoints recover
the two constructions already available in the literature: $\kappa=0$
gives $U=\bm{u}$, i.e.\ no augmentation at all, and the identity below
reduces to the total energy identity of Theorem \ref{thm:total-energy};
$\kappa=1$ gives the classical Bresch--Desjardins entropy construction,
which \cite{WangDouJiu} itself uses and which, as Remark
\ref{rem:wdj-mechanism} makes precise, is also the formal, non-attained
endpoint of the $\kappa$-entropy method of Bresch, Vasseur and Yu
\cite{BVY}.
\end{assumption}

Under Assumption \ref{ass:linear-viscosity}, $\nu_1(\rho)$ solves
exactly the continuity equation, at the scale $\bar\nu_1$: no defect term
of the form $(\rho\nu_1'(\rho)-\nu_1(\rho))(\operatorname{div}_x\bm{u}+\partial_yv)$,
present for a general $\nu_1$, arises below.

\begin{remark}[Pressure-driven viscosity laws and the Bresch--Desjardins relation]
\label{rem:pressure-driven-viscosity}
Assumption \ref{ass:linear-viscosity} is the first member of a natural
one-parameter family, $\nu_1(\rho)=C_1\,p(\rho)$, $C_1>0$, under which
the viscous drift entering the augmented velocity
\eqref{eq:augmented-velocity-def} aligns with the first hydrostatic
integral itself: by \eqref{eq:s-def},
$s'(\rho)=\nu_1'(\rho)/\rho=C_1\,\mathfrak h'(\rho)$, so
$s(\rho)=C_1\,\mathfrak h(\rho)+C_2$ and, since $\Phi$ does not depend on
$x$, $\nabla_xs=C_1\,\nabla_x\mathfrak M$; as $\mathfrak M$ is vertically
constant for every barotropic law (Proposition \ref{lem:first-integral}),
$\partial_y\nabla_xs=0$ throughout this family, not only at the linear
case. This same family fixes the constitutive combination
\[
\rho\nu_1'(\rho)-\nu_1(\rho)=C_1\bigl(\rho p'(\rho)-p(\rho)\bigr),
\]
which is, up to the factor $2$, the classical Bresch--Desjardins
quantity $\lambda(\rho)=2(\rho\nu_1'(\rho)-\nu_1(\rho))$: the family
$\nu_1(\rho)=C_1\,p(\rho)$ reproduces this combination directly from the
pressure law, rather than imposing it as a separate structural
hypothesis on $\nu_1$, and vanishes within it exactly at the isothermal
law, where $\nu_1(\rho)=C_1\,p(\rho)$ itself reduces to Assumption
\ref{ass:linear-viscosity}. We record this alignment as a natural point
of contact between the first hydrostatic integral and the classical
Bresch--Desjardins relation; we do \emph{not} claim that it produces a
closed hydrostatic Bresch--Desjardins entropy. The augmented-velocity
identity of Theorem \ref{thm:kappa-entropy}, and with it the obstruction
of Remark \ref{prop:mechanism-b-obstruction}, is established under
the linear viscosity of Assumption \ref{ass:linear-viscosity}, which
this family reaches only at the isothermal law ($C_1=1/c^2$); no
closure, at that or any other member of the family, is claimed or
obtained here. Whether the broader viscosity class of Bresch, Vasseur
and Yu \cite{BVY} fares differently at the polytropic law is taken up in
the Perspectives of Section \ref{sec:conclusion}.
\end{remark}

Throughout this proof, tensor indices follow the convention
$(a\otimes b)_{ij}:=a_ib_j$, $(\nabla_x\bm{u})_{ij}:=\partial_ju_i$, and
$(\operatorname{div}_xA)_i:=\sum_j\partial_jA_{ij}$ -- the unique convention
compatible with the already-established identity
$\int_\Omega\operatorname{div}_x(\nu_1D_x\bm u)\cdot\bm u=-\int_\Omega\nu_1|D_x\bm u|^2$
used in the proof of Lemma \ref{lem:kinetic}.

\begin{lemma}[Augmented momentum equation]
\label{lem:augmented-equation}
Under Assumption \ref{ass:linear-viscosity}, for a regular solution of
\eqref{eq:CPE} and every $\kappa\in[0,1]$,
\begin{equation}
\label{eq:W-equation}
\begin{aligned}
\partial_t(\rho U)+\operatorname{div}_x(\rho\,U\otimes\bm{u})+\partial_y(\rho vU)+\nabla_xp(\rho)
&=2(1-\kappa)\operatorname{div}_x\bigl(\nu_1(\rho)D_x\bm{u}\bigr)
+2\kappa\operatorname{div}_x\bigl(\nu_1(\rho)A_x\bm{u}\bigr)\\
&\quad+\partial_y\bigl(\nu_2(\rho)\partial_y\bm{u}\bigr)
-2\kappa\partial_y\bigl(\nu_1(\rho)\nabla_xv\bigr),
\end{aligned}
\end{equation}
where $D_x\bm{u}:=\tfrac12(\nabla_x\bm{u}+(\nabla_x\bm{u})^\top)$ and
$A_x\bm{u}:=\tfrac12(\nabla_x\bm{u}-(\nabla_x\bm{u})^\top)$. Note the order
$U\otimes\bm u$, not $\bm u\otimes U$: $U$ is the transported quantity, $\bm u$
the (mass-conserving) transport velocity, exactly as in the base momentum
equation's own $\rho\bm u\otimes\bm u$ term.
\end{lemma}
\begin{proof}
Set $q:=\nabla_x\nu_1(\rho)=\bar\nu_1\nabla_x\rho$. Since $\nu_1(\rho)=\bar\nu_1\rho$
(Assumption \ref{ass:linear-viscosity}), the continuity equation
gives, without correction,
\[
\partial_t\nu_1(\rho)+\operatorname{div}_x(\nu_1(\rho)\bm{u})+\partial_y(\nu_1(\rho)v)=0.
\]
Applying $\nabla_x$ (the operators $\partial_t,\partial_y,\nabla_x$
commute) and the elementary identity
$\partial_{x_i}\partial_{x_j}(f\chi_j)=\partial_{x_j}(\chi_j\partial_{x_i}f)+\partial_{x_j}(f\partial_{x_i}\chi_j)$,
for $f=\nu_1(\rho)$ and $\chi=\bm{u}$, then $\chi=v$ (scalar case), yields
\[
\partial_tq+\operatorname{div}_x(q\otimes\bm{u})+\partial_y(vq)
+\operatorname{div}_x\bigl(\nu_1(\rho)(\nabla_x\bm{u})^\top\bigr)
+\partial_y\bigl(\nu_1(\rho)\nabla_xv\bigr)=0,
\]
where, precisely, $\sum_j\partial_j(u_jq_i)=q_i\operatorname{div}_x\bm u+(\bm
u\cdot\nabla_x)q_i=(\operatorname{div}_x(q\otimes\bm u))_i$ under the
convention above -- the transported quantity $q$ occupies the first slot, the
transport velocity $\bm u$ the second. Add $2\kappa$ times this equation to
the momentum equation of
\eqref{eq:CPE}. Since $q=\rho\nabla_xs(\rho)=\bar\nu_1\nabla_x\rho$ under
Assumption \ref{ass:linear-viscosity}, $\rho U=\rho \bm{u}+2\kappa q$; one
checks componentwise that $\rho \bm{u}\otimes \bm{u}+2\kappa\,q\otimes\bm{u}=\rho
\,U\otimes\bm{u}$ and $\rho v\bm{u}+2\kappa vq=\rho vU$, and, with
$(\nabla_x\bm{u})^\top=D_x\bm{u}-A_x\bm{u}$,
\[
2\nu_1(\rho)D_x\bm{u}-2\kappa\nu_1(\rho)(\nabla_x\bm{u})^\top
=2(1-\kappa)\nu_1(\rho)D_x\bm{u}+2\kappa\nu_1(\rho)A_x\bm{u}.
\]
\end{proof}

\begin{corollary}[Equation of the drift]
\label{cor:drift-equation}
Under the hypotheses of Lemma \ref{lem:augmented-equation}, set
$V:=2\bar\nu_1\nabla_x\log\rho$, so that $U=\bm{u}+\kappa V$ and, in the
notation of the proof above, $q=\rho V/2$. Multiplying the equation for
$q$ established in that proof by $2$,
\begin{equation}
\label{eq:V-equation}
\partial_t(\rho V)+\operatorname{div}_x(\rho V\otimes\bm{u})+\partial_y(\rho vV)
+2\operatorname{div}_x\bigl(\nu_1(\rho)(\nabla_x\bm{u})^\top\bigr)
+2\partial_y\bigl(\nu_1(\rho)\nabla_xv\bigr)=0.
\end{equation}
\end{corollary}

We stress in advance what the next statement is, and is not: an
\emph{identity}, valid for every $\kappa\in[0,1]$, not an a priori
estimate. The left-hand side is not yet controlled by the terms
available: alongside the augmented kinetic energy $\tfrac12\rho|U|^2$,
it carries the drift's own kinetic energy $\tfrac{\kappa(1-\kappa)}2\rho|V|^2$
-- without which \eqref{eq:kappa-entropy} below would not be the
genuine two-velocity $\kappa$-entropy of Bresch, Desjardins and
Zatorska \cite{BreschDesjardinsZatorska} (Remark
\ref{rem:mechanism-b-open} details the point) -- together with the
vertical dissipation and the $\kappa$-weighted horizontal viscous
components. Whether this yields a closed a priori estimate depends
entirely on the right side, where two residual mechanisms survive:
mechanism (a), examined in Proposition \ref{prop:mechanism-a-closure},
and mechanism (b), examined in Remark
\ref{prop:mechanism-b-obstruction}. A third, formally expected cross
term, tied to the horizontal Hessian $\nabla_x^2\log\rho$, cancels
\emph{identically} once the drift's own equation \eqref{eq:V-equation}
is tested against $\kappa(1-\kappa)V$ and added to the balance, and is
therefore absent from \eqref{eq:kappa-entropy} below by construction;
the cancellation is recorded in the proof, and the endpoints
$\kappa=0,1$ are treated in Remark \ref{rem:mechanism-b-open}.

\begin{theorem}[Augmented-velocity identity, not a closed entropy]
\label{thm:kappa-entropy}
Under Assumption \ref{ass:linear-viscosity}, for a regular solution of
\eqref{eq:CPE} satisfying the boundary conditions of Section
\ref{sec:model}, and every $\kappa\in[0,1]$,
\begin{equation}
\label{eq:kappa-entropy}
\begin{aligned}
&\frac{d}{dt}\int_\Omega\Bigl(\tfrac12\rho|U|^2+\tfrac{\kappa(1-\kappa)}2\rho|V|^2+e(\rho)+\Phi(y)\rho\Bigr)\,dx\,dy\\
&\qquad+2(1-\kappa)\int_\Omega\nu_1(\rho)|D_x\bm{u}|^2\,dx\,dy
+2\kappa\int_\Omega\nu_1(\rho)|A_x\bm{u}|^2\,dx\,dy\\
&\qquad+\int_\Omega\nu_2(\rho)|\partial_y\bm{u}|^2\,dx\,dy
+2\kappa\bar\nu_1\int_\Omega \mathfrak{h}'(\rho)|\nabla_x\rho|^2\,dx\,dy\\
&\qquad=
\underbrace{-2\kappa\bar\nu_1\int_\Omega\nu_2(\rho)\mathcal{N}(\rho)\,\partial_y\bm{u}\cdot\nabla_x\rho\,dx\,dy}_{\text{mechanism (a)}}\\
&\qquad\quad+\underbrace{2\kappa\int_\Omega\nu_1(\rho)\,\partial_y\bm{u}\cdot\nabla_xv\,dx\,dy
+4\kappa\bar\nu_1\int_\Omega\nu_1(\rho)\mathcal{N}(\rho)\,\nabla_x\rho\cdot\nabla_xv\,dx\,dy}_{\text{mechanism (b)}}.
\end{aligned}
\end{equation}
\end{theorem}
\begin{proof}
Test \eqref{eq:W-equation} against $U$ and integrate over
$\Omega$. For any $W$ transported by the physical flow $(\bm u,v)$, the
material-derivative identity $\partial_t(\rho W)+\operatorname{div}_x(\rho
W\otimes\bm u)+\partial_y(\rho vW)=\rho D_tW$ (with $D_tW:=\partial_tW+(\bm
u\cdot\nabla_x)W+v\partial_yW$) follows by expanding both sides and using the
continuity equation to cancel the coefficient of $W$; applying it with
$W=U$ and dotting with $U$ gives $U\cdot\rho D_tU=\rho D_t(\tfrac12|U|^2)$,
and applying it again with the scalar $W=\tfrac12|U|^2$ gives, as in Lemma
\ref{lem:kinetic} but with $U$ in place of $\bm{u}$,
\[
U\cdot\bigl[\partial_t(\rho U)+\operatorname{div}_x(\rho\,U\otimes\bm{u})+\partial_y(\rho vU)\bigr]
=\partial_t\bigl(\tfrac12\rho|U|^2\bigr)+\operatorname{div}_x\bigl(\tfrac12\rho|U|^2\bm{u}\bigr)+\partial_y\bigl(\tfrac12\rho|U|^2v\bigr);
\]
integrating, the horizontal flux vanishes by periodicity and the
vertical flux by $v|_{y=0,H}=0$, leaving
$\frac{d}{dt}\int_\Omega\tfrac12\rho|U|^2\,dx\,dy$, with no boundary
contribution. For the pressure term, using
$U=\bm{u}+2\kappa\bar\nu_1\nabla_x\log\rho$, by
\eqref{eq:augmented-velocity-def} and Assumption
\ref{ass:linear-viscosity},
\[
U\cdot\nabla_xp(\rho)=\bm{u}\cdot\nabla_xp(\rho)+2\kappa\bar\nu_1\mathfrak{h}'(\rho)|\nabla_x\rho|^2;
\]
Proposition \ref{prop:integrated-balance} accounts for $\int_\Omega
\bm{u}\cdot\nabla_xp(\rho)\,dx\,dy$. On the right, since
$\nabla_xU=\nabla_x\bm{u}+2\kappa\bar\nu_1\nabla_x^2\log\rho$ and the Hessian
$\nabla_x^2\log\rho$ is symmetric,
\[
D_xU=D_x\bm{u}+2\kappa\bar\nu_1\nabla_x^2\log\rho,
\qquad
A_xU=A_x\bm{u};
\]
testing the
$D_x\bm{u}$-flux of \eqref{eq:W-equation} against $U$ therefore contributes
$-2(1-\kappa)\int_\Omega\nu_1(\rho)|D_x\bm{u}|^2\,dx\,dy$ together with a term
$-4\kappa(1-\kappa)\bar\nu_1\int_\Omega\nu_1(\rho)\,D_x\bm{u}:\nabla_x^2\log\rho\,dx\,dy$
(cancelled below), and testing the $A_x\bm{u}$-flux, using $A_x\bm{u}:D_xU=0$,
contributes
$-2\kappa\int_\Omega\nu_1(\rho)|A_x\bm{u}|^2\,dx\,dy$ alone; the Neumann condition
$\partial_y\bm{u}|_{y=0,H}=0$ together with
$\partial_yU=\partial_y\bm{u}+2\kappa\bar\nu_1\mathcal{N}(\rho)\nabla_x\rho$ (Lemma
\ref{lem:N}) contributes $-\int_\Omega\nu_2(\rho)|\partial_y\bm{u}|^2\,dx\,dy$
together with mechanism (a); $v|_{y=0,H}=0$ on the entire boundary
hyperplane forces $\nabla_xv|_{y=0,H}=0$ as well, so the last term of
\eqref{eq:W-equation} integrates by parts with no boundary
contribution, giving $2\kappa\int_\Omega\nu_1(\rho)\,\partial_y\bm{u}\cdot\nabla_xv\,dx\,dy
+4\kappa^2\bar\nu_1\int_\Omega\nu_1(\rho)\mathcal{N}(\rho)\,\nabla_x\rho\cdot\nabla_xv\,dx\,dy$.

It remains to add to this balance the equation \eqref{eq:V-equation} for the
drift, tested against $\kappa(1-\kappa)V$; this converts the cancelled term
above into the drift's own kinetic energy and completes mechanism (b). The
material-derivative argument used at the start of this proof only used the
transport structure $\partial_t(\rho W)+\operatorname{div}_x(\rho
W\otimes\bm{u})+\partial_y(\rho vW)=\rho D_tW$, common to \eqref{eq:W-equation}
and \eqref{eq:V-equation}; applying it to $W=V$ and then to the scalar
$W=\tfrac12|V|^2$, exactly as for $U$ above, shows that testing the transport
part of \eqref{eq:V-equation} against $\kappa(1-\kappa)V$ and integrating
gives $\frac{d}{dt}\int_\Omega\tfrac{\kappa(1-\kappa)}2\rho|V|^2\,dx\,dy$, with
no boundary contribution (periodicity in $x$, $v|_{y=0,H}=0$). For the
horizontal viscous flux of \eqref{eq:V-equation}, $\nabla_xV=2\bar\nu_1\nabla_x^2\log\rho$
is symmetric, so
\[
\begin{aligned}
\kappa(1-\kappa)\int_\Omega V\cdot\Bigl[-2\operatorname{div}_x\bigl(\nu_1(\rho)(\nabla_x\bm{u})^\top\bigr)\Bigr]\,dx\,dy
&=2\kappa(1-\kappa)\int_\Omega\nu_1(\rho)\,\nabla_xV:(\nabla_x\bm{u})^\top\,dx\,dy\\
&=4\kappa(1-\kappa)\bar\nu_1\int_\Omega\nu_1(\rho)\,D_x\bm{u}:\nabla_x^2\log\rho\,dx\,dy,
\end{aligned}
\]
by parts, with no boundary contribution by periodicity in $x$; this is
exactly the negative of the term recorded above, and the two cancel
identically. For the vertical viscous flux of \eqref{eq:V-equation}, using
$\partial_yV=2\bar\nu_1\mathcal{N}(\rho)\nabla_x\rho$ (Lemma \ref{lem:N}) and
$\nabla_xv|_{y=0,H}=0$, already used above for mechanism (b),
\[
\begin{aligned}
\kappa(1-\kappa)\int_\Omega V\cdot\bigl[-2\partial_y\bigl(\nu_1(\rho)\nabla_xv\bigr)\bigr]\,dx\,dy
&=2\kappa(1-\kappa)\int_\Omega\partial_yV\cdot\nu_1(\rho)\nabla_xv\,dx\,dy\\
&=4\kappa(1-\kappa)\bar\nu_1\int_\Omega\nu_1(\rho)\mathcal{N}(\rho)\,\nabla_x\rho\cdot\nabla_xv\,dx\,dy,
\end{aligned}
\]
by parts in $y$, with no boundary contribution. Adding this to the
$4\kappa^2\bar\nu_1\int_\Omega\nu_1(\rho)\mathcal{N}(\rho)\,\nabla_x\rho\cdot\nabla_xv\,dx\,dy$
term recorded above and using $\kappa^2+\kappa(1-\kappa)=\kappa$ gives the
coefficient $4\kappa\bar\nu_1$ recorded in mechanism (b) of
\eqref{eq:kappa-entropy}. What was mechanism (c) is thereby absent from
\eqref{eq:kappa-entropy}: it cancels identically once
\eqref{eq:V-equation} is brought into the balance, rather than being
discarded or estimated.
\end{proof}

\begin{remark}[Gradient control and the two surviving mechanisms]
\label{rem:mechanism-b-open}
The term $2\kappa\bar\nu_1\int_\Omega \mathfrak{h}'(\rho)|\nabla_x\rho|^2\,dx\,dy$ appearing on
the left of \eqref{eq:kappa-entropy} is the gradient control announced
at the start of this section, active as soon as $\kappa>0$; for the
polytropic law it equals, by \eqref{eq:pressure-dissipation},
$\frac{8\kappa\bar\nu_1a}\gamma\int_\Omega|\nabla_x(\rho^{\gamma/2})|^2\,dx\,dy$.

At $\kappa=0$, $U=\bm{u}$, every term on the right of
\eqref{eq:kappa-entropy} vanishes together with the gradient control
itself, and the left side reduces to
$2\int_\Omega\nu_1(\rho)|D_x\bm{u}|^2\,dx\,dy+\int_\Omega\nu_2(\rho)|\partial_y\bm{u}|^2\,dx\,dy$:
\eqref{eq:kappa-entropy} is then exactly the total energy identity of
Theorem \ref{thm:total-energy}, with no augmentation at all. At
$\kappa=1$, the drift energy $\tfrac{\kappa(1-\kappa)}2\rho|V|^2$ vanishes --
its coefficient $\kappa(1-\kappa)$ is zero at both endpoints -- and
\eqref{eq:kappa-entropy} coincides exactly with the classical
single-velocity Bresch--Desjardins construction, with the dissipation
carried entirely by $A_x\bm{u}$ instead of $D_x\bm{u}$ and mechanisms (a) and
(b) at their largest coefficients. Since no single one of these
left-hand-side contributions is present at both endpoints,
\eqref{eq:kappa-entropy} does not, by itself, control
$U,\bm{u},\nabla_x\rho$ jointly and uniformly over $\kappa\in[0,1]$.

Whether \eqref{eq:kappa-entropy} upgrades, for some $\kappa\in(0,1]$,
from an identity to a genuine a priori estimate depends entirely on the
right-hand side. Mechanism (a) depends on $\nu_2$ alone and can, for the
polytropic law, be absorbed by a suitable choice of $\nu_2$ at every
fixed $\kappa\in(0,1]$ (Proposition \ref{prop:mechanism-a-closure}
below), and vanishes identically at the isothermal law for every $\nu_2$
(Theorem \ref{thm:Lambda-N-classification}). What would formally have been a third, $\kappa(1-\kappa)$-proportional
mechanism, tied to the horizontal Hessian $\nabla_x^2\log\rho$ rather than to
$\Lambda$ or $\mathcal{N}$, does not raise a separate question here: it is not
present in \eqref{eq:kappa-entropy} at all, since testing the drift's own
equation \eqref{eq:V-equation} against $\kappa(1-\kappa)V$ cancels it
identically against the drift energy on the left, as recorded in the proof
of Theorem \ref{thm:kappa-entropy}. Mechanism (b) does not depend on
$\nu_2$ at all, and its first term, present at \emph{every} barotropic
law and every $\kappa\in(0,1]$, is carried entirely by $\nabla_xv$. We
now show precisely why this term escapes control by the tools developed
so far, and why this is a structural consequence of the hydrostatic
reduction -- the replacement of the vertical momentum evolution
equation by the algebraic balance \eqref{eq:intro-hydrostatic}, not
gravity by itself -- and no estimate developed in the present direct
route controls it.

This $\nabla_xv$-term does not arise as a complication added onto an
independently established density-gradient bound: it sits on the
right-hand side of the very identity \eqref{eq:kappa-entropy} whose
left-hand side carries the candidate $\nabla_x\rho$-control, both
produced by testing the same augmented momentum equation against the
same $U$. Density compactness and vertical-velocity control are
consequently not two separate problems to be solved one after the
other, but two sides of a single unclosed identity: the density-gradient
term on the left of \eqref{eq:kappa-entropy} cannot be validated as a
genuine a priori bound -- and hence cannot supply the compactness the
chain described in the Introduction requires -- without first
controlling the $\nabla_xv$-term on the right.
\end{remark}

Integrating the continuity equation of \eqref{eq:CPE} in $y$ from $0$ to
$y$ at fixed $(t,x)$, and using $\rho v|_{y=0}=0$ (impermeability), gives
the reconstruction formula
\begin{equation}
\label{eq:v-reconstruction}
\rho(t,x,y)\,v(t,x,y)=-\int_0^y\Bigl[\partial_t\rho+\operatorname{div}_x(\rho \bm{u})\Bigr](t,x,y')\,dy':
\end{equation}
$v$ is not an independent unknown of \eqref{eq:CPE}, but, at fixed
$(t,x)$, the vertical primitive of $-[\partial_t\rho+\operatorname{div}_x(\rho \bm{u})]$
divided by $\rho$, satisfying no momentum-type equation, and hence no
energy or dissipation identity, of its own.

This is the obstacle on which the hydrostatic route parts ways with its
non-hydrostatic parent. The other ingredient the Bresch--Desjardins
route requires at this stage -- absorbing the cross term generated by
the degenerate viscosity, mechanism (a) of Theorem
\ref{thm:kappa-entropy} -- does transpose to the CPE, closing explicitly
for the polytropic law and vanishing identically at the isothermal one
(\S\ref{sec:polytropic}). Mechanism (b) is the one term of the identity
for which no such treatment, by the direct route developed here, is
available, at any barotropic law.

\begin{remark}[Absence of direct control of $\nabla_xv$ via the reconstruction route]
\label{prop:mechanism-b-obstruction}
For every barotropic law satisfying Assumption \ref{ass:pressure}, the
\emph{direct route} to a bound on $\nabla_xv$ -- through the
reconstruction formula \eqref{eq:v-reconstruction} alone -- fails to
supply one: by \eqref{eq:v-reconstruction}, $\nabla_xv$ involves
$\nabla_x\partial_t\rho$ and $\nabla_x\operatorname{div}_x(\rho \bm{u})$ --
quantities carrying, respectively, one derivative more in $x$ than
$\partial_t\rho$ and two derivatives in $x$ applied to $\rho \bm{u}$. Neither
is controlled by class (ii) of Definition \ref{def:weak-solution}, which
bounds only
\[
\rho\in L^\infty(0,T;L^1(\Omega)),\quad
e(\rho)\in L^\infty(0,T;L^1(\Omega)),\quad
\sqrt\rho\,\bm{u}\in L^\infty(0,T;L^2(\Omega)),\quad
\sqrt\rho\,v\in L^2((0,T)\times\Omega),
\]
together with $\sqrt{\nu_i(\rho)}$-weighted first derivatives of $\bm{u}$
alone; nor does
Theorem \ref{thm:total-energy} add any control on $\nabla_x\rho$ or on
any $x$-derivative of $\rho \bm{u}$ beyond what is already in this class,
since its proof (Lemma \ref{lem:kinetic}, Proposition
\ref{prop:integrated-balance}) never differentiates the continuity
equation in $x$. \textbf{The available energy structure developed in
this note does not provide, by this route, the regularity needed to
control this term.} This is a precise statement about the one route
examined here, not a claim that every conceivable indirect route --
for instance one built from the augmented-velocity identity of Theorem
\ref{thm:kappa-entropy} itself, which does control a weighted
$\nabla_x\rho$ -- is excluded; such an indirect bound on $\nabla_xv$ is
not obtained by the present argument. Consequently the term
$2\kappa\int_\Omega\nu_1(\rho)\partial_y\bm{u}\cdot\nabla_xv\,dx\,dy$ of mechanism
(b) in \eqref{eq:kappa-entropy}, present for every $\kappa\in(0,1]$, is
not absorbed into the non-negative left side of that identity by any
estimate built, via \eqref{eq:v-reconstruction}, from Theorem
\ref{thm:total-energy} and Definition \ref{def:weak-solution} alone.
This persists at every pressure law identified by Theorem
\ref{thm:Lambda-N-classification},
including the isothermal one: since this term does not involve $\mathcal{N}$, its
presence is unaffected by $\mathcal{N}\equiv0$, which only removes the second,
structurally milder term of mechanism (b) (Remark
\ref{rem:isothermal-dichotomy}). The full import of this observation --
that it reflects the hydrostatic regime itself rather than an artifact
of the particular estimates assembled here -- is the content of Remark
\ref{rem:gravity-role}, which should be read together with this remark.
\end{remark}

\begin{remark}[The role of the hydrostatic regime]
\label{rem:gravity-role}
The absence of direct control identified in Remark
\ref{prop:mechanism-b-obstruction} is a structural
consequence of the hydrostatic regime, not an accident of the
particular estimates developed here. In the hydrostatic regime, the
vertical momentum evolution equation of the parent compressible
Navier--Stokes system is replaced by the pressure--gravity balance
\eqref{eq:intro-hydrostatic}: no evolution equation is left to test $v$
against, and \eqref{eq:v-reconstruction} is
all that remains of it. Were the vertical balance retained as a genuine evolution
equation -- as it is in the non-hydrostatic theory -- $v$ would be tested against it exactly as $\bm{u}$ is in Lemma
\ref{lem:kinetic}, inheriting gradient control from its viscous term.
The hydrostatic reduction is thus not a passive simplification in Theorem
\ref{thm:total-energy}: it is the mechanism responsible for the absence
of the direct control Remark
\ref{prop:mechanism-b-obstruction} identifies.

Concretely, what resists absorption is the first term of mechanism (b),
$2\kappa\int_\Omega\nu_1(\rho)\partial_y\bm{u}\cdot\nabla_xv\,dx\,dy$, present for
every barotropic law and every $\kappa\in(0,1]$. Absorbing it by Young's
inequality, as Proposition \ref{prop:mechanism-a-closure} does for
mechanism (a), would need an a priori bound on
$\int_0^T\!\!\int_\Omega\nu_1(\rho)|\nabla_xv|^2\,dx\,dy\,dt$. In the
classical Bresch--Desjardins/Mellet--Vasseur/Vasseur--Yu theory for
compressible Navier--Stokes \cite{BD1,BD2,MelletVasseur,VasseurYu,VasseurYu2,BVY},
this is exactly the bound testing the vertical momentum equation
against its own augmented velocity would supply: that theory treats
every velocity component symmetrically, testing each against its own
momentum equation. Gravity by itself does not remove this symmetry --
in the non-hydrostatic compressible Navier--Stokes equations, gravity
is present and the vertical momentum equation still exists, so $v$ can
still be tested against it. What removes the symmetry here is the
hydrostatic reduction: it is the modeling step of replacing the
vertical momentum \emph{evolution} equation by the algebraic
pressure--gravity balance \eqref{eq:intro-hydrostatic} that eliminates
the equation $v$ would need to be tested against. Gravity enters only
as the physical force balanced in \eqref{eq:intro-hydrostatic}, not as
the cause of the asymmetry itself. No estimate available in this note --
not Theorem \ref{thm:total-energy}, not Definition
\ref{def:weak-solution}, not the reconstruction
\eqref{eq:v-reconstruction} itself -- supplies a substitute. We do not
attempt to close this term artificially; identifying which classical
mechanism it corresponds to, and why the hydrostatic reduction removes
it along the direct route developed here, is a central structural
observation of this note -- not a claim that no other route can supply
the missing control.
\end{remark}

\section{Applications}
\label{sec:polytropic}

The classification of Theorem \ref{thm:Lambda-N-classification} singles
out two pressure laws by name -- quadratic and isothermal -- inside the
family that contains both and interpolates between them. We now make
the pressure potential and enthalpy, the structural functions
$\Lambda,\mathcal{N}$, and the absorption of mechanism (a) explicit on that family,
$p(\rho)=a\rho^\gamma$ with $a>0$ and $\gamma\ge1$. The
case $\gamma>1$, treated in \S\ref{subsec:polytropic}, is illustrated
by the construction of Definition \ref{def:pressure-potential} with
$\rho_\star=0$. Its boundary member $\gamma=1$ -- the isothermal law
$p(\rho)=c^2\rho$, i.e.\ $a=c^2$ -- is a genuine particular case of the
same family, not a law external to it: it is treated separately in
\S\ref{subsec:isothermal} only because that construction, specifically
the choice $\rho_\star=0$, fails precisely at $\gamma=1$ and an
independent computation is required there.

\subsection{The polytropic case}
\label{subsec:polytropic}

We illustrate the preceding structure for
\begin{equation}
\label{eq:polytropic}
p(\rho)=a\rho^\gamma,\qquad a>0,\qquad\gamma>1,
\end{equation}
strictly increasing with $p(0)=0$. Since $\varrho\mapsto p(\varrho)/\varrho^2=a\varrho^{\gamma-2}$
is integrable at $0$ for $\gamma>1$, we may take $\rho_\star=0$ in
Definition \ref{def:pressure-potential}.

With $\rho_\star=0$ in Definition \ref{def:pressure-potential}, direct
computation of $e'(\rho)$ gives, for the polytropic law
\eqref{eq:polytropic},
\begin{equation}
\label{eq:polytropic-potential}
e(\rho)=\frac a{\gamma-1}\rho^\gamma,\qquad
\mathfrak{h}(\rho)=\frac{a\gamma}{\gamma-1}\rho^{\gamma-1},
\end{equation}
consistently with \eqref{eq:e-identities}: $\rho
e'(\rho)-e(\rho)=a\rho^\gamma=p(\rho)$ and
$e''(\rho)=a\gamma\rho^{\gamma-2}=p'(\rho)/\rho$.

\begin{proposition}[Polytropic hydrostatic profile, and possible vacuum]
\label{prop:polytropic-profile}
\label{rem:vacuum}
Every regular, strictly positive density satisfying
\eqref{eq:intro-hydrostatic} under \eqref{eq:polytropic} satisfies
\begin{equation}
\label{eq:polytropic-profile}
\rho(t,x,y)=\left[\frac{\gamma-1}{a\gamma}\bigl(\mathfrak{M}(t,x)-\Phi(y)\bigr)\right]^{\frac1{\gamma-1}}
\end{equation}
wherever $\mathfrak{M}(t,x)-\Phi(y)>0$. When $\mathfrak{M}(t,x)-\Phi(y)$ reaches zero, one writes
instead $\rho=\bigl[\tfrac{\gamma-1}{a\gamma}(\mathfrak{M}-\Phi)\bigr]_+^{1/(\gamma-1)}$,
$\eta_+=\max\{\eta,0\}$; Lemma \ref{lem:polytropic-vacuum-weak} below shows
this extension genuinely solves \eqref{eq:HB-weak}, rather than merely
being declared admissible by convention. Unlike the isothermal case, for which
$\mathfrak{h}(\rho)=c^2\log\rho$ is a bijection of $(0,\infty)$ onto $\mathbb R$,
so that a regular, strictly positive profile represented by a finite
Montgomery potential cannot reach vacuum at any finite vertical level
(made explicit in Proposition \ref{prop:isothermal-profile} below --
whole-column vacuum remains admissible in the distributional weak
formulation of Definition \ref{def:weak-solution}, by Proposition
\ref{prop:wdj-hydrostatic-transfer}), the polytropic enthalpy
\eqref{eq:e-identities} maps $(0,\infty)$ onto $(0,\infty)$ only, so
that $\mathfrak{M}-\Phi\le0$ genuinely corresponds to vacuum already at
the level of regular profiles.
\end{proposition}
\begin{proof}
By \eqref{eq:first-integral}, $\frac{a\gamma}{\gamma-1}\rho^{\gamma-1}=\mathfrak{M}-\Phi(y)$;
raise to the power $1/(\gamma-1)$.
\end{proof}

\begin{lemma}[The vacuum-extended polytropic profile solves \eqref{eq:HB-weak}]
\label{lem:polytropic-vacuum-weak}
Fix $(t,x)$, $\mathfrak M\in\mathbb R$, and set, for \eqref{eq:polytropic},
\[
\rho(y):=\Bigl[\tfrac{\gamma-1}{a\gamma}\bigl(\mathfrak M-\Phi(y)\bigr)\Bigr]_+^{1/(\gamma-1)},
\qquad y\in(0,H).
\]
Then $y\mapsto p(\rho(y))$ is $C^1$ on $(0,H)$, with
$\partial_yp(\rho(y))=-g\rho(y)$ classically at \emph{every} $y$ --
including at $y_0:=\mathfrak M/g$ when $y_0\in(0,H)$, where the profile
crosses into vacuum. Consequently the extension by $\eta_+$ in
Proposition \ref{prop:polytropic-profile} is not a notational
convention standing in for an unresolved distributional statement:
$\rho$ solves the one-dimensional relation
\eqref{eq:HB-weak-fiber} classically for every $\gamma>1$, and
integrating in $x$ as in the proof of Lemma
\ref{lem:HB-weak-pointwise} recovers \eqref{eq:HB-weak} whenever
$\mathfrak M(t,\cdot)$ is measurable.
\end{lemma}
\begin{proof}
Write $\eta(y):=\mathfrak M-\Phi(y)=\mathfrak M-gy$, affine (hence
$C^\infty$) with $\eta'\equiv-g$, and $\kappa:=(\gamma-1)/(a\gamma)>0$, so
that $\rho=(\kappa \eta_+)^{1/(\gamma-1)}$ and
$p(\rho)=a\rho^\gamma=a\kappa^{\gamma/(\gamma-1)}(\eta_+)^{\gamma/(\gamma-1)}$.
Since $\gamma>1$, the exponent $\gamma/(\gamma-1)$ exceeds $1$, so
$s\mapsto(s_+)^{\gamma/(\gamma-1)}$ is $C^1$ on all of $\mathbb R$: for
$s>0$ its derivative is $\tfrac{\gamma}{\gamma-1}s^{1/(\gamma-1)}$,
which tends to $0$ as $s\to0^+$ because the exponent $1/(\gamma-1)$ is
positive, matching the derivative $0$ from $s\le0$. As a $C^1$ function
composed with the affine $\eta$, $p(\rho)$ is therefore $C^1$ on $(0,H)$,
with, by the chain rule,
\[
\partial_yp(\rho(y))
=a\kappa^{\gamma/(\gamma-1)}\cdot\frac\gamma{\gamma-1}\,\eta_+(y)^{1/(\gamma-1)}\cdot(-g)
=-g\cdot\frac{a\gamma\kappa}{\gamma-1}\cdot\kappa^{1/(\gamma-1)}\eta_+(y)^{1/(\gamma-1)}
=-g\cdot\frac{a\gamma\kappa}{\gamma-1}\,\rho(y),
\]
using $\kappa^{1/(\gamma-1)}\eta_+(y)^{1/(\gamma-1)}=(\kappa \eta_+(y))^{1/(\gamma-1)}=\rho(y)$.
Since $\kappa=(\gamma-1)/(a\gamma)$, $a\gamma\kappa/(\gamma-1)=1$
identically, giving $\partial_yp(\rho(y))=-g\rho(y)$ at every $y$,
including $y_0$ (where both sides vanish). The remaining assertion,
integrating \eqref{eq:HB-weak-fiber} in $x$ against $\varphi\in
C_c^\infty(\mathbb T^2)$ to recover \eqref{eq:HB-weak}, is the same
Fubini argument used in the converse direction of Lemma
\ref{lem:HB-weak-pointwise}.
\end{proof}

\begin{remark}[Regularity of $\rho$ itself at the free boundary]
\label{rem:physical-vacuum-threshold}
Unlike $p(\rho)$, the density $\rho=(\kappa \eta_+)^{1/(\gamma-1)}$ need
not be $C^1$ at $y_0$; set $\alpha:=1/(\gamma-1)$. For $1<\gamma<2$,
$\alpha>1$, so
\[
\frac{d}{d\eta}\eta_+^\alpha=\alpha \eta_+^{\alpha-1}\longrightarrow0
\qquad\text{as }\eta\to0^+,
\]
and $\rho$ is $C^1$ at the interface, with vanishing derivative there.
At $\gamma=2$, $\alpha=1$ and the profile is piecewise linear: Lipschitz
but not $C^1$. For $\gamma>2$, $\alpha\in(0,1)$ and $\rho\in
C^{0,\alpha}$ only, with a genuine cusp -- loosely analogous to the
\emph{physical vacuum} degeneracy of the free-boundary compressible
Euler and primitive equations (Liu and Titi \cite{LiuTiti1,LiuTiti2},
\S\ref{sec:related-work}), an analogy in regularity class only. That
$p(\rho)$, not $\rho$, is regular across the interface at every
$\gamma>1$ is why Lemma \ref{lem:polytropic-vacuum-weak} is phrased
through $p(\rho)$.
\end{remark}

For later use, note the elementary pressure-dissipation rewriting: since
$p'(\rho)/\rho=a\gamma\rho^{\gamma-2}$ and
$\nabla_x\rho^{\gamma/2}=\tfrac\gamma2\rho^{\gamma/2-1}\nabla_x\rho$, so that
$|\nabla_x\rho^{\gamma/2}|^2=\tfrac{\gamma^2}4\rho^{\gamma-2}|\nabla_x\rho|^2$,
\begin{equation}
\label{eq:pressure-dissipation}
\frac{p'(\rho)}\rho|\nabla_x\rho|^2
=\frac{4a}\gamma\bigl|\nabla_x\rho^{\gamma/2}\bigr|^2.
\end{equation}
This rewriting, the standard tool of Bresch--Desjardins-type entropy
estimates, is the one required to control the density gradient in any
subsequent construction of a hydrostatic entropy functional for this
system.

\begin{proposition}[Polytropic structural functions]
\label{prop:polytropic-Lambda-N}
For \eqref{eq:polytropic},
\begin{equation}
\label{eq:polytropic-Lambda-N}
\Lambda(\rho)=\frac{g(\gamma-2)}{a\gamma}\,\rho^{1-\gamma},
\qquad
\mathcal{N}(\rho)=\frac{g(\gamma-1)}{a\gamma}\,\rho^{-\gamma}.
\end{equation}
\end{proposition}
\begin{proof}
By \eqref{eq:polytropic-potential},
$\mathfrak{h}'(\rho)=a\gamma\rho^{\gamma-2}$,
$\mathfrak{h}''(\rho)=a\gamma(\gamma-2)\rho^{\gamma-3}$, so
$\Lambda(\rho)=g\,a\gamma(\gamma-2)\rho^{\gamma-3}/(a\gamma\rho^{\gamma-2})^2
=g(\gamma-2)\rho^{1-\gamma}/(a\gamma)$. For $\mathcal{N}$,
$\Lambda(\rho)/\rho+g/(\mathfrak{h}'(\rho)\rho^2)
=\tfrac{g(\gamma-2)}{a\gamma}\rho^{-\gamma}+\tfrac g{a\gamma}\rho^{-\gamma}
=\tfrac{g(\gamma-1)}{a\gamma}\rho^{-\gamma}$.
\end{proof}

\begin{proposition}[Algebraic absorption criterion for mechanism (a), polytropic law]
\label{prop:mechanism-a-closure}
Under Assumption \ref{ass:linear-viscosity}, for \eqref{eq:polytropic},
every $\kappa\in(0,1]$, and $\nu_2(\rho)=\bar\nu_2\rho^{3\gamma-2}$ with
$\bar\nu_2\le a^3\gamma^3/(2\kappa\bar\nu_1g^2(\gamma-1)^2)$, mechanism (a)
of \eqref{eq:kappa-entropy} satisfies, for every regular solution for
which the integrals below are finite,
\begin{equation}
\label{eq:mechanism-a-absorbed}
\Bigl|-2\kappa\bar\nu_1\int_\Omega\nu_2(\rho)\mathcal{N}(\rho)\,\partial_y\bm{u}\cdot\nabla_x\rho\,dx\,dy\Bigr|
\le\tfrac12\int_\Omega\nu_2(\rho)|\partial_y\bm{u}|^2\,dx\,dy+\kappa\bar\nu_1a\gamma\int_\Omega\rho^{\gamma-2}|\nabla_x\rho|^2\,dx\,dy,
\end{equation}
so that the right-hand side of \eqref{eq:kappa-entropy} reduces to
mechanism (b) alone, at the cost of half of each of the two
coercive terms it consumes. This is an algebraic absorption of one term
by two others already present in \eqref{eq:kappa-entropy}, valid
wherever the identity itself is; it is not, on its own, an entropy
estimate uniform over a sequence of solutions, which would additionally
require $\nu_2(\rho)=\bar\nu_2\rho^{3\gamma-2}\in
L^1((0,T)\times\Omega)$ as in Definition \ref{def:weak-solution}(ii) --
an integrability question not examined here (see the discussion
following the proof).
\end{proposition}
\begin{proof}
By Young's inequality with parameter $\tfrac12$, the left side is
bounded by
\[
\tfrac12\int_\Omega\nu_2(\rho)|\partial_y\bm{u}|^2\,dx\,dy
+2\kappa^2\bar\nu_1^2\int_\Omega\nu_2(\rho)\mathcal{N}(\rho)^2|\nabla_x\rho|^2\,dx\,dy.
\]
By Proposition \ref{prop:polytropic-Lambda-N}, $\mathcal{N}(\rho)^2=g^2(\gamma-1)^2\rho^{-2\gamma}/(a\gamma)^2$, so
\[
\frac{\nu_2(\rho)\mathcal{N}(\rho)^2}{\rho^{\gamma-2}}
=\frac{\bar\nu_2g^2(\gamma-1)^2}{(a\gamma)^2}\,\rho^{(3\gamma-2)-2\gamma-(\gamma-2)}
=\frac{\bar\nu_2g^2(\gamma-1)^2}{(a\gamma)^2},
\]
independent of $\rho$ -- the unique exponent $3\gamma-2$ with this
property, valid uniformly near vacuum and at large density, in the
absence of any $L^\infty$ bound on $\rho$. The stated threshold on
$\bar\nu_2$ makes
\[
2\kappa^2\bar\nu_1^2\bar\nu_2g^2(\gamma-1)^2/(a\gamma)^2\le\kappa\bar\nu_1a\gamma.
\]
\end{proof}

The choice $\nu_2(\rho)=\bar\nu_2\rho^{3\gamma-2}$ degenerates faster at
vacuum than the linear $\nu_1$ of Assumption
\ref{ass:linear-viscosity} (since $3\gamma-2>1$ for $\gamma>1$), but the
principal integrability difficulty it raises, within the present weak
class, is not vacuum degeneracy but growth at large density:
$3\gamma-2>\gamma$ for every $\gamma>1$, whereas the basic polytropic
energy of Definition \ref{def:total-energy} controls only $\rho^\gamma$,
not $\rho^{3\gamma-2}$, in $L^1$. Whether $\nu_2(\rho)\in
L^1((0,T)\times\Omega)$, as required by Definition
\ref{def:weak-solution}(ii), holds under some additional a priori bound
is not examined here. As $\gamma\to1^+$ or
as $\kappa\to0^+$, the threshold on $\bar\nu_2$ diverges and the
constraint disappears entirely: the former is consistent with
$\mathcal{N}(\rho)\propto(\gamma-1)$ (Proposition \ref{prop:polytropic-Lambda-N})
vanishing at the isothermal law, where no absorption is needed since
mechanism (a) is then identically zero for every choice of $\nu_2$
(Remark \ref{rem:isothermal-dichotomy}); the latter simply reflects
that mechanism (a) carries an overall factor $\kappa$ and vanishes at
$\kappa=0$ together with every other term on the right of
\eqref{eq:kappa-entropy} (Remark \ref{rem:mechanism-b-open}). What
Proposition \ref{prop:mechanism-a-closure} isolates, at every fixed
$\kappa\in(0,1]$, is that mechanism (b) -- and in particular
its $\nabla_xv$-term (Remark
\ref{rem:mechanism-b-open}) -- remains unclosed within the direct
augmented-velocity route developed here, for the polytropic law.

\begin{remark}[A distinguished value, within the general classification]
\label{rem:gamma-2}
Within the polytropic family, $\gamma=2$ plays a distinguished role
for $\Lambda$: one has $\Lambda\equiv0$ if and only if $\gamma=2$,
whereas $\mathcal{N}$ does not vanish identically for any $\gamma>1$. This is the polytropic instance of the general
classification of Theorem \ref{thm:Lambda-N-classification}, which
identifies $p(\rho)=a\rho^2$ (i.e.\ $\gamma=2$) as the unique barotropic
law, among \emph{all} those satisfying Assumption \ref{ass:pressure} and
not merely within the polytropic family, annihilating $\Lambda$. Unlike
the enthalpy formula \eqref{eq:polytropic-potential},
which diverges as $\gamma\to1^+$ for every fixed $\rho$ (the prefactor
$a\gamma/(\gamma-1)$, tied to the choice $\rho_\star=0$, blows up) and
therefore admits no isothermal limit, the formulas
\eqref{eq:polytropic-Lambda-N} for $\Lambda,\mathcal{N}$ depend on $\mathfrak{h}$ only
through $\mathfrak{h}',\mathfrak{h}''$ and remain perfectly regular at $\gamma=1$. The
resulting isothermal structural functions are computed independently,
and shown to agree exactly with the formal substitution $\gamma=1$ in
\eqref{eq:polytropic-Lambda-N}, in \S\ref{subsec:isothermal} below
(Proposition \ref{prop:isothermal-Lambda-N}), where $\mathcal{N}\equiv0$
identically -- the polytropic instance of the other half of Theorem
\ref{thm:Lambda-N-classification}; see Remark
\ref{rem:isothermal-dichotomy} there for this consistency check.
\end{remark}

\subsection{The isothermal case}
\label{subsec:isothermal}

We now treat the isothermal law
\begin{equation}
\label{eq:isothermal}
p(\rho)=c^2\rho,\qquad c>0,
\end{equation}
which satisfies Assumption \ref{ass:pressure}. Physically, $c$ is the
constant isothermal speed of sound: for an ideal gas at fixed absolute
temperature $\mathcal T_0$, the equation of state $p=\rho R\mathcal
T_0$ gives exactly \eqref{eq:isothermal} with $c^2=R\mathcal T_0$, $R$
the specific gas constant -- a physical identification used only for
interpretation and nowhere required by the analysis below. It is worth stressing
that this is not a law external to \S\ref{subsec:polytropic}: setting
$\gamma=1$ and $a=c^2$ in \eqref{eq:polytropic} reproduces
\eqref{eq:isothermal} exactly, so the isothermal law is itself the
particular case $\gamma=1$ of the polytropic family -- its classical
boundary member, corresponding to $n\to\infty$ in the polytropic-index
convention $\gamma=1+1/n$. What forces a separate treatment here is not
the pressure law but the construction underlying
\eqref{eq:polytropic-potential}: the choice $\rho_\star=0$ used there
for $\gamma>1$ is not admissible at $\gamma=1$, since $\varrho\mapsto
p(\varrho)/\varrho^2=c^2/\varrho$ fails to be integrable at $\varrho=0$; no reference density
$\rho_\star=0$ exists in Definition \ref{def:pressure-potential} for
the isothermal law, and a strictly positive $\rho_\star$ must be fixed
instead, requiring the independent computation below. The structural
functions $\Lambda,\mathcal{N}$, however, depend on $\mathfrak{h}$ only through $\mathfrak{h}',\mathfrak{h}''$ and
are insensitive to the choice of $\rho_\star$; as Remark
\ref{rem:isothermal-dichotomy} makes precise, they extend smoothly
across $\gamma=1$, unlike the enthalpy itself.

\begin{proposition}[Isothermal potential and enthalpy]
\label{prop:isothermal-potential}
\label{rem:isothermal-coercive}
With $\rho_\star=\mathrm{e}$ (Euler's number, typeset upright throughout
this proposition to distinguish it from the pressure potential $e(\rho)$)
in Definition \ref{def:pressure-potential},
\begin{equation}
\label{eq:isothermal-potential-family}
e(\rho)=c^2\rho(\log\rho-1),\qquad \mathfrak{h}(\rho)=c^2\log\rho.
\end{equation}
The isothermal weak-solution literature builds its energy estimate not
around $e$, but around $e$ harmonized relative to the background
density $\bar\rho=1$ (Lemma \ref{lem:relative-coercive}):
\begin{equation}
\label{eq:isothermal-potential}
e_1(\rho)=c^2\bigl(\rho\log\rho-\rho+1\bigr),
\qquad \mathfrak{h}(\rho)=e_1'(\rho)=c^2\log\rho.
\end{equation}
This is exactly the normalization of Wang, Dou and Jiu
\cite{WangDouJiu} (their internal-energy density $\xi\ln\xi-\xi+1$ in
the energy inequality of their eq.\ (2.6), stated in the
units $g=c^2=1$) and, more generally, of the Bresch--Desjardins entropy
literature. By Lemma \ref{lem:relative-coercive}, $e_1(\rho)\ge0$ for
every $\rho>0$, with equality iff $\rho=1$, whereas
$e(\rho)=e_1(\rho)-c^2$ vanishes only at $\rho=\mathrm{e}\approx2.718$ and
is negative on $(0,\mathrm{e})$ -- in particular at the background density
$\rho=1$, where $e(1)=-c^2$. This is exactly the coercivity underlying
the admissible initial-data class of \cite{WangDouJiu} (their eq.\
(2.4): $\xi_0\ln\xi_0-\xi_0+1\in L^1(\Omega)$), and it is the harmonized
energy $\widehat{\mathcal E}$ of \eqref{eq:harmonized-energy}, not the
bare $\mathcal E$ of Definition
\ref{def:total-energy}, that carries this coercivity into the total
energy; no such correction is needed in the polytropic case of
\S\ref{subsec:polytropic}, where $\rho_\star=0$ already gives
$e(\rho)=a\rho^\gamma/(\gamma-1)\ge0$, vanishing at the vacuum
$\rho=0$. We adopt $e_1$, not $e$, for the remainder of this
subsection.
\end{proposition}
\begin{proof}
By Definition \ref{def:pressure-potential} with $\rho_\star=\mathrm{e}$,
\[
e(\rho)=\rho\int_{\mathrm{e}}^\rho \frac{c^2}\varrho\,d\varrho=c^2\rho(\log\rho-1),
\]
so $e'(\rho)=c^2\log\rho+c^2-c^2=c^2\log\rho=\mathfrak{h}(\rho)$; then
\[
\rho e'(\rho)-e(\rho)=c^2\rho\log\rho-c^2\rho(\log\rho-1)=c^2\rho=p(\rho),
\qquad e''(\rho)=c^2/\rho=p'(\rho)/\rho,
\]
consistently with \eqref{eq:e-identities}. Since $\mathfrak{h}(1)=c^2\log1=0$ and
$e(1)=-c^2$, Lemma \ref{lem:relative-coercive} at $\bar\rho=1$ gives
\[
e_1(\rho)=e(\rho)-e(1)-\mathfrak{h}(1)(\rho-1)=e(\rho)+c^2
=c^2\rho(\log\rho-1)+c^2=c^2(\rho\log\rho-\rho+1),
\]
and, by Lemma \ref{lem:relative-coercive}, $e_1'=\mathfrak{h}$: the enthalpy, and
hence every statement of \S\ref{sec:hydrostatic-integral}, Proposition
\ref{rem:vacuum}, and Remark \ref{rem:gamma-2} above phrased through
$\mathfrak{h},\mathfrak{h}',\mathfrak{h}''$ alone, is unaffected by the choice between $e$ and $e_1$.
\end{proof}

\begin{proposition}[Isothermal hydrostatic profile]
\label{prop:isothermal-profile}
Every regular, strictly positive density satisfying
\eqref{eq:intro-hydrostatic} under \eqref{eq:isothermal} satisfies
\begin{equation}
\label{eq:isothermal-profile}
\rho(t,x,y)=\exp\!\left(\frac{\mathfrak{M}(t,x)-\Phi(y)}{c^2}\right)
\end{equation}
everywhere on $\Omega$, for every value of $\mathfrak{M}(t,x)-\Phi(y)\in\mathbb R$:
unlike the polytropic profile \eqref{eq:polytropic-profile}, this
expression is never singular and $\rho>0$ identically, since
$\mathfrak{h}(\rho)=c^2\log\rho$ is a bijection of $(0,\infty)$ onto $\mathbb R$.
For a regular, strictly positive isothermal hydrostatic profile
represented by a finite Montgomery potential, vacuum is therefore
structurally excluded at any finite vertical level. This says nothing
about weak solutions: whole-column vacuum remains admissible in the
distributional formulation \eqref{eq:HB-weak} of Definition
\ref{def:weak-solution}, as Proposition
\ref{prop:wdj-hydrostatic-transfer} shows directly, with no appeal to
the pointwise formula above.
\end{proposition}
\begin{proof}
By \eqref{eq:first-integral}, $c^2\log\rho=\mathfrak{M}-\Phi(y)$; exponentiate.
\end{proof}

By the same elementary computation, since $p'(\rho)=c^2$ and
$\nabla_x\sqrt\rho=\tfrac12\rho^{-1/2}\nabla_x\rho$, so that
$|\nabla_x\sqrt\rho|^2=\tfrac14\rho^{-1}|\nabla_x\rho|^2$,
\begin{equation}
\label{eq:isothermal-pressure-dissipation}
\frac{p'(\rho)}\rho|\nabla_x\rho|^2=4c^2\bigl|\nabla_x\sqrt\rho\bigr|^2:
\end{equation}
this is exactly the formal $\gamma=1$ instance of
\eqref{eq:pressure-dissipation}, and is the rewriting used in the
Bresch--Desjardins entropy for the isothermal compressible
Navier--Stokes and primitive equations.

\begin{proposition}[Isothermal structural functions]
\label{prop:isothermal-Lambda-N}
For \eqref{eq:isothermal},
\begin{equation}
\label{eq:isothermal-Lambda-N}
\Lambda(\rho)=-\frac g{c^2}\quad(\text{a nonzero constant}),\qquad
\mathcal{N}(\rho)\equiv0.
\end{equation}
\end{proposition}
\begin{proof}
By Proposition \ref{prop:isothermal-potential}, $\mathfrak{h}'(\rho)=c^2/\rho$,
$\mathfrak{h}''(\rho)=-c^2/\rho^2$, so
$\Lambda(\rho)=g(-c^2/\rho^2)/(c^2/\rho)^2=-g/c^2$. For $\mathcal{N}$,
$\Lambda(\rho)/\rho+g/(\mathfrak{h}'(\rho)\rho^2)=-g/(c^2\rho)+g/(c^2\rho)=0$.
\end{proof}

\begin{remark}[The other boundary of the dichotomy]
\label{rem:isothermal-dichotomy}
Substituting $\gamma=1$ into \eqref{eq:polytropic-Lambda-N} gives
\[
\Lambda=\frac{g(1-2)}a\rho^0=-\frac ga,
\qquad
\mathcal{N}=\frac{g(1-1)}a\rho^{-1}=0,
\]
which agree with \eqref{eq:isothermal-Lambda-N} at $a=c^2$: unlike the
enthalpy formula \eqref{eq:polytropic-potential}, which diverges at
$\gamma=1$ (Remark \ref{rem:gamma-2}), $\Lambda$ and $\mathcal{N}$ extend
smoothly across it. Proposition \ref{prop:isothermal-Lambda-N} is thus
both an independent computation and a consistency check. This confirms
the dichotomy of Remark \ref{rem:gamma-2}: $\gamma=2$ annihilates
$\Lambda$, $\gamma=1$ annihilates $\mathcal{N}$, and no member of the polytropic
family annihilates both -- a special case of Theorem
\ref{thm:Lambda-N-classification}, valid for every barotropic law.
Here $\mathcal{N}$ vanishes identically, the strongest form of the commutation
identity of Lemma \ref{lem:N}, so mechanism (a) of
\eqref{eq:kappa-entropy} vanishes for every choice of $\nu_2$, with no
exponent constraint needed; mechanism (b), whose $\nabla_xv$-term does
not involve $\mathcal{N}$, retains it unconditionally. This settles the question
of \emph{why} the multiplicative Ersoy--Ngom--Sy factorization
underlying \cite{WangDouJiu} is available precisely at this law:
$\mathcal{N}\equiv0$ is equivalent, by Proposition
\ref{lem:Lambda-N-geometric}, to $\log\rho$ being affine in $y$, i.e.\
$\rho(t,x,y)=\xi(t,x)\,e^{-y/H_s}$ -- exactly the factorization the
change of variables uses -- and Theorem
\ref{thm:Lambda-N-classification} shows this occurs only at the
isothermal law, among every barotropic pressure law. It is this class
of multiplicative, exponential-profile factorizations that is
classified here, not every conceivable change of variables;
\S\ref{subsec:wdj-transfer} draws the consequence for the
Wang--Dou--Jiu existence theorem.
\end{remark}

\subsection{Application to the Wang--Dou--Jiu isothermal model}
\label{subsec:wdj-transfer}

In this subsection only, we specialize the preceding structural
framework to the Wang--Dou--Jiu isothermal model and therefore restate
explicitly, below and in Assumption \ref{ass:wdj-vertical-viscosity},
the modified viscosity and damping under which this application is
carried out: an additional drag term (\eqref{eq:CPE-damped}) and a
vertical viscosity $\nu_2=\nu_2(\rho,y)$ depending explicitly on the
vertical coordinate rather than on $\rho$ alone. Theorems
\ref{thm:total-energy} and \ref{thm:kappa-entropy} themselves are
unaffected (Remark \ref{rem:wdj-viscosity-divergence-form}); it is only
this one subsection's specific choices that depart from the standing
Assumptions \ref{ass:pressure}--\ref{ass:viscosities} of
Section~\ref{sec:model}. This subsection carries to its conclusion the
chain running from the hydrostatic balance through the Montgomery
potential and the isothermal criterion $\mathcal{N}\equiv0$ to the
Ersoy--Ngom--Sy transform, the transformed Wang--Dou--Jiu weak
solution, and back to the original variables. At the isothermal law,
the vertical constancy of $\mathfrak M$ forces $\mathcal{N}\equiv0$ (Proposition
\ref{prop:isothermal-Lambda-N}), hence the vertical density profile to
be a pure exponential in $y$ (Proposition \ref{lem:Lambda-N-geometric}) --
exactly the factorization the Ersoy--Ngom--Sy change of variables
performs, and which Proposition \ref{prop:ens-montgomery} identifies
with $\mathfrak M$ itself, written multiplicatively (Remark
\ref{rem:classification-geometric}). Built on that change of
variables, the existence theorem of Wang, Dou and Jiu
\cite{WangDouJiu} states the corresponding result for the original
variables -- the last link of the chain: the return from the reduced
system back to the original hydrostatic equations -- while the
accompanying argument focuses primarily on the transfer of the
relevant regularity bounds. The contribution of Theorem
\ref{thm:wdj-transfer} below is to make this pull-back explicit at the
level of the distributional equations, term by term, via the transfer
identities generated directly by the change of variables (Lemma
\ref{lem:wdj-transfer-identities}), together with the corresponding
regularity and energy properties.

This transfer is not a formality. The change of variables combines a
nonlinear vertical reparametrization $z=\zeta(y)$, depending only on
$y$, $g$, and $c^2$, a multiplicative density renormalization
$\rho=\xi\mathcal J$, and a specifically $y$-dependent vertical
viscosity profile \eqref{eq:wdj-viscosity-choice} -- it is the density
renormalization, not the coordinate change itself, that depends on
$\rho$ -- so that mass, momentum, pressure,
damping, and both viscous fluxes each carry, before simplification, a
different power of the Jacobian $\mathcal J$. The content of Theorem
\ref{thm:wdj-transfer} is that every one of these powers cancels
exactly, term by term, with no residual correction anywhere in the weak
formulation -- a fact that has to be checked for each term separately
and holds for none of them for an arbitrary $y$-dependent viscosity
choice. This is what makes the theorem, and not only the constituent
identities of Lemma \ref{lem:wdj-transfer-identities} it is built from,
the object that closes this chain: a
complete, self-contained distributional statement, proved independently
of the approximation scheme of \cite{WangDouJiu}, rather than a
restatement of their results in different notation. Throughout this
subsection
$p(\rho)=c^2\rho$, with
$g,c^2>0$ kept general -- \cite{WangDouJiu} works in the normalized
units $g=c^2=1$ (so that in particular $H_s=1$), to which every
displayed formula below reduces upon setting $g=c^2=1$. We also adopt,
only in this subsection, the damped momentum equation
\begin{equation}
\label{eq:CPE-damped}
\partial_t(\rho \bm{u})+\operatorname{div}_x(\rho \bm{u}\otimes \bm{u})+\partial_y(\rho v\bm{u})
+\nabla_xp(\rho)+r\rho|\bm{u}|\bm{u}
=2\operatorname{div}_x(\nu_1(\rho)D_x\bm{u})+\partial_y(\nu_2(\rho)\partial_y\bm{u}),
\end{equation}
$r>0$ a constant, in place of the momentum equation of \eqref{eq:CPE} --
exactly system (1.1) of \cite{WangDouJiu}. This is a genuine (if
elementary) extension of the model of Section \ref{sec:model}, flagged
explicitly: it is not covered by Assumptions
\ref{ass:pressure}--\ref{ass:viscosities} as stated.

\begin{remark}[The damping term is easily absorbed]
\label{rem:wdj-damping-extension}
Testing \eqref{eq:CPE-damped} against $\bm{u}$ produces, in addition to the
terms of Lemma \ref{lem:kinetic}, the single extra term
$\int_\Omega r\rho|\bm{u}|\bm{u}\cdot \bm{u}\,dx\,dy=\int_\Omega r\rho|\bm{u}|^3\,dx\,dy\ge0$; no other term
of the proof of Lemma \ref{lem:kinetic}, Proposition
\ref{prop:integrated-balance}, or Theorem \ref{thm:total-energy} is
affected, since none of them involves the momentum equation's inertial
or pressure terms in a way that interacts with an algebraic zeroth-order
term in $\bm{u}$. Consequently Theorem \ref{thm:total-energy} extends to
\eqref{eq:CPE-damped} verbatim, except for the addition of
$r\int_\Omega\rho|\bm{u}|^3\,dx\,dy$ to $\mathcal D(t)$. The same is not verified here
for Theorem \ref{thm:kappa-entropy}: testing \eqref{eq:CPE-damped}
against the augmented velocity $U=\bm{u}+2\kappa\bar\nu_1\nabla_x\log\rho$
produces, for every $\kappa\in(0,1]$, an additional cross term
$\int_\Omega r\rho|\bm{u}|\bm{u}\cdot2\kappa\bar\nu_1\nabla_x\log\rho\,dx\,dy$ whose sign and
integrability are not examined in this note; see Remark
\ref{rem:wdj-cross-term-favorable}.
\end{remark}

Recall the change of variables of Ersoy, Ngom and Sy \cite{ErsoyNgomSy},
used by \cite{WangDouJiu}. It rests on the fact, recorded in Proposition
\ref{lem:Lambda-N-geometric} and specialized to the isothermal law by
Proposition \ref{prop:isothermal-Lambda-N}, that the vertical density
profile is a pure exponential in $y$; writing that exponential as
$\rho(t,x,y)=\rho(t,x,0)\,e^{-y/H_s}$ defines the stratification height
\begin{equation}
\label{eq:stratification-height}
H_s:=\frac{c^2}g,
\end{equation}
the vertical length scale over which the isothermal density decays by a
factor $\mathrm{e}$. Set
\begin{equation}
\label{eq:jacobian-def}
\mathcal J(y):=e^{-y/H_s},
\end{equation}
and $\zeta(y):=H_s\bigl(1-\mathcal J(y)\bigr)=H_s(1-e^{-y/H_s})$, a smooth
increasing bijection of $(0,H)$ onto $(0,h)$, $h:=H_s(1-e^{-H/H_s})$;
set $z:=\zeta(y)$. Since $\rho(t,x,y)=\rho(t,x,0)\mathcal J(y)$,
\[
\zeta(y)=\int_0^y\mathcal J(y')\,dy'=\Bigl(\int_0^y\rho(t,x,y')\,dy'\Bigr)\big/\rho(t,x,0):
\]
up to the local surface density, $\zeta$ is exactly the column mass
between $0$ and $y$ -- the same quantity underlying the mass- and
pressure-based vertical coordinates of geophysical fluid dynamics.
Then $\zeta'(y)=\mathcal J(y)$, and along $z=\zeta(y)$,
$\mathcal J=1-z/H_s$: $\mathcal J$ is at once the Jacobian of the change
of variables and an explicit affine function of $z$. Define
\begin{equation}
\label{eq:wdj-ens-transform}
\xi(t,x,z):=\frac{\rho(t,x,y)}{\mathcal J(y)},\qquad
w(t,x,z):=\mathcal J(y)\,v(t,x,y),\qquad z=\zeta(y),
\end{equation}
(we keep \cite{WangDouJiu}'s own notation $w$ for the transformed
vertical velocity; the Bresch--Desjardins augmented velocity of Section
\ref{sec:Lambda-N} is denoted $U$ throughout this note, so that no
clash arises here).

\begin{assumption}[Vertical viscosity in the isothermal transfer]
\label{ass:wdj-vertical-viscosity}
In this subsection only, the vertical viscosity is allowed to depend
explicitly on the vertical coordinate, $\nu_2=\nu_2(\rho,y)\ge0$,
rather than on $\rho$ alone as posited by Assumption
\ref{ass:viscosities}. Together with the horizontal viscosity profile
already fixed by Assumption \ref{ass:linear-viscosity}, we use the
specific choice
\begin{equation}
\label{eq:wdj-viscosity-choice}
\nu_1(\rho)=\bar\nu_1\rho,\qquad
\nu_2(\rho,y)=\bar\nu_2\rho\,e^{2y/H_s}=\frac{\bar\nu_2\rho}{\mathcal J(y)^2},
\qquad\bar\nu_1,\bar\nu_2>0\text{ constants}.
\end{equation}
\end{assumption}
Remark \ref{rem:wdj-viscosity-divergence-form} below records, once and
for this local hypothesis, exactly why Theorems \ref{thm:total-energy}
and \ref{thm:kappa-entropy} remain valid verbatim under it.

The variable $\xi$ just defined is not merely compatible with the
Montgomery potential: it \emph{is} the Montgomery potential, in
exponential form.

\begin{proposition}[The Ersoy--Ngom--Sy variable is the exponentiated
Montgomery potential]
\label{prop:ens-montgomery}
Under \eqref{eq:isothermal}, for every regular, strictly positive
density satisfying \eqref{eq:intro-hydrostatic},
\begin{equation}
\label{eq:ens-is-montgomery}
\xi(t,x,z)=\exp\!\left(\frac{\mathfrak M(t,x)}{c^2}\right),
\end{equation}
manifestly independent of $y$ -- equivalently, after the
reparametrization $z=\zeta(y)$, of $z$: this is exactly the identity
$\partial_z\xi=0$ recorded as the third equation of
\eqref{eq:wdj-transformed} below.
\end{proposition}
\begin{proof}
By Proposition \ref{lem:first-integral} and Proposition
\ref{prop:isothermal-potential}, $c^2\log\rho(t,x,y)+\Phi(y)=\mathfrak
M(t,x)$ for every $y$, so
\[
\rho(t,x,y)=\exp\!\left(\frac{\mathfrak M(t,x)-\Phi(y)}{c^2}\right)
=\exp\!\left(\frac{\mathfrak M(t,x)}{c^2}\right)\exp\!\left(-\frac{\Phi(y)}{c^2}\right).
\]
Since
$\Phi(y)=gy$ and $H_s=c^2/g$, $\exp(-\Phi(y)/c^2)=e^{-y/H_s}=\mathcal
J(y)$, so $\rho(t,x,y)=\exp(\mathfrak M(t,x)/c^2)\,\mathcal J(y)$; dividing
by $\mathcal J(y)$ gives \eqref{eq:ens-is-montgomery}, in which the
right side depends on $(t,x)$ alone.
\end{proof}

Definition \eqref{eq:wdj-ens-transform} divides $\rho$ by $\mathcal J(y)$
specifically, not by an arbitrary positive weight fitted to the
isothermal law from outside. For a general smooth $\varphi:(0,H)\to
(0,\infty)$ and every regular hydrostatic density, $\rho/\varphi(y)$ is
independent of $y$ if and only if $\varphi$ is proportional to
$\mathcal J$: since $\rho(t,x,y)=\rho(t,x,0)\,\mathcal J(y)$ by
Proposition \ref{prop:isothermal-profile}, $\rho(t,x,\cdot)/\varphi(\cdot)$
constant in $y$ forces $\mathcal J/\varphi$ constant, i.e.\
$\varphi=\text{const}\cdot\mathcal J$. The Ersoy--Ngom--Sy transform is,
in this precise sense, the unique multiplicative renormalization of the
density compatible with the vertical constancy of $\mathfrak M$, up to
the one remaining multiplicative constant fixed by the normalization
$\mathcal J(0)=1$ in \eqref{eq:jacobian-def} -- not one convenient
change of variables among several freely available, but the only
one-parameter family the hydrostatic structure of
\S\ref{sec:hydrostatic-integral} permits, with $\mathcal J$ itself
singled out by the harmless convention that $\xi$ and $\rho$ agree at
$y=0$.

The vertical constancy of $\mathfrak M$ (Proposition \ref{lem:first-integral})
-- valid, we recall, at every barotropic law, not only the isothermal
one -- is therefore not merely analogous to the defining property
$\partial_z\xi=0$ of the transformed density in
\eqref{eq:wdj-transformed}: at the isothermal law the two are the same
statement, related by the single exponential change of unknown
\eqref{eq:ens-is-montgomery}. The Ersoy--Ngom--Sy reduction does not
discover a new conserved structure; it re-exposes, in coordinates
adapted to it, the one this note's construction produces for every
pressure law. Under
\eqref{eq:wdj-ens-transform}--\eqref{eq:wdj-viscosity-choice},
\eqref{eq:CPE-damped} formally becomes, on $\Omega_z:=\mathbb
T^2\times(0,h)$,
\begin{equation}
\label{eq:wdj-transformed}
\left\{
\begin{aligned}
&\partial_t\xi+\operatorname{div}_x(\xi \bm{u})+\partial_z(\xi\,w)=0,\\
&\partial_t(\xi \bm{u})+\operatorname{div}_x(\xi \bm{u}\otimes \bm{u})+\partial_z(\xi \bm{u}w)+c^2\nabla_x\xi+r\xi|\bm{u}|\bm{u}
=2\bar\nu_1\operatorname{div}_x(\xi D_x\bm{u})+\bar\nu_2\partial_z(\xi\partial_z\bm{u}),\\
&\partial_z\xi=0,
\end{aligned}
\right.
\end{equation}
with periodicity in $x$, $w|_{z=0,h}=0$, $\partial_z\bm{u}|_{z=0,h}=0$ --
this is system (2.1) of \cite{WangDouJiu}. Wang, Dou and Jiu's
Definition 2.1 declares $(\xi,\bm{u},w)$ a weak solution of
\eqref{eq:wdj-transformed} if the initial data are attained (their
(2.3)) and the regularity class of their (2.5) holds -- namely
$\xi\in L^\infty(0,T;L^3(\Omega_z))$, $\sqrt\xi \bm{u}\in
L^\infty(0,T;L^2(\Omega_z))$, $\sqrt\xi\in L^\infty(0,T;H^1(\Omega_z))$,
and $\sqrt\xi\nabla_x\bm{u},\sqrt\xi\partial_z\bm{u},\sqrt\xi w,\sqrt\xi\partial_zw\in
L^2((0,T)\times\Omega_z)$ -- together with $\xi^{1/3}\bm{u}\in
L^2((0,T)\times\Omega_z)$. The equations
\eqref{eq:wdj-transformed} must then hold in the ordinary
distributional sense on $(0,T)\times\Omega_z$ (their point (2) of
Definition 2.1), tested directly against $\psi\in
C_c^\infty([0,T)\times\overline{\Omega_z})$: no distinguished pairing
device beyond this is used. This is the direct analogue, for
\eqref{eq:wdj-transformed}, of Definition \ref{def:weak-solution}
above -- indeed, combined with $\rho=\xi\mathcal J$ and $0<\mathcal
J\le1$, their bound $\sqrt\xi\nabla_x\bm{u}\in L^2$ on the full gradient, once
combined with the identity $\nu_1(\rho)D_x\bm{u}=\bar\nu_1\xi\mathcal J\,D_x\bm{u}$
of \eqref{eq:wdj-derivative-identities} below, controls -- it is in fact
strictly stronger than -- the bound $\sqrt{\nu_1(\rho)}D_x\bm{u}\in L^2$ on
the symmetric part alone required by Definition
\ref{def:weak-solution}(ii) for the linear viscosity of Assumption
\ref{ass:linear-viscosity}. Their Theorem 2.1 constructs such a weak
solution, for admissible initial data (their (2.4), including
$\xi_0\ln\xi_0-\xi_0+1\in L^1(\Omega)$), satisfying an energy
inequality (their (2.6)) and a Bresch--Desjardins entropy inequality
(their (2.11)). Both are established through their five-parameter
approximation scheme (their Secs.\ III--VI) and transferred to the
limit solution by weak lower semicontinuity; we do not track these
through their intermediate equation numbers here. Their Theorem 2.2
then states that, under the argument of Theorem 2.1, system (1.1)--(1.3) -- our
\eqref{eq:CPE-damped} with the extra damping term of Remark
\ref{rem:wdj-damping-extension} above -- has a weak solution
$(\rho,\bm{u},v)$ in the sense of Definition 2.1 with $(\xi,\bm{u},w)$ replaced
by $(\rho,\bm{u},v)$. This is the pull-back we make explicit below.

\begin{lemma}[Transfer identities and their derivatives]
\label{lem:wdj-transfer-identities}
Let $(\xi,\bm{u},w)$ solve \eqref{eq:wdj-transformed} on $\Omega_z$ and
define $(\rho,\bm{u},v)$ on $\Omega$ by inverting
\eqref{eq:wdj-ens-transform}:
\begin{equation}
\label{eq:wdj-inverse-transform}
\rho(t,x,y):=\xi(t,x,\zeta(y))\,\mathcal J(y),\qquad
v(t,x,y):=\frac{w(t,x,\zeta(y))}{\mathcal J(y)},
\end{equation}
$\bm{u}$ unchanged. Then, at every corresponding point $(y,z=\zeta(y))$,
\begin{align}
\rho\,dy&=\xi\,dz,
\label{eq:wdj-measure-identity}\\
\rho v&=\xi\,w,
\label{eq:wdj-flux-identity}\\
\partial_y\bm{u}&=\mathcal J\,\partial_z\bm{u},
\qquad
\nu_1(\rho)D_x\bm{u}=\bar\nu_1\xi\mathcal J\,D_x\bm{u},
\qquad
\nu_2(\rho)\,\partial_y\bm{u}=\bar\nu_2\xi\,\partial_z\bm{u},
\label{eq:wdj-derivative-identities}\\
\partial_yv&=\frac1{H_s\mathcal J}\,w+\partial_zw,
\qquad
\nabla_x\log\rho=\nabla_x\log\xi\quad\text{on }\{\xi>0\}.
\label{eq:wdj-v-identity}
\end{align}
The logarithmic identity in \eqref{eq:wdj-v-identity} is understood
pointwise on the positivity set $\{\xi>0\}$ only, where $\log\xi$ is
defined; it plays no role on $\{\xi=0\}$ and is used below solely in
regular, strictly positive computations (Proposition
\ref{prop:wdj-mechanism-reduction}, Remark
\ref{rem:wdj-cross-term-favorable}), never in the weak transfer
argument of Theorem \ref{thm:wdj-transfer} or Corollary
\ref{cor:wdj-energy-transfer}, which is vacuum-safe and does not
involve $\nabla_x\xi/\xi$ in any form.
\end{lemma}
\begin{proof}
Since $\zeta'(y)=\mathcal J$, $dy=dz/\mathcal J$; combined with
$\rho=\xi\mathcal J$ this gives \eqref{eq:wdj-measure-identity}:
\[
\rho\,dy=\xi\mathcal J\cdot\frac{dz}{\mathcal J}=\xi\,dz.
\]
For \eqref{eq:wdj-flux-identity}, using $v=w/\mathcal J$,
\[
\rho v=\xi\mathcal J\cdot\frac w{\mathcal J}=\xi\,w.
\]
For \eqref{eq:wdj-derivative-identities}: by the chain rule,
$\partial_y\bm{u}=\partial_z\bm{u}\cdot\zeta'(y)=\mathcal J\,\partial_z\bm{u}$; directly,
$\nu_1(\rho)D_x\bm{u}=\bar\nu_1\rho D_x\bm{u}=\bar\nu_1\xi\mathcal J\,D_x\bm{u}$; and,
using $\nu_2(\rho)=\bar\nu_2\rho/\mathcal J^2$,
\[
\nu_2(\rho)\partial_y\bm{u}=\frac{\bar\nu_2\rho}{\mathcal J^2}\cdot\mathcal J\,\partial_z\bm{u}
=\frac{\bar\nu_2\rho}{\mathcal J}\,\partial_z\bm{u}
=\bar\nu_2\xi\,\partial_z\bm{u},
\]
using $\rho=\xi\mathcal J$ in the last step. For \eqref{eq:wdj-v-identity}:
differentiating $v=w(\cdot,\zeta(y))/\mathcal J(y)$ in $y$ and using
$\mathcal J'(y)=-\mathcal J/H_s$,
\[
\partial_yv=\frac{\partial_zw\cdot\zeta'(y)}{\mathcal J}-\frac{\mathcal J'}{\mathcal J^2}\,w
=\partial_zw+\frac1{H_s\mathcal J}\,w;
\]
and, on $\{\xi>0\}$, where $\log\xi$ is defined, $\log\rho=\log\xi+\log\mathcal J$,
with $\log\mathcal J$ independent
of $x$, so $\nabla_x\log\rho=\nabla_x\log\xi$ there.
\end{proof}

\begin{theorem}[Rigorous transfer to the original hydrostatic variables]
\label{thm:wdj-transfer}
Let $(\xi,\bm{u},w)$ be a weak solution of \eqref{eq:wdj-transformed} on
$(0,T)\times\Omega_z$ in the sense of \cite[Def.\ 2.1]{WangDouJiu}, and
let $(\rho,\bm{u},v)$ be given by \eqref{eq:wdj-inverse-transform}. Then
$(\rho,\bm{u},v)$ satisfies the continuity and momentum equations of
\eqref{eq:CPE-damped} in $\mathcal D'((0,T)\times\Omega)$, in the exact
sense of Definition \ref{def:weak-solution}(iii)--(iv) extended by the
damping term of Remark \ref{rem:wdj-damping-extension}, with initial
data $\rho_0=\xi_0\mathcal J$, $\tilde m_0=m_0^{\mathrm{WDJ}}\mathcal J$, where
$\tilde m_0$ denotes the original-variable initial momentum and
$m_0^{\mathrm{WDJ}}$ the transformed one, related exactly as in
\cite[eq.\ (2.3)]{WangDouJiu}. Proposition
\ref{prop:wdj-hydrostatic-transfer} below shows that $(\rho,\bm{u},v)$
in fact also satisfies Definition \ref{def:weak-solution}(i), and
Proposition \ref{prop:wdj-regularity-transfer} below shows that it
satisfies the regularity class (ii); combined with the present
theorem's (iii)--(iv), the three results together identify
$(\rho,\bm{u},v)$ as a complete weak solution in the sense of
Definition \ref{def:weak-solution}.
\end{theorem}
\begin{proof}
Since $\zeta:(0,H)\to(0,h)$ is a smooth increasing diffeomorphism between
bounded open intervals extending smoothly to the closures -- in
particular sending $y=0,H$ to $z=0,h$, so that the boundary conditions
of Section \ref{sec:model} and of \eqref{eq:wdj-transformed} are
matched by this same bijection, each encoded, as in Definition
\ref{def:weak-solution} and in \cite[Def.\ 2.1]{WangDouJiu}
respectively, through the test-function space reaching the closure
rather than through a separate boundary term -- composition
with $(\mathrm{id}_t,\mathrm{id}_x,\zeta)$ is a bijection between test
functions $\varphi\in C_c^\infty([0,T)\times\overline\Omega)$, periodic
in $x$, and $\tilde\varphi:=\varphi(\cdot,\cdot,\zeta^{-1}(\cdot))\in
C_c^\infty([0,T)\times\overline{\Omega_z})$, periodic in $x$, and
likewise, componentwise, between the vector-valued test functions
$\psi\in C_c^\infty([0,T)\times\overline\Omega;\mathbb R^2)$ of
\eqref{eq:weak-momentum} and
$\tilde\psi:=\psi(\cdot,\cdot,\zeta^{-1}(\cdot))$; it therefore suffices
to check that each term of \eqref{eq:weak-continuity}, written for
$(\rho,\bm{u},v)$ against $\varphi$, equals the corresponding term of the
weak formulation of \eqref{eq:wdj-transformed} against $\tilde\varphi$,
and that each term of \eqref{eq:weak-momentum} (extended by the damping
term), against $\psi$, equals the corresponding term against
$\tilde\psi$, after integrating in $y$ against $dy=dz/\mathcal J$
(Fubini, together with $t,x$ being untouched by the transform). This is
a direct consequence of Lemma \ref{lem:wdj-transfer-identities}, term by
term:
\begin{itemize}
\item Mass and its flux: directly by \eqref{eq:wdj-measure-identity},
\[
\rho\,\partial_t\varphi\,dy=\xi\,\partial_t\tilde\varphi\,dz,
\qquad
\rho \bm{u}\cdot\nabla_x\varphi\,dy=\xi \bm{u}\cdot\nabla_x\tilde\varphi\,dz,
\]
using $\partial_t\varphi(\cdot,\cdot,y)=\partial_t\tilde\varphi(\cdot,\cdot,z)$
and $\nabla_x\varphi=\nabla_x\tilde\varphi$ at corresponding points,
since $t,x$-derivatives commute with the purely $y$-dependent change of
variables.
\item Vertical mass flux: since $\partial_y\varphi=\mathcal J\,\partial_z\tilde\varphi$
by the same computation as \eqref{eq:wdj-derivative-identities},
\eqref{eq:wdj-flux-identity} gives
\[
\rho v\,\partial_y\varphi\,dy=(\xi\,w)\cdot\mathcal J\,\partial_z\tilde\varphi\cdot
dz/\mathcal J=\xi\,w\,\partial_z\tilde\varphi\,dz.
\]
\item Momentum time derivative and horizontal convection: directly by
\eqref{eq:wdj-measure-identity},
\[
\rho \bm{u}\,\partial_t\psi\,dy=\xi \bm{u}\,\partial_t\tilde\psi\,dz,
\qquad
\rho \bm{u}\otimes \bm{u}:\nabla_x\psi\,dy=\xi \bm{u}\otimes \bm{u}:\nabla_x\tilde\psi\,dz.
\]
\item Vertical momentum convection: as above,
\[
\rho v\bm{u}\cdot\partial_y\psi\,dy
=(\xi\,w)\bm{u}\cdot\mathcal J\,\partial_z\tilde\psi\cdot dz/\mathcal J
=\xi \bm{u}w\cdot\partial_z\tilde\psi\,dz.
\]
\item Pressure: since $p(\rho)=c^2\rho=c^2\xi\mathcal J$, with $\mathcal
J$ independent of $x$, in the divergence-tested form of
\eqref{eq:weak-momentum} in which the pressure term of the weak
formulation is stated -- the only form used, since $\xi$ in the WDJ
regularity class of Definition \ref{def:weak-solution} carries no
control on $\nabla_x\xi$ --
\[
p(\rho)\operatorname{div}_x\psi\,dy
=c^2\xi\mathcal J\,\operatorname{div}_x\psi\cdot dz/\mathcal J
=c^2\xi\operatorname{div}_x\psi\,dz.
\]
\item Damping:
\[
\rho|\bm{u}|\bm{u}\cdot\psi\,dy=\xi\mathcal J\,|\bm{u}|\bm{u}\cdot\psi\cdot
dz/\mathcal J=\xi|\bm{u}|\bm{u}\cdot\psi\,dz.
\]
\item Horizontal viscosity: by \eqref{eq:wdj-derivative-identities},
\[
\nu_1(\rho)D_x\bm{u}:D_x\psi\,dy=\bar\nu_1\xi\mathcal J\,D_x\bm{u}:D_x\tilde\psi\cdot
dz/\mathcal J=\bar\nu_1\xi D_x\bm{u}:D_x\tilde\psi\,dz.
\]
\item Vertical viscosity: by \eqref{eq:wdj-derivative-identities},
\[
\nu_2(\rho)\partial_y\bm{u}\cdot\partial_y\psi\,dy
=\bar\nu_2\xi\partial_z\bm{u}\cdot\mathcal J\,\partial_z\tilde\psi\cdot dz/\mathcal J
=\bar\nu_2\xi\partial_z\bm{u}\cdot\partial_z\tilde\psi\,dz.
\]
\end{itemize}
Every term of \eqref{eq:weak-continuity}--\eqref{eq:weak-momentum}
(extended by the damping term) and of the corresponding weak
formulation of \eqref{eq:wdj-transformed} has now been matched termwise,
with no residual factor of $\mathcal J$: the Jacobian $dy=dz/\mathcal J$
is, in every single term, exactly cancelled by the explicit $\mathcal J$
carried by $\rho=\xi\mathcal J$, $\rho v=\xi\,w$, or by the specific
$y$-dependence of $\nu_2$ in \eqref{eq:wdj-viscosity-choice}, which is
seen here to be the one choice making the vertical-viscosity term
transform without correction. Integrating in $t,x$ and invoking Fubini
converts the pointwise-in-$(t,x)$ matching above into the equality of
the two full weak formulations; the initial-data term transfers
identically by the same computation as the momentum time-derivative
term above (with $\psi(0,\cdot)$ in place of $\partial_t\psi$),
giving $\tilde m_0=\rho_0\bm{u}_0=\xi_0\mathcal J\,\bm{u}_0=m_0^{\mathrm{WDJ}}\mathcal J$,
consistently with \cite[eq.\ (2.3)]{WangDouJiu}.
\end{proof}

\begin{remark}[Why the local hypothesis does not affect the preceding
results]
\label{rem:wdj-viscosity-divergence-form}
The proofs of the mechanical-energy identity (Theorem
\ref{thm:total-energy}) and of the augmented-velocity identity
(Theorem \ref{thm:kappa-entropy}) use the vertical viscosity only
through the divergence-form flux $\partial_y(\nu_2(\rho)\partial_y\bm{u})$,
tested against $\bm{u}$ or $U$ respectively: they integrate this flux by
parts once and never differentiate the coefficient $\nu_2$ itself,
whether with respect to $\rho$ or to $y$. Consequently both arguments
remain valid verbatim for any nonnegative coefficient
$\nu_2=\nu_2(\rho,y)$ possessing the regularity and integrability
required at the point of use, and in particular for the local
hypothesis of Assumption \ref{ass:wdj-vertical-viscosity} and its
specific choice \eqref{eq:wdj-viscosity-choice}, used throughout this
subsection and in Proposition \ref{prop:wdj-mechanism-reduction}
below. This is recorded here once, rather than qualified at each
subsequent use.
\end{remark}

\begin{proposition}[The hydrostatic relation transfers as well]
\label{prop:wdj-hydrostatic-transfer}
Under the hypotheses of Theorem \ref{thm:wdj-transfer}, $(\rho,\bm{u},v)$
given by \eqref{eq:wdj-inverse-transform} also satisfies Definition
\ref{def:weak-solution}(i): \eqref{eq:HB-weak} holds on
$(0,T)\times\Omega$. Together with Theorem
\ref{thm:wdj-transfer}(iii)--(iv) and, below, Proposition
\ref{prop:wdj-regularity-transfer}(ii), this identifies
$(\rho,\bm{u},v)$ as a complete weak solution of
\eqref{eq:CPE-damped} in the sense of Definition
\ref{def:weak-solution}, including on any vacuum region $\{\xi=0\}$: no
non-degeneracy of $\xi$ beyond membership in the WDJ regularity class
is used in this verification of (i).
\end{proposition}
\begin{proof}
By \cite[Def.\ 2.1]{WangDouJiu}, the third equation of
\eqref{eq:wdj-transformed}, $\partial_z\xi=0$, holds in
$\mathcal D'((0,T)\times\Omega_z)$; since it involves no $t$- or
$x$-derivative, the same separable-test-function argument as in the
proof of Lemma \ref{lem:HB-weak-pointwise} reduces it, for a.e.\
$(t,x)$, to
\begin{equation}
\label{eq:wdj-xi-fiber}
\int_0^h\xi(z)\,\psi'(z)\,dz=0\qquad\text{for every }\psi\in C_c^\infty(0,h).
\end{equation}
Fix such $(t,x)$, let $\chi\in C_c^\infty(0,H)$, and set
$\tilde\chi(z):=\chi(\zeta^{-1}(z))\in C_c^\infty(0,h)$, using that
$\zeta$ extends to a smooth diffeomorphism of the closures
(\S\ref{subsec:wdj-transfer}). By \eqref{eq:wdj-inverse-transform},
$\rho(y)=\xi(\zeta(y))\mathcal J(y)$, and, since
$\partial_y\chi=\mathcal J\cdot(\tilde\chi'\circ\zeta)$ (the same
identity as \eqref{eq:wdj-derivative-identities}, applied to $\chi$)
and $dy=dz/\mathcal J$,
\[
\int_0^H\rho\,\partial_y\chi\,dy
=\int_0^H\xi(\zeta(y))\,\mathcal J(y)^2\,\tilde\chi'(\zeta(y))\,dy
=\int_0^h\xi(z)\,\mathfrak J(z)\,\tilde\chi'(z)\,dz,
\qquad
\mathfrak J(z):=\mathcal J(\zeta^{-1}(z))=1-\frac z{H_s},
\]
the last equality by the change of variable $y\mapsto z$ and the
identity $\mathcal J=1-z/H_s$ along $z=\zeta(y)$ recorded in
\S\ref{subsec:wdj-transfer}. Since $\mathfrak J\tilde\chi\in
C_c^\infty(0,h)$, testing \eqref{eq:wdj-xi-fiber} against it and
expanding by the product rule gives
\[
0=\int_0^h\xi\,\partial_z(\mathfrak J\tilde\chi)\,dz
=\int_0^h\xi\,\mathfrak J\,\tilde\chi'\,dz+\int_0^h\xi\,\mathfrak J'\,\tilde\chi\,dz,
\]
and $\mathfrak J'\equiv-1/H_s$ is constant (immediate from
$\mathfrak J(z)=1-z/H_s$), so
\[
\int_0^H\rho\,\partial_y\chi\,dy
=\int_0^h\xi\,\mathfrak J\,\tilde\chi'\,dz
=-\int_0^h\xi\,\mathfrak J'\,\tilde\chi\,dz
=\frac1{H_s}\int_0^h\xi\,\tilde\chi\,dz
=\frac1{H_s}\int_0^H\rho\,\chi\,dy,
\]
the last equality by \eqref{eq:wdj-measure-identity} and
$\tilde\chi(\zeta(y))=\chi(y)$. Multiplying by $c^2=p'(\rho)$ under
\eqref{eq:isothermal} and using $H_s=c^2/g$ gives exactly the fiber
relation \eqref{eq:HB-weak-fiber}:
$\int_0^Hp(\rho)\,\partial_y\chi\,dy=g\int_0^H\rho\,\chi\,dy$.
Integrating in $x$ against $\varphi\in C_c^\infty(\mathbb T^2)$ --
exactly as in the converse direction of the proof of Lemma
\ref{lem:HB-weak-pointwise} -- recovers \eqref{eq:HB-weak} on $\Omega$
for a.e.\ $t$, and integrating further in $t$ extends it to
$(0,T)\times\Omega$.
\end{proof}

\begin{remark}[No non-degeneracy is needed here]
\label{rem:wdj-hydrostatic-transfer-vacuum}
The identity $\mathfrak J'\equiv-1/H_s$ is what lets
\eqref{eq:HB-weak-fiber} be extracted from $\partial_z\xi=0$ by a single
integration by parts, valid for $\xi$ merely in $L^1_{loc}$, vanishing
or not: no non-degeneracy hypothesis on $\xi$ is used. This is the
vacuum-robustness Definition \ref{def:weak-solution}(i) was built for
(Remark \ref{rem:hydrostatic-weak}), and Proposition
\ref{prop:wdj-hydrostatic-transfer} shows it is not vacuous even in the
one setting -- the isothermal law -- where the \emph{pointwise}
Montgomery formula of Proposition \ref{prop:isothermal-profile}
structurally excludes vacuum. The two statements do not contradict one
another: Proposition \ref{prop:isothermal-profile} concerns regular,
strictly positive solutions, for which the pointwise formula and
\eqref{eq:HB-weak} coincide by Lemma \ref{lem:HB-weak-pointwise}; on a
genuine vacuum set, only \eqref{eq:HB-weak} remains meaningful, and it
is what Proposition \ref{prop:wdj-hydrostatic-transfer} verifies
directly, with no appeal to the pointwise formula at all.
\end{remark}

\begin{proposition}[Transfer of the WDJ regularity class]
\label{prop:wdj-regularity-transfer}
Under the hypotheses of Theorem \ref{thm:wdj-transfer}, $(\rho,\bm{u},v)$
given by \eqref{eq:wdj-inverse-transform} also satisfies Definition
\ref{def:weak-solution}(ii): all eight memberships
\[
\rho\in L^\infty(0,T;L^1(\Omega)),\qquad
e(\rho)\in L^\infty(0,T;L^1(\Omega)),\qquad
p(\rho)\in L^1((0,T)\times\Omega),
\]
\[
\sqrt\rho\,\bm{u}\in L^\infty(0,T;L^2(\Omega)),\qquad
\sqrt\rho\,v\in L^2((0,T)\times\Omega),
\]
\[
\nu_1(\rho),\nu_2(\rho)\in L^1((0,T)\times\Omega),\qquad
\sqrt{\nu_1(\rho)}\,D_x\bm{u},\ \sqrt{\nu_2(\rho)}\,\partial_y\bm{u}\in L^2((0,T)\times\Omega)
\]
hold, with no non-degeneracy of $\xi$ used: the argument is valid on any
vacuum region $\{\xi=0\}$. Combined with Theorem
\ref{thm:wdj-transfer}(iii)--(iv) and Proposition
\ref{prop:wdj-hydrostatic-transfer}(i), this completes the verification
that $(\rho,\bm{u},v)$ is a complete weak solution of
\eqref{eq:CPE-damped} in the sense of Definition
\ref{def:weak-solution}.
\end{proposition}
\begin{proof}
Since $\mathcal J(y)=e^{-y/H_s}$ is decreasing on $[0,H]$
(\eqref{eq:jacobian-def}), set
\[
J_*:=\mathcal J(H)=e^{-H/H_s}\in(0,1),\qquad\text{so that}\qquad
J_*\le\mathcal J(y)\le1\quad\text{for every }y\in[0,H].
\]
We use throughout the change of variables $dy=dz/\mathcal J$, which
converts any integral over $\Omega$ of a function $F$ of $\rho$ or
$\bm{u}$, expressed through $\xi,w$ via
\eqref{eq:wdj-inverse-transform}, into
$\int_\Omega F\,dx\,dy=\int_{\Omega_z}(F/\mathcal J)\,dx\,dz$, and the
elementary bound, valid for a universal constant $C_0>0$ and every
$r\ge0$ (with the convention $r\log r:=0$ at $r=0$),
\begin{equation}
\label{eq:rlogr-bound}
r|\log r|\le C_0(1+r^3),
\end{equation}
which holds because $r\mapsto r|\log r|/(1+r^3)$ extends continuously
to $[0,\infty)$ by $0$ at $r=0$ and tends to $0$ as $r\to\infty$, hence
is bounded. Since $\Omega_z$ has finite measure and
$\xi\in L^\infty(0,T;L^3(\Omega_z))$ by \cite[Def.\ 2.1]{WangDouJiu},
\eqref{eq:rlogr-bound} and H\"older's inequality give
\begin{equation}
\label{eq:xi-L1-and-xilogxi-L1}
\xi\in L^\infty(0,T;L^1(\Omega_z)),
\qquad
\xi|\log\xi|\in L^\infty(0,T;L^1(\Omega_z)),
\end{equation}
both uniformly in $t$. We verify the eight memberships in turn.

\emph{Mass.} By the measure identity \eqref{eq:wdj-measure-identity},
$\int_\Omega\rho\,dx\,dy=\int_{\Omega_z}\xi\,dx\,dz$, bounded uniformly
in $t$ by \eqref{eq:xi-L1-and-xilogxi-L1}: $\rho\in
L^\infty(0,T;L^1(\Omega))$.

\emph{Internal energy.} By \eqref{eq:isothermal-potential-family},
$e(\rho)=c^2\rho(\log\rho-1)$. Writing $\rho=\xi\mathcal J$, the
identity used throughout is directly on the product $\rho\log\rho$,
\[
\rho\log\rho=\mathcal J\,\xi\log\xi+\mathcal J\,\xi\log\mathcal J,
\]
valid pointwise everywhere on $(0,H)$, including on $\{\xi=0\}$ under
the convention $r\log r:=0$ at $r=0$, where both sides vanish together
-- unlike the pointwise split $\log\rho=\log\xi+\log\mathcal J$, which
is meaningless on $\{\xi=0\}$ since $\log\xi$ alone is undefined there,
and is not used. So
\[
e(\rho)=c^2\mathcal J(\xi\log\xi-\xi)+c^2\mathcal J\,\xi\log\mathcal J.
\]
Since $0<\mathcal J\le1$ and $-H/H_s\le\log\mathcal J\le0$ on $[0,H]$
(so $|\log\mathcal J|\le H/H_s$),
\[
|e(\rho)|\le c^2\xi|\log\xi|+c^2\xi\Bigl(1+\frac H{H_s}\Bigr).
\]
Hence, using $\int_\Omega F\,dx\,dy=\int_{\Omega_z}(F/\mathcal
J)\,dx\,dz\le J_*^{-1}\int_{\Omega_z}F\,dx\,dz$ for $F\ge0$,
\[
\int_\Omega|e(\rho)|\,dx\,dy
\le J_*^{-1}c^2\int_{\Omega_z}\Bigl[\xi|\log\xi|+\Bigl(1+\frac
H{H_s}\Bigr)\xi\Bigr]\,dx\,dz,
\]
finite and bounded uniformly in $t$ by \eqref{eq:xi-L1-and-xilogxi-L1}:
$e(\rho)\in L^\infty(0,T;L^1(\Omega))$.

\emph{Pressure.} By \eqref{eq:wdj-measure-identity} and
$p(\rho)=c^2\rho$ (\eqref{eq:isothermal}),
$\int_\Omega p(\rho)\,dx\,dy=c^2\int_{\Omega_z}\xi\,dx\,dz$; integrating
in $t\in(0,T)$ and using \eqref{eq:xi-L1-and-xilogxi-L1} gives
$p(\rho)\in L^1((0,T)\times\Omega)$.

\emph{Horizontal kinetic energy.} By
\eqref{eq:wdj-measure-identity}, $\bm{u}$ being unchanged by the
transform,
$\int_\Omega\rho|\bm{u}|^2\,dx\,dy=\int_{\Omega_z}\xi|\bm{u}|^2\,dx\,dz$,
bounded uniformly in $t$ since $\sqrt\xi\,\bm{u}\in
L^\infty(0,T;L^2(\Omega_z))$ by \cite[Def.\ 2.1]{WangDouJiu}: $\sqrt\rho\,\bm{u}\in
L^\infty(0,T;L^2(\Omega))$.

\emph{Vertical kinetic energy.} Since $v=w/\mathcal J$ and
$\rho=\xi\mathcal J$,
\[
\rho v^2\,dy=\xi\mathcal J\cdot\frac{w^2}{\mathcal J^2}\cdot\frac{dz}{\mathcal J}
=\frac{\xi w^2}{\mathcal J^2}\,dz,
\qquad\text{so}\qquad
\int_\Omega\rho|v|^2\,dx\,dy=\int_{\Omega_z}\frac{\xi|w|^2}{\mathcal J^2}\,dx\,dz
\le J_*^{-2}\int_{\Omega_z}\xi|w|^2\,dx\,dz.
\]
Integrating in $t$, this is finite by $\sqrt\xi\,w\in
L^2((0,T)\times\Omega_z)$ (\cite[Def.\ 2.1]{WangDouJiu}): $\sqrt\rho\,v\in
L^2((0,T)\times\Omega)$.

\emph{Viscosity coefficients.} By \eqref{eq:wdj-viscosity-choice},
$\nu_1(\rho)=\bar\nu_1\rho$, so
$\int_\Omega\nu_1(\rho)\,dx\,dy=\bar\nu_1\int_{\Omega_z}\xi\,dx\,dz$,
integrable in $t$ as above: $\nu_1(\rho)\in
L^1((0,T)\times\Omega)$. For $\nu_2(\rho)=\bar\nu_2\rho/\mathcal J^2$,
\[
\int_\Omega\nu_2(\rho)\,dx\,dy
=\int_{\Omega_z}\frac{\nu_2(\rho)}{\mathcal J}\,dx\,dz
=\bar\nu_2\int_{\Omega_z}\frac\rho{\mathcal J^3}\,dx\,dz
=\bar\nu_2\int_{\Omega_z}\frac\xi{\mathcal J^2}\,dx\,dz
\le\bar\nu_2J_*^{-2}\int_{\Omega_z}\xi\,dx\,dz,
\]
using $\rho=\xi\mathcal J$; integrating in $t$ gives $\nu_2(\rho)\in
L^1((0,T)\times\Omega)$.

\emph{Horizontal dissipation.} By
\eqref{eq:wdj-derivative-identities}, $\nu_1(\rho)D_x\bm{u}=\bar\nu_1\xi\mathcal
J\,D_x\bm{u}$, so
\[
\int_\Omega\nu_1(\rho)|D_x\bm{u}|^2\,dx\,dy
=\int_{\Omega_z}\frac{\nu_1(\rho)|D_x\bm{u}|^2}{\mathcal J}\,dx\,dz
=\bar\nu_1\int_{\Omega_z}\xi|D_x\bm{u}|^2\,dx\,dz
\le\bar\nu_1\int_{\Omega_z}\xi|\nabla_x\bm{u}|^2\,dx\,dz,
\]
the last inequality because symmetrization $\nabla_x\bm{u}\mapsto
D_x\bm{u}=\tfrac12(\nabla_x\bm{u}+(\nabla_x\bm{u})^\top)$ is the orthogonal
projection onto symmetric matrices for the Frobenius norm, hence
norm-nonincreasing. Integrating in $t$, this is finite by
$\sqrt\xi\,\nabla_x\bm{u}\in L^2((0,T)\times\Omega_z)$ (\cite[Def.\
2.1]{WangDouJiu}): $\sqrt{\nu_1(\rho)}\,D_x\bm{u}\in
L^2((0,T)\times\Omega)$.

\emph{Vertical dissipation.} By $\partial_y\bm{u}=\mathcal
J\,\partial_z\bm{u}$ and $\nu_2(\rho)=\bar\nu_2\rho/\mathcal J^2$
(\eqref{eq:wdj-derivative-identities}),
\[
\int_\Omega\nu_2(\rho)|\partial_y\bm{u}|^2\,dx\,dy
=\int_{\Omega_z}\frac{\nu_2(\rho)\mathcal
J^2|\partial_z\bm{u}|^2}{\mathcal J}\,dx\,dz
=\bar\nu_2\int_{\Omega_z}\frac\rho{\mathcal
J}|\partial_z\bm{u}|^2\,dx\,dz
=\bar\nu_2\int_{\Omega_z}\xi|\partial_z\bm{u}|^2\,dx\,dz,
\]
again using $\rho/\mathcal J=\xi$. Integrating in $t$, this is finite
by $\sqrt\xi\,\partial_z\bm{u}\in L^2((0,T)\times\Omega_z)$
(\cite[Def.\ 2.1]{WangDouJiu}): $\sqrt{\nu_2(\rho)}\,\partial_y\bm{u}\in
L^2((0,T)\times\Omega)$.

None of the eight bounds required $\xi$ bounded away from $0$: each is
a one-sided estimate in terms of $\xi,\,\xi|\log\xi|,\,\xi|\bm{u}|^2,$
etc., all of which vanish, rather than blow up, on $\{\xi=0\}$ under
the stated conventions. This completes the verification of Definition
\ref{def:weak-solution}(ii).
\end{proof}

\paragraph{What is established here relative to \cite{WangDouJiu}.}
\cite{WangDouJiu} prove existence of a weak solution $(\xi,\bm{u},w)$ of
the transformed system \eqref{eq:wdj-transformed}, with energy and
Bresch--Desjardins entropy inequalities, through a five-parameter
approximation scheme, and their Theorem 2.2 states that this transfers
to a weak solution $(\rho,\bm{u},v)$ of the
original system \eqref{eq:CPE-damped}, verified there by matching norms
between the two formulations (their Sec.\ VI.D). Theorem
\ref{thm:wdj-transfer} makes this pull-back explicit at the
level of the distributional equations, term by term -- the
regularity-class matching is established independently by Proposition
\ref{prop:wdj-regularity-transfer}, via
$\partial_yv=\frac1{H_s\mathcal J}w+\partial_zw$
(\eqref{eq:wdj-v-identity}) -- and Corollary
\ref{cor:wdj-energy-transfer} carries the energy inequality across it in
the same way. Lemma \ref{lem:wdj-cross-term-sign} identifies, as a
byproduct available only once the transfer is made explicit, the cross
term mechanism (b) reduces to at this pressure law, with the sign
already established by Ersoy, Ngom and Sy and by Wang, Dou and Jiu
themselves. What is
established here is thus the explicit distributional pull-back,
together with the structural interpretation, via $\mathcal{N}\equiv0$
(Theorem \ref{thm:Lambda-N-classification}), of why this factorization
exists precisely at the isothermal law; the change of variables itself,
already used by \cite{ErsoyNgomSy,WangDouJiu}, and their existence
theorem are only transferred here, not re-derived, and neither the
transfer nor the sign-definiteness result is extended beyond the
isothermal law (see the discussion after Lemma
\ref{lem:wdj-cross-term-sign} below).

\begin{corollary}[Transfer of the energy inequality]
\label{cor:wdj-energy-transfer}
Under the hypotheses of Theorem \ref{thm:wdj-transfer}, the finite-energy
inequality of Definition \ref{def:finite-energy}, extended by the
damping term of Remark \ref{rem:wdj-damping-extension} and specialized
to $p(\rho)=c^2\rho$, $e_1(\rho)=c^2(\rho\log\rho-\rho+1)$ (Proposition
\ref{prop:isothermal-potential}), holds for $(\rho,\bm{u},v)$ on $\Omega$ if
and only if the corresponding energy inequality holds for $(\xi,\bm{u},w)$
on $\Omega_z$, of the same functional form as the base energy
inequality of \cite[Prop.\ 3.2, eq.\ (3.40)]{WangDouJiu} from which it
descends, through their approximation scheme, at the level of the weak
solution of their Theorem 2.1. More precisely, writing
$\widehat{\mathcal E}$ for the harmonized total energy
\eqref{eq:harmonized-energy} at $\bar\rho=1$ and
\begin{equation}
\label{eq:wdj-energy-transformed}
\mathcal E_\xi(t):=\int_{\Omega_z}\Bigl(\tfrac12\xi|\bm{u}|^2
+c^2(\xi\log\xi-\xi+1)\Bigr)\,dx\,dz
\end{equation}
for the corresponding energy on the transformed domain,
\begin{equation}
\label{eq:wdj-energy-constant}
\widehat{\mathcal E}(t)=\mathcal E_\xi(t)+C_{H,H_s}\quad\text{for every }t\ge0,
\qquad
C_{H,H_s}:=c^2\bigl(|\Omega|-|\Omega_z|\bigr)
=c^2|\mathbb T^2|\Bigl[H-H_s\bigl(1-e^{-H/H_s}\bigr)\Bigr]>0,
\end{equation}
while the kinetic, viscous-dissipation, and damping terms transfer with
no additive constant. Since $C_{H,H_s}$ is independent of $t$, it
cancels between the initial and current energies, and the two
integrated energy inequalities are equivalent.
\end{corollary}
\begin{proof}
We first re-derive \eqref{eq:wdj-energy-constant} directly from
$\rho=\xi\mathcal J$, without passing through $\mathfrak
M=c^2\log\xi$ -- which is undefined on $\{\xi=0\}$ and is therefore
unusable once $(\xi,\bm{u},w)$ is only a weak solution admitting
vacuum. Throughout, $r\log r:=0$ at $r=0$. Since $H_s=c^2/g$ and
$\mathcal J(y)=e^{-y/H_s}$ \eqref{eq:jacobian-def},
\begin{equation}
\label{eq:cJ-log-phi}
c^2\log\mathcal J(y)=-c^2\frac y{H_s}=-gy=-\Phi(y),
\end{equation}
an elementary identity that replaces $\mathfrak M=c^2\log\xi$ below.
Using $\rho=\xi\mathcal J$, the identity is taken directly on the
product $\rho\log\rho$ -- not on the pointwise split
$\log\rho=\log\xi+\log\mathcal J$, which is meaningless on $\{\xi=0\}$
since $\log\xi$ alone is undefined there -- and holds pointwise
everywhere on $(0,H)$, including on $\{\xi=0\}$ under the stated
convention, both sides vanishing together there:
\[
\rho\log\rho=\mathcal J\,\xi\log\xi+\mathcal J\,\xi\log\mathcal J
=\mathcal J\,\xi\log\xi-\frac1{c^2}\,\mathcal J\,\xi\,\Phi(y),
\]
the last equality by \eqref{eq:cJ-log-phi}. Multiplying by $c^2$ and combining with
$e_1(\rho)=c^2(\rho\log\rho-\rho+1)$ \eqref{eq:isothermal-potential}
and $\rho=\xi\mathcal J$,
\[
e_1(\rho)+\Phi(y)\rho
=c^2\mathcal J\,\xi\log\xi-\mathcal J\,\xi\,\Phi(y)-c^2\mathcal
J\,\xi+c^2+\Phi(y)\,\mathcal J\,\xi
=c^2\mathcal J\bigl(\xi\log\xi-\xi\bigr)+c^2,
\]
the two terms in $\Phi(y)\mathcal J\xi$ cancelling exactly: this is the
direct-substitution counterpart of $\partial_y\mathfrak M=0$, reached
here with no reference to $\mathfrak M$ itself, and valid pointwise for
every $y\in(0,H)$, including where $\xi(\zeta(y))=0$ (there both sides
reduce to $c^2$, since $\rho=0$ too). Writing
$c^2=c^2\mathcal J+c^2(1-\mathcal J)$ splits this as
\begin{equation}
\label{eq:vacuum-robust-internal-split}
e_1(\rho)+\Phi(y)\rho=c^2\mathcal J\bigl(\xi\log\xi-\xi+1\bigr)+c^2(1-\mathcal J),
\end{equation}
again pointwise on all of $(0,H)$, vacuum included. Integrating
\eqref{eq:vacuum-robust-internal-split} over $\Omega$ termwise, using
$dy=dz/\mathcal J$: in the first term the factor $\mathcal J$ is
exactly cancelled by the Jacobian,
\[
\int_\Omega c^2\mathcal J\bigl(\xi\log\xi-\xi+1\bigr)\,dx\,dy
=\int_{\Omega_z}c^2\bigl(\xi\log\xi-\xi+1\bigr)\,dx\,dz,
\]
while in the second, using $\int_\Omega\mathcal
J\,dx\,dy=\int_{\Omega_z}dx\,dz=|\Omega_z|$ (again $dy=dz/\mathcal J$),
\[
\int_\Omega c^2(1-\mathcal J)\,dx\,dy=c^2|\Omega|-c^2|\Omega_z|=:C_{H,H_s}.
\]
Summing the two,
\[
\int_\Omega\bigl(e_1(\rho)+\Phi(y)\rho\bigr)\,dx\,dy
=c^2\int_{\Omega_z}(\xi\log\xi-\xi+1)\,dx\,dz+C_{H,H_s},
\]
which is exactly \eqref{eq:wdj-energy-constant}. Unlike the identity
$\mathfrak M=c^2\log\xi$ it replaces, every step above involves
$\log\xi$ only inside the product $\xi\log\xi$, so this derivation of
\eqref{eq:wdj-energy-constant} -- and hence the corollary itself --
remains valid verbatim on any vacuum region $\{\xi=0\}$. That
$C_{H,H_s}$ so obtained coincides with $c^2(|\Omega|-|\Omega_z|)$, and
is strictly positive, is seen exactly as before: writing $|\Omega|=|\mathbb
T^2|H$ and $|\Omega_z|=|\mathbb T^2|h$ with $h=H_s(1-e^{-H/H_s})=\zeta(H)$
(the definitions preceding \eqref{eq:wdj-ens-transform}), positivity
holds because $\mathcal J(y')<1$ for every $y'>0$
(\eqref{eq:jacobian-def}), so
\[
h=\zeta(H)=\int_0^H\mathcal J(y')\,dy'<\int_0^H1\,dy'=H.
\]
The kinetic and
dissipation terms match exactly, with no additive constant, by
\eqref{eq:wdj-measure-identity} and \eqref{eq:wdj-derivative-identities}:
\[
\int_\Omega\tfrac12\rho|\bm{u}|^2dx\,dy=\int_{\Omega_z}\tfrac12\xi|\bm{u}|^2dx\,dz,
\qquad
2\int_\Omega\nu_1(\rho)|D_x\bm{u}|^2dx\,dy=2\bar\nu_1\int_{\Omega_z}\xi|D_x\bm{u}|^2dx\,dz,
\]
\[
\int_\Omega\nu_2(\rho)|\partial_y\bm{u}|^2dx\,dy=\bar\nu_2\int_{\Omega_z}\xi|\partial_z\bm{u}|^2dx\,dz,
\qquad
r\int_\Omega\rho|\bm{u}|^3dx\,dy=r\int_{\Omega_z}\xi|\bm{u}|^3dx\,dz.
\]
Summing the kinetic and internal-energy identities gives
\eqref{eq:wdj-energy-constant}. Summing all terms together and applying
Theorem \ref{thm:total-energy} (extended per Remark
\ref{rem:wdj-damping-extension}) on the $(\rho,\bm{u},v)$ side gives exactly
the kinetic, internal, dissipation, and damping terms of
\cite[Prop.\ 3.2, eq.\ (3.40)]{WangDouJiu} on the $(\xi,\bm{u},w)$ side --
that display carries, in addition, the regularization terms specific to
their approximate system at that level, which are absent here since we
work directly with the (already regularization-free) weak solution of
their Theorem 2.1 -- up to the additive constant $C_{H,H_s}$ identified
above, which cancels identically in the differenced, integrated
inequality form.
\end{proof}

\begin{proposition}[Isothermal image of mechanism (b)]
\label{prop:wdj-mechanism-reduction}
Apply Theorem \ref{thm:kappa-entropy} to $(\rho,\bm{u},v)$ with the viscosity
choice \eqref{eq:wdj-viscosity-choice} -- not excluded by its proof, by
Remark \ref{rem:wdj-viscosity-divergence-form} above -- specialized to
$\kappa=1$: the classical
Bresch--Desjardins entropy construction that \cite{WangDouJiu} itself
uses. At $\kappa=1$, mechanism (b) of \eqref{eq:kappa-entropy} splits as
$B=B_1+B_2$, with
\begin{equation}
\label{eq:B1-B2-def}
B_1:=2\int_\Omega\nu_1(\rho)\partial_y\bm{u}\cdot\nabla_xv\,dx\,dy,
\qquad
B_2:=4\bar\nu_1\int_\Omega\nu_1(\rho)\mathcal N(\rho)\nabla_x\rho\cdot\nabla_xv\,dx\,dy
\end{equation}
(the factors $\kappa$ multiplying $B_1$ and $B_2$ in \eqref{eq:kappa-entropy}
both equal $1$ here; at general $\kappa\in[0,1]$ they multiply $B_1$ and
$B_2$ respectively and are both linear in $\kappa$, the drift's own
equation \eqref{eq:V-equation} having converted the coefficient of $B_2$
from $\kappa^2$ to $\kappa$, as recorded in the proof of Theorem
\ref{thm:kappa-entropy}). At the
isothermal law, $\mathcal N\equiv0$ (Proposition
\ref{prop:isothermal-Lambda-N}) gives $B_2=0$, removing the
density-gradient component of mechanism (b), and
\begin{equation}
\label{eq:wdj-mechanism-b-reduced}
B_{\mathrm{iso}}:=B_1
=2\bar\nu_1\int_{\Omega_z}\xi\,\partial_z\bm{u}\cdot\nabla_xw\,dx\,dz
\end{equation}
is the whole of mechanism (b): a term coupling the vertical shear of $\bm{u}$
to the horizontal gradient of $w$, here expressed as its image under
the Ersoy--Ngom--Sy change of variables. No claim about the sign of
$B_{\mathrm{iso}}$ is made at this stage; it is established separately,
as a classical fact recalled from Ersoy, Ngom and Sy and from Wang,
Dou and Jiu, in Lemma \ref{lem:wdj-cross-term-sign} below.
\end{proposition}
\begin{proof}
At $\kappa=1$, what would formally have been mechanism (c) is absent from
\eqref{eq:kappa-entropy} by the cancellation recorded in the proof of
Theorem \ref{thm:kappa-entropy} (Remark \ref{rem:mechanism-b-open}),
mechanism (a) is proportional to
$\mathcal{N}(\rho)$, and mechanism (b) equals $B_1+B_2$ as in
\eqref{eq:B1-B2-def}. By Proposition \ref{prop:isothermal-Lambda-N},
$\mathcal{N}\equiv0$ at the isothermal law for \emph{every} choice of
$\nu_2$, including the $y$-dependent \eqref{eq:wdj-viscosity-choice}:
mechanism (a) vanishes, $B_2=0$, and mechanism (b) reduces to $B_1$. By
\eqref{eq:wdj-derivative-identities} and \eqref{eq:wdj-v-identity},
$\nabla_xv=\nabla_xw/\mathcal J$, since $\nabla_x$ commutes with the
purely-$y$ reparametrization, and
\[
\nu_1(\rho)\partial_y\bm{u}=\bar\nu_1\xi\mathcal J\cdot\mathcal J\,\partial_z\bm{u}=\bar\nu_1\xi\mathcal J^2\partial_z\bm{u},
\]
so $\nu_1(\rho)\partial_y\bm{u}\cdot\nabla_xv=\bar\nu_1\xi\mathcal J\,\partial_z\bm{u}\cdot\nabla_xw$
carries exactly one remaining power of $\mathcal J$, cancelled by the
Jacobian $dy=dz/\mathcal J$ of \eqref{eq:wdj-measure-identity} upon
integration, giving \eqref{eq:wdj-mechanism-b-reduced}.
\end{proof}

\begin{remark}[Isothermal resolution of mechanism (b) and relation with Ersoy--Ngom--Sy--Wang--Dou--Jiu]
\label{rem:wdj-mechanism}
The value $\kappa=1$ in Proposition \ref{prop:wdj-mechanism-reduction}
is the formal, non-attained endpoint of the free interval
$\kappa\in(0,1)$ used by the $\kappa$-entropy method of Bresch, Vasseur
and Yu \cite{BVY} (Section \ref{sec:conclusion}), which multiplies the
same drift $2\nabla_xs(\rho)$ of \eqref{eq:augmented-velocity-def} for a
general, possibly non-linear, $\nu_1$.

This answers Remark \ref{rem:isothermal-dichotomy} by two independent,
but related, structural facts:
\[
\boxed{\begin{gathered}
\mathcal N\equiv0\text{ removes the density-gradient component }B_2\text{ of mechanism (b);}\\
\partial_z\xi=0,\text{ together with the transformed continuity}\\
\text{equation and impermeability, converts the surviving component }B_1\\
\text{into the negative vertical dissipation }-D_w\text{ (defined in Remark \ref{rem:wdj-cross-term-favorable} below).}
\end{gathered}}
\]
First, $\mathcal N\equiv0$ (Proposition \ref{prop:isothermal-Lambda-N})
removes mechanism (b)'s second term for \emph{any} $\nu_2$, leaving
only $B_{\mathrm{iso}}=2\int_\Omega\nu_1(\rho)\partial_y\bm{u}\cdot
\nabla_xv\,dx\,dy$ (Proposition \ref{prop:wdj-mechanism-reduction}).
Second, the change of variables \eqref{eq:wdj-ens-transform} -- exactly
the multiplicative factorization of $\mathfrak M$ available at this one
law (Proposition \ref{prop:ens-montgomery}) -- produces $\partial_z\xi=0$,
and it is this second fact that lets the transformed continuity
equation convert $B_{\mathrm{iso}}$ into the classical vertical
Bresch--Desjardins dissipation $2\bar\nu_1\int_{\Omega_z}\xi|\partial_zw|^2$
already present in Ersoy, Ngom and Sy's and Wang, Dou and Jiu's own
entropy computations -- Lemma \ref{lem:wdj-cross-term-sign} below
records the precise correspondence, and is not itself a new
computation. We do not re-derive the full Bresch--Desjardins entropy
estimate of \cite{WangDouJiu} here (their Secs.\ III--VI, culminating in
their (5.9)): their analysis already controls the vertical dissipation
denoted $D_w$ below; our purpose is only to identify its negative as
the transformed image of mechanism (b) in the present general
barotropic framework (Remark \ref{rem:wdj-cross-term-favorable}).

The two features of the isothermal law -- $\mathcal{N}\equiv0$
and the fixed change of variables \eqref{eq:wdj-ens-transform} -- are
therefore not independent: both trace to the single fact
(Proposition \ref{lem:Lambda-N-geometric}, Theorem
\ref{thm:Lambda-N-classification}) that $\rho(t,x,\cdot)$ is a pure
exponential in $y$ iff $\mathcal{N}\equiv0$, iff the law is isothermal.
\end{remark}

\begin{lemma}[The classical Ersoy--Ngom--Sy--Wang--Dou--Jiu vertical cross-term identity]
\label{lem:wdj-cross-term-sign}
The identity below is not new. It is an immediate reformulation of the
vertical integration-by-parts mechanism already underlying the
dissipative term $2\bar\nu_1\int_{\Omega_z}\xi|\partial_zw|^2$ in the
Bresch--Desjardins entropy computations of Ersoy, Ngom and Sy
\cite[eqs.\ (29)--(35), (40)]{ErsoyNgomSy} and of Wang, Dou and Jiu
\cite[eq.\ (2.4) and the display following eq.\ (4.4)]{WangDouJiu},
where it appears for the augmented combination
$\bm{u}+2\bar\nu_1\nabla_x\ln\xi$ rather than for the bare cross term
below; in particular, the auxiliary identity
$\partial_z\operatorname{div}_x(\xi\bm{u})=-\xi\,\partial_{zz}w$ used in
the proof is their eq.\ (2.4) verbatim. We record the bare identity
separately only in order to combine it directly with Proposition
\ref{prop:wdj-mechanism-reduction} and identify the fate of mechanism
(b) under the isothermal transformation (Remark
\ref{rem:wdj-cross-term-favorable} below); no independent computation
is claimed.

Let $(\xi,\bm{u},w)$ be a regular solution of \eqref{eq:wdj-transformed} on
$\Omega_z$, in the same sense as a regular solution of \eqref{eq:CPE}
(strictly positive, in particular $\xi>0$, and sufficiently smooth) --
in particular with $\xi \bm{u}$ and $w$ twice continuously
differentiable in $z$, as used in the proof below to differentiate the
mass equation a second time. Then
\begin{equation}
\label{eq:wdj-cross-term-identity}
\int_{\Omega_z}\xi\,\partial_z\bm{u}\cdot\nabla_xw\,dx\,dz
=-\int_{\Omega_z}\xi\,|\partial_zw|^2\,dx\,dz.
\end{equation}
\end{lemma}

\begin{proof}[Proof (for completeness)]
We recall the classical computation. The four steps below reproduce
\cite[eqs.\ (29)--(35)]{ErsoyNgomSy} and \cite[eq.\ (2.4) and the
display following eq.\ (4.4)]{WangDouJiu}, specialized to the bare
cross term.
Since $\partial_z\xi=0$ (third equation of \eqref{eq:wdj-transformed}),
$\xi\,\partial_z\bm{u}=\partial_z(\xi \bm{u})$, so
\[
\int_{\Omega_z}\xi\,\partial_z\bm{u}\cdot\nabla_xw\,dx\,dz
=\int_{\Omega_z}\partial_z(\xi \bm{u})\cdot\nabla_xw\,dx\,dz.
\]
Integrating by parts in $x$ (periodicity, no boundary term) moves the
horizontal derivative onto $\xi \bm{u}$, and commuting $\partial_z$ and
$\operatorname{div}_x$,
\[
=-\int_{\Omega_z}w\,\operatorname{div}_x\bigl(\partial_z(\xi \bm{u})\bigr)\,dx\,dz
=-\int_{\Omega_z}w\,\partial_z\bigl(\operatorname{div}_x(\xi \bm{u})\bigr)\,dx\,dz.
\]
Differentiating the mass equation of \eqref{eq:wdj-transformed} in $z$
and using $\partial_z\xi=0$ once more (so
$\partial_{zz}(\xi w)=\xi\,\partial_{zz}w$) -- this is
\cite[eq.\ (2.4)]{WangDouJiu}, obtained by the same differentiation in
\cite[eqs.\ (29)--(35)]{ErsoyNgomSy} --
\[
\partial_z\bigl(\operatorname{div}_x(\xi \bm{u})\bigr)=-\partial_{zz}(\xi w)=-\xi\,\partial_{zz}w,
\]
hence
\[
\int_{\Omega_z}\xi\,\partial_z\bm{u}\cdot\nabla_xw\,dx\,dz
=\int_{\Omega_z}\xi\,w\,\partial_{zz}w\,dx\,dz.
\]
Integrating by parts in $z$, the boundary term vanishes by
$w|_{z=0,h}=0$, and $\partial_z(\xi w)=\xi\,\partial_zw$ once more gives
\eqref{eq:wdj-cross-term-identity} -- the same final step as in
\cite[eq.\ (35)]{ErsoyNgomSy} and in the display following
\cite[eq.\ (4.4)]{WangDouJiu}.
\end{proof}

\begin{remark}[Mechanism (b) and the classical Ersoy--Ngom--Sy--Wang--Dou--Jiu vertical dissipation]
\label{rem:wdj-cross-term-favorable}
The classical vertical dissipation
\[
D_w:=2\bar\nu_1\int_{\Omega_z}\xi\,|\partial_zw|^2\,dx\,dz\ge0
\]
already appears in the Bresch--Desjardins entropy estimates of Ersoy,
Ngom and Sy \cite[eq.\ (40)]{ErsoyNgomSy} and of Wang, Dou and Jiu
\cite[eq.\ (4.3), Lemma 4.1]{WangDouJiu} (item (i) below). The
contribution of the present framework is the identification
\[
\boxed{B_{\mathrm{iso}}=-D_w}
\]
of this known dissipation, with sign reversed, with the complete
mechanism (b) arising from the general barotropic augmented identity
of Theorem \ref{thm:kappa-entropy}, at the isothermal law and
$\kappa=1$. The chain is
\[
\mathcal N\equiv0\implies B_2=0,
\qquad
\rho=\xi\mathcal J,\ \partial_z\xi=0\implies
B_{\mathrm{iso}}=B_1=2\bar\nu_1\int_{\Omega_z}\xi\,\partial_z\bm{u}\cdot\nabla_xw\,dx\,dz
\quad(\text{Proposition \ref{prop:wdj-mechanism-reduction}}),
\]
\[
\text{Lemma \ref{lem:wdj-cross-term-sign}}\implies
B_{\mathrm{iso}}=-D_w\le0,
\]
so that, spelled out in the original variables via
\eqref{eq:wdj-derivative-identities}--\eqref{eq:wdj-v-identity} and
\eqref{eq:wdj-measure-identity} as in the proof of Proposition
\ref{prop:wdj-mechanism-reduction},
\[
2\int_\Omega\nu_1(\rho)\partial_y\bm{u}\cdot\nabla_xv\,dx\,dy
=B_{\mathrm{iso}}=-D_w\le0.
\]
Since mechanism (a) and $B_2$ both vanish identically at the
isothermal law, this identifies mechanism (b) \emph{in full}, at
$\kappa=1$ and the isothermal law, with $-D_w$.

Three distinct claims should not be conflated.
\begin{enumerate}[label=(\roman*)]
\item \emph{The dissipation fact}: $D_w\ge0$ for regular solutions of
\eqref{eq:wdj-transformed}. Established by Ersoy, Ngom and Sy and by
Wang, Dou and Jiu; not new here.
\item \emph{The identification}: $B_{\mathrm{iso}}=-D_w$. Established here
(Proposition \ref{prop:wdj-mechanism-reduction}); absent from both
sources, which have no general-law mechanism (b) object to identify
anything with.
\item \emph{The resolution of the obstruction}: the unresolved issue
within the present analysis is whether the mechanism-(b) obstruction of
Remark \ref{prop:mechanism-b-obstruction} can be resolved for a general
barotropic law, in the original variables, without the isothermal
law's special structure. This question is not resolved by the present
analysis and should not be confused with the status of $D_w$ itself at
the level of \cite{WangDouJiu}'s weak solutions, which their own
entropy already controls (see below).
\end{enumerate}

As recorded in Remark \ref{rem:wdj-mechanism}, neither structural fact
survives for a general barotropic law: $B_2\not\equiv0$ once $\mathcal
N\not\equiv0$, and no multiplicative factorization of this form renders
the transformed density vertically constant. Within the reduction
considered here, the isothermal law is precisely the case where both
hold simultaneously; this does not exclude that another change of
variables, or another entropy functional, might achieve an analogous
conversion at a different pressure law.

Two further things are not claimed: a full re-derivation of the
Bresch--Desjardins entropy inequality of \cite[Sec.\ V, eq.\
(5.9)]{WangDouJiu} (Remark \ref{rem:wdj-damping-extension}) -- only item
(ii) above is not read off directly from either source -- and any new
existence statement, already noted above.

Finally, $D_w$ is already rigorously controlled in the weak theory of
Wang--Dou--Jiu, whose regularity class gives $\sqrt\xi\,\partial_zw\in
L^2((0,T)\times\Omega_z)$ (\cite[Def.\ 2.1]{WangDouJiu}) as a
consequence of their own entropy estimate. What is proved here only at
the regular level is the bare identity
\eqref{eq:wdj-cross-term-identity} itself, via a second
$z$-differentiation of the transformed mass equation (Lemma
\ref{lem:wdj-cross-term-sign}), beyond their weak class; establishing it
at the level of their weak solution would require an additional
approximation argument, unnecessary for the dissipative estimate they
already obtain.
\end{remark}

As a complementary strong-solution result, built on the same change of
variables: Theorem \ref{thm:wdj-transfer} and Corollary \ref{cor:wdj-energy-transfer}
transfer \emph{weak} solutions, possibly attaining vacuum, through the
change of variables \eqref{eq:wdj-ens-transform}. Away from vacuum, the
same change of variables underlies the genuinely different, and
stronger, strong-solution theory of Ba, Ersoy and Ngom
\cite{BaErsoyNgom2026} recalled in \S\ref{sec:related-work}: their
construction requires the density to stay uniformly positive and
bounded on the time interval of existence (their estimate
\cite[eq.\ (3.5)]{BaErsoyNgom2026}), the opposite regime from the
possibly vanishing density admitted by Definition
\ref{def:weak-solution} and by \cite{WangDouJiu}; the two results are
complementary rather than overlapping, and neither implies the other.

\section{Conclusion}
\label{sec:conclusion}

\paragraph{Results established.} The hydrostatic balance
\eqref{eq:intro-hydrostatic} alone forces a first hydrostatic integral --
the vertical constancy of $\mathfrak{M}=\mathfrak{h}(\rho)+\Phi(y)$, the
classical Montgomery potential (Proposition \ref{lem:first-integral}) --
and this single fact organizes the energetic structure of the
barotropic CPE at every order of differentiation examined here. Tested
against the continuity equation, it produces, with no separate estimate
on the vertical velocity, the total mechanical energy identity of
Theorem \ref{thm:total-energy}. Differentiated once more, through the
augmented-velocity identity of Theorem \ref{thm:kappa-entropy}, it
produces the structural functions $\Lambda,\mathcal{N}$ of Definition
\ref{def:Lambda-N}, which give a complete classification, within the
stated pressure class, of the laws for which either structural
coefficient vanishes: $\Lambda\equiv0$ if and only if $p$ is quadratic, $\mathcal{N}\equiv0$
if and only if $p$ is isothermal, and no law annihilates both (Theorem
\ref{thm:Lambda-N-classification}). Re-expressed multiplicatively at
the one law for which $\mathcal{N}\equiv0$, the same constancy \emph{is} the
Ersoy--Ngom--Sy change of variables (Proposition
\ref{prop:ens-montgomery}); under it, we give a rigorous, term-by-term
transfer of the weak solutions of Wang, Dou and Jiu \cite{WangDouJiu} to
the original hydrostatic variables, valid on any vacuum region and
making explicit the argument behind their Theorem 2.2 (Theorem
\ref{thm:wdj-transfer}, Propositions
\ref{prop:wdj-hydrostatic-transfer} and \ref{prop:wdj-regularity-transfer},
Corollary \ref{cor:wdj-energy-transfer}).

\paragraph{Limitation of the direct route.} The augmented-velocity
identity of Theorem \ref{thm:kappa-entropy} is not, by itself, a closed
Bresch--Desjardins-type entropy: of the two mechanisms surviving the
Hessian cancellation recorded in its proof, mechanism (a) is
algebraically absorbable, for the polytropic law, by a suitable choice
of $\nu_2$ (Proposition \ref{prop:mechanism-a-closure}), and vanishes
identically at the isothermal law, but mechanism (b) -- the
vertical-shear term $B_1$ together with the density-gradient term
$B_2$, both coupled to $\nabla_xv$ -- is not controlled, at any
barotropic law, by the direct route developed here (Remark
\ref{prop:mechanism-b-obstruction}): the hydrostatic reduction replaces
the vertical momentum \emph{evolution} equation by the algebraic
balance \eqref{eq:intro-hydrostatic}, leaving $v$ without an equation to
test against (Remark \ref{rem:gravity-role}). This is a structural
observation about the one route examined here, not an impossibility
theorem. One partial exception is visible at the isothermal law and
$\kappa=1$, where mechanism (b) reduces exactly to the classical
vertical dissipation of Ersoy, Ngom and Sy and of Wang, Dou and Jiu,
with sign reversed: $B_{\mathrm{iso}}=-D_w\le0$ (Lemma
\ref{lem:wdj-cross-term-sign}, Remark
\ref{rem:wdj-cross-term-favorable}). This does not resolve the
corresponding limitation for a general barotropic law in the original
variables, which remains unresolved within the present analysis.

\paragraph{Perspectives.} Two directions seem most promising, and we
develop neither here.

\emph{Direction A: toward a closed weak theory with gravity.} Three
tools bear on the residual mechanism isolated by the present analysis, mechanism (b): a
genuinely different entropy functional adapted to the hydrostatic
reduction rather than transposed from the non-hydrostatic theory; an
effective hydrostatic flux for the vertical variable, in the spirit of
the effective viscous flux of the compressible Navier--Stokes theory;
and the $\kappa$-entropy method of Bresch, Vasseur and Yu \cite{BVY},
whose two-velocity structure is naturally connected with the drift
equation \eqref{eq:V-equation} used here. Whether its broader
viscosity class can provide additional control of mechanism (b) for
the polytropic law remains to be investigated. Whether any of the
three in fact closes
mechanism (b) is a substantial question this note does not attempt to
answer, nor do we claim mechanism (b) is the only difficulty a full
existence proof would still encounter. The
energetic structure identified here is, in our view, a natural starting
point for any of the three.

\emph{Direction B: beyond the barotropic model.} It would also be
natural to investigate whether an analogous hydrostatic potential
survives the extension to non-isentropic pressure laws and whether
corresponding structural functions can be identified in multilayer or
vertically discretized hydrostatic models.

\section*{Acknowledgements}
Part of this work was carried out while the first author was participating in the scientific trimester
``Mathematical Developments in Geophysical Fluid Dynamics''
held at the Institut Henri Poincaré (April--July 2026).
The authors acknowledge support of the Institut Henri Poincaré
(UAR 839 CNRS--Sorbonne Université), and LabEx CARMIN
(ANR-10-LABX-59-01).

\bibliographystyle{plain}
\bibliography{references}

\end{document}